\documentclass[final]{siamltex}
\usepackage{}
\usepackage{graphicx}
\usepackage{multirow,bigstrut}
\usepackage{amsmath,amsfonts,amssymb}
\usepackage{mathrsfs}
\usepackage{color}
\usepackage{upgreek}
\usepackage{bm}
\usepackage{verbatim}

\usepackage{tikz}
\usetikzlibrary{shapes,arrows}

\newtheorem{rem}{\hspace{1mm}Remark}[section]
\newtheorem{exam}{\hspace{1mm}Example}[section]

\newcommand{\beq}{\begin{equation}}
\newcommand{\eeq}{\end{equation}}
\newcommand{\bey}{\begin{eqnarray}}
\newcommand{\eey}{\end{eqnarray}}

\title{Multiscale Approach and The Convergence for the time-dependent Maxwell-Schr\"{o}dinger System in
Heterogeneous Nanostructures
\thanks{{This work is supported by National Natural Science Foundation
of China (grant 11571353, 91330202),  and
Project supported by the Funds for Creative Research Group of China
(grant 11321061).}}
}
\author{Liqun Cao\thanks{{Corresponding author. LSEC, NCMIS, Institute of
      Computational Mathematics and Scientific/Engineering Computing,
      Academy of Mathematics and Systems Science, Chinese Academy of
      Sciences, Beijing 100190, China; ({\tt clq@lsec.cc.ac.cn}).}}
\and Chupeng Ma\thanks{{Institute of
      Computational Mathematics and Scientific/Engineering Computing,
      Academy of Mathematics and Systems Science, Chinese Academy of
      Sciences, Beijing 100190, China; ({\tt machupeng@lsec.cc.ac.cn}).}}
  \and Jianlan Luo\thanks{{ School of Physics and Nuclear Energy Engineering, Beijing University
of Aeronautics and Astronautics, Beijing 100191, China; ({\tt
luojianlan@126.com}). }}
\and Lei Zhang\thanks{{Department of Logistics Management, Logistics Academy, Beijing
100858, China; ({\tt zhanglei@lsec.cc.ac.cn}).}}}

\begin{document}

\maketitle

\begin{abstract}
This paper discusses the multiscale approach and the convergence of
the time-dependent Maxwell-Schr\"{o}dinger system with rapidly oscillating
discontinuous coefficients arising from the modeling of a heterogeneous
nanostructure with a periodic microstructure. The homogenization method and the multiscale asymptotic
method for the nonlinear coupled equations are presented. The efficient
numerical algorithms based on the above methods are proposed. Numerical simulations are then carried out to validate
the method presented in this paper.
\end{abstract}



\section{Introduction}\label{sec-1}

It is well-known that the classical Maxwell's equations are
widely used in the macroscopic electromagnetic theory. However,
when the size of physical devices reaches the wavelength of electron,
quantum effects become important even dominant and can not be neglected.
To analyze and model such physical devices, coupled numerical simulations of Maxwell and Schr\"{o}dinger equations need to be performed \cite{Pi}.
For example, an array of quantum dots is irradiated by the nearly  electromagnetic field.
The quantum dots, influenced by the incoming electromagnetic field, form in general a superposition of the ground and excited states,
which create charge oscillations and refine the initial electromagnetic field as new source of radiation.
This process leads to the interaction between the quantum dots and the electromagnetic field \cite{Oh}.
Hence the Maxwell equations with the quantum current density can be written as
\begin{equation}\label{eq:1-1}
\begin{array}{@{}l@{}}
{\displaystyle  \nabla \times \mathbf{E}+\frac{\partial \mathbf{B}}{\partial t}=0}\\[2mm]
{\displaystyle \nabla \times \mathbf{H}-\frac{\partial \mathbf{D}}{\partial
t}=\mathbf{J}_s+\mathbf{J}_q}\\[2mm]
{\displaystyle \nabla \cdot \mathbf{D}=\rho}\\[2mm]
{\displaystyle\nabla \cdot \mathbf{B}=0,}
\end{array}
\end{equation}
where
$\mathbf{E}(\mathbf{x},t)$, $\mathbf{B}(\mathbf{x},t)$, $\mathbf{H}(\mathbf{x},t)$,
$\mathbf{D}(\mathbf{x},t)$,
 $\rho(\mathbf{x},t)$, $\mathbf{J}_s(\mathbf{x},t)$, $\mathbf{J}_q(\mathbf{x},t)$ denote
the electric field intensity, the magnetic flux density, the magnetic field intensity, the electric
displacement, the electric charge density,
the source current density and the quantum current density, respectively. They are functions of the space
$\mathbf{x}$ and time $ t $. The quantum current density is derived through the method of quantum mechanics.
In a linear medium, we have
\begin{equation}\label{eq:1-2}
{\displaystyle\mathbf{D}=\eta\mathbf{E},\quad \mathbf{B}=\mu\mathbf{H},}
\end{equation}
where $\eta=(\eta_{ij})$, $\mu=(\mu_{ij}) $ are the electric permittivity and magnetic permeability,
which are $ 3\times 3 $ positive-definite matrix-valued
functions of the position, respectively.


In the study of the interaction of an electron with the incoming electromagnetic field, the time-dependent
Schr\"{o}dinger equation can be written as follows:
\begin{equation}\label{eq:1-3}
\begin{array}{@{}l@{}}
{\displaystyle i\hbar\frac{\partial \Psi(\mathbf{x},t)}{\partial t}=\hat{H}\Psi(\mathbf{x},t),
\quad \hat{H}=\hat{H}_0+\hat{V}(\mathbf{x},t)},
\end{array}
\end{equation}
where  $\hat{H}_0 $ is the effective Hamiltonian operator in a crystal structure with the effective mass approximation (EMA) and
$\hat{V}(\mathbf{x},t)$ is an interaction Hamiltonian with the incoming electromagnetic field. In this paper,
we take the interaction Hamiltonian to be of the form
\begin{equation}\label{eq:1-4}
{\displaystyle \hat{V}(\mathbf{x},t)=-\mathbf{E}(\mathbf{x},t)\cdot\widehat{\bm{\zeta}},}
\end{equation}
where $\widehat{\bm{\zeta}}=-e\mathbf{x}$ is the electric dipole moment operator,  $-e $ is the charge of the electron
and $\mathbf{E}(\mathbf{x},t)\cdot\widehat{\bm{\zeta}} $ denotes the scalar product of $ \mathbf{E}(\mathbf{x},t) $
and $ \widehat{\bm{\zeta}} $. The quantum current density that can be injected into the time-dependent Maxwell's equations is given as
\begin{equation}\label{eq:1-5}
{\displaystyle \mathbf{J}_q(\mathbf{x},t)=\frac{-e\hbar}{2im}\big(\overline{\Psi}\nabla\Psi-
\Psi\nabla\overline{\Psi}\big)},
\end{equation}
where $m$ is the effective mass and $i$ is the imaginary unit, i.e. $i^2 = -1$.
$ \overline{\Psi} $ denotes the complex conjugate of $ \Psi $.
As usual, here we employ the atomic units, i.e. $\hbar=e=1 $.

In this paper, we consider the following time-dependent
Maxwell-Schr\"{o}dinger system with rapidly oscillating discontinuous coefficients given by
\begin{equation}\label{eq:1-6}
\left\{
\begin{array}{@{}l@{}}
{\displaystyle i\frac{\partial \Psi^{\varepsilon}(\mathbf{x},t)}{\partial t}=
-\nabla \cdot \big(A(\frac{\mathbf{x}}{\varepsilon})\nabla
\Psi^\varepsilon(\mathbf{x},t)\big)
+\big(V_c(\frac{\mathbf{x}}{\varepsilon})-\textbf{E}^{\varepsilon}\cdot\widehat{\bm{\zeta}}
+V_{xc}[\rho^\varepsilon]\big)
\Psi^\varepsilon,}\\[2.3mm]
{\displaystyle \qquad \qquad \qquad \qquad (\mathbf{x},t)\in
\Omega\times(0,T),}\\[2.3mm]
{\displaystyle \eta(\frac{\mathbf{x}}{\varepsilon})\frac{\partial \mathbf{E}^\varepsilon(\mathbf{x},t)}{\partial t}
=\mathbf{curl}\/ \mathbf{H}^\varepsilon(\mathbf{x},t)+\mathbf{f}(\mathbf{x},t)
-\mathbf{J}_q^\varepsilon,\,\, \nabla\cdot \mathbf{f}=0,\,\,(\mathbf{x},t)\in
\Omega\times(0,T),}\\[2.3mm]
{\displaystyle \mu(\frac{\mathbf{x}}{\varepsilon})
\frac{\partial \mathbf{H}^\varepsilon(\mathbf{x},t)}{\partial t}=-\mathbf{curl}\/ \mathbf{E}^\varepsilon(\mathbf{x},t)
,\,\, (\mathbf{x},t)\in
\Omega\times(0,T),}\\[2.3mm]
{\displaystyle \nabla\cdot \big(\eta(\frac{\mathbf{x}}{\varepsilon})\mathbf{E}^\varepsilon(\mathbf{x},t)\big)=\rho^\varepsilon(\frac{\mathbf{x}}{\varepsilon},t),
\quad\nabla\cdot \big(\mu(\frac{\mathbf{x}}{\varepsilon})\mathbf{H}^\varepsilon(\mathbf{x},t)\big)=0,
\,\, (\mathbf{x},t)\in
\Omega\times(0,T),}\\[2.3mm]
{\displaystyle \hat{\bm{\zeta}}=-\mathbf{x},\,\, \rho^{\varepsilon}=N{\vert\Psi^{\varepsilon}\vert}^2
,\,\,\mathbf{J}_q^{\varepsilon}=iN\big[\overline{\Psi}^{\varepsilon} A(\frac{\mathbf{x}}{\varepsilon})\nabla\Psi^{\varepsilon}-
\Psi^{\varepsilon}A(\frac{\mathbf{x}}{\varepsilon})\nabla\overline{\Psi}^{\varepsilon}\big].}
\end{array}
\right.
\end{equation}

Let $ \partial \Omega $ be the boundary of $ \Omega $ and  $\mathbf{n}=(n_1,n_2,n_3) $
be the outward unit normal to $ \partial \Omega $. We take (\ref{eq:1-6}) to hold
in $\Omega $ subject to the boundary conditions
\begin{equation}\label{eq:1-7}
{\displaystyle \Psi^\varepsilon(\mathbf{x},t)=0,\quad \mathbf{E}^{\varepsilon}(\mathbf{x},t)\times\mathbf{n}=0, \quad (\mathbf{x},t)\in
\partial \Omega\times(0,T).}
\end{equation}

For initial conditions we take
\begin{equation}\label{eq:1-8}
{\displaystyle \Psi^\varepsilon(\mathbf{x},0)=\Psi_0(\mathbf{x}),\quad
\mathbf{E}^\varepsilon(\mathbf{x},0)=\mathbf{\eta}^{-1}(\frac{\mathbf{x}}{\varepsilon}){\bm\varphi}(\mathbf{x}),\quad
\mathbf{H}^\varepsilon(\mathbf{x},0)=\mathbf{\mu}^{-1}(\frac{\mathbf{x}}{\varepsilon}){\bm\psi}(\mathbf{x}),}
\end{equation}
and
\begin{equation}\label{eq:1-9}
{\displaystyle \nabla\cdot {\bm\varphi}=\nabla\cdot {\bm\rho^\varepsilon}|_{t=0},\quad
\nabla\cdot {\bm\psi}=0,}
\end{equation}
where $ \nabla\cdot $ and $ \nabla $ are the divergence
operator and the gradient operator, respectively.
$\widehat{\bm{\zeta}}$, $\rho^{\varepsilon}$ and $\mathbf{J}_q^{\varepsilon} $ are
the electric dipole moment, the electron density and the quantum current density, respectively.
Here $N$ denotes the number density of electrons.

Suppose that $\Omega \subset \mathbb{R}^3 $ is a
bounded smooth domain or a bounded polyhedral convex domain
with a periodic microstructure as shown in Fig.~\ref{fig1-1}.
$ \varepsilon>0 $ is a small parameter which it
represents the relative size of a periodic cell for the
heterogeneous nanostructures, i.e. $
0<\varepsilon=l_p/L<1 $, where $
l_p$ and $ L $ are respectively the sizes of a periodic cell and a
whole domain $ \Omega $. $A(\frac{\displaystyle
\mathbf{x}}{\displaystyle \varepsilon})=(a_{ij}(\frac{\displaystyle
\mathbf{x}}{\displaystyle \varepsilon}))_{3\times 3} $ is the inverse of the
effective masses tensor of materials. $ V_c(\frac{\displaystyle
\mathbf{x}}{\displaystyle \varepsilon}) $ is the constraint step
potential function. $ V_{xc}(\rho^\varepsilon)
=V_{xc}(\mathbf{x},\frac{\displaystyle \mathbf{x}}{\displaystyle \varepsilon}) $ is the
exchange-correlation potential function. The boundary condition $
\Psi^\varepsilon(\mathbf{x},t)=0 $ on $\partial \Omega $ implies that the
wave function satisfies the tight-binding state condition.
The perfect conductive boundary condition for $  \mathbf{E}^{\varepsilon}(\mathbf{x},t)\times\mathbf{n}=0 $
is imposed on $\partial \Omega$. Here and in the sequel, the Einstein summation
convention is used: summation is taken over repeated indices.

\begin{figure}
\begin{center}
\tiny{(a)}\includegraphics[width=6.1cm,height=5cm]{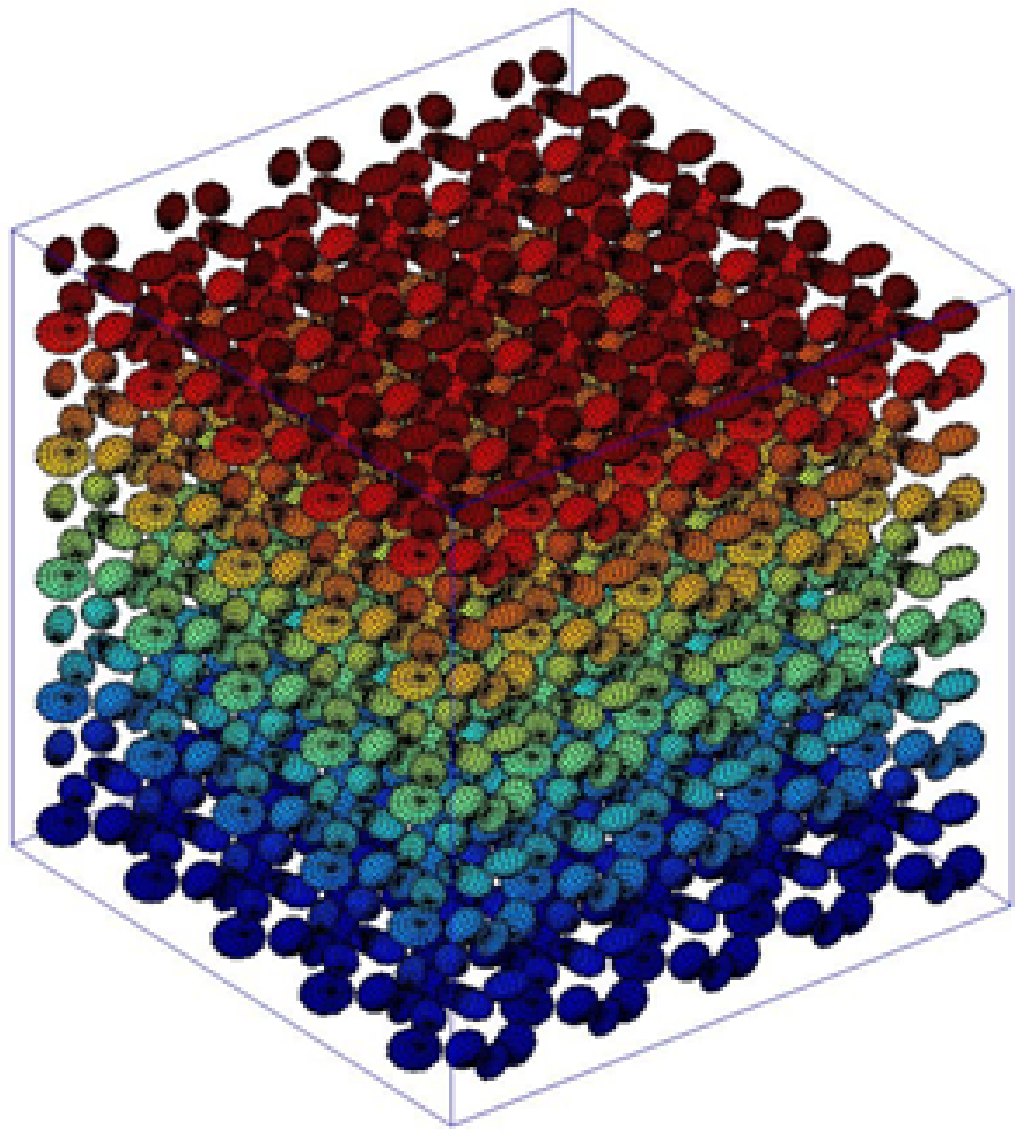}
\tiny{(b)}\includegraphics[width=6.1cm,height=5cm]{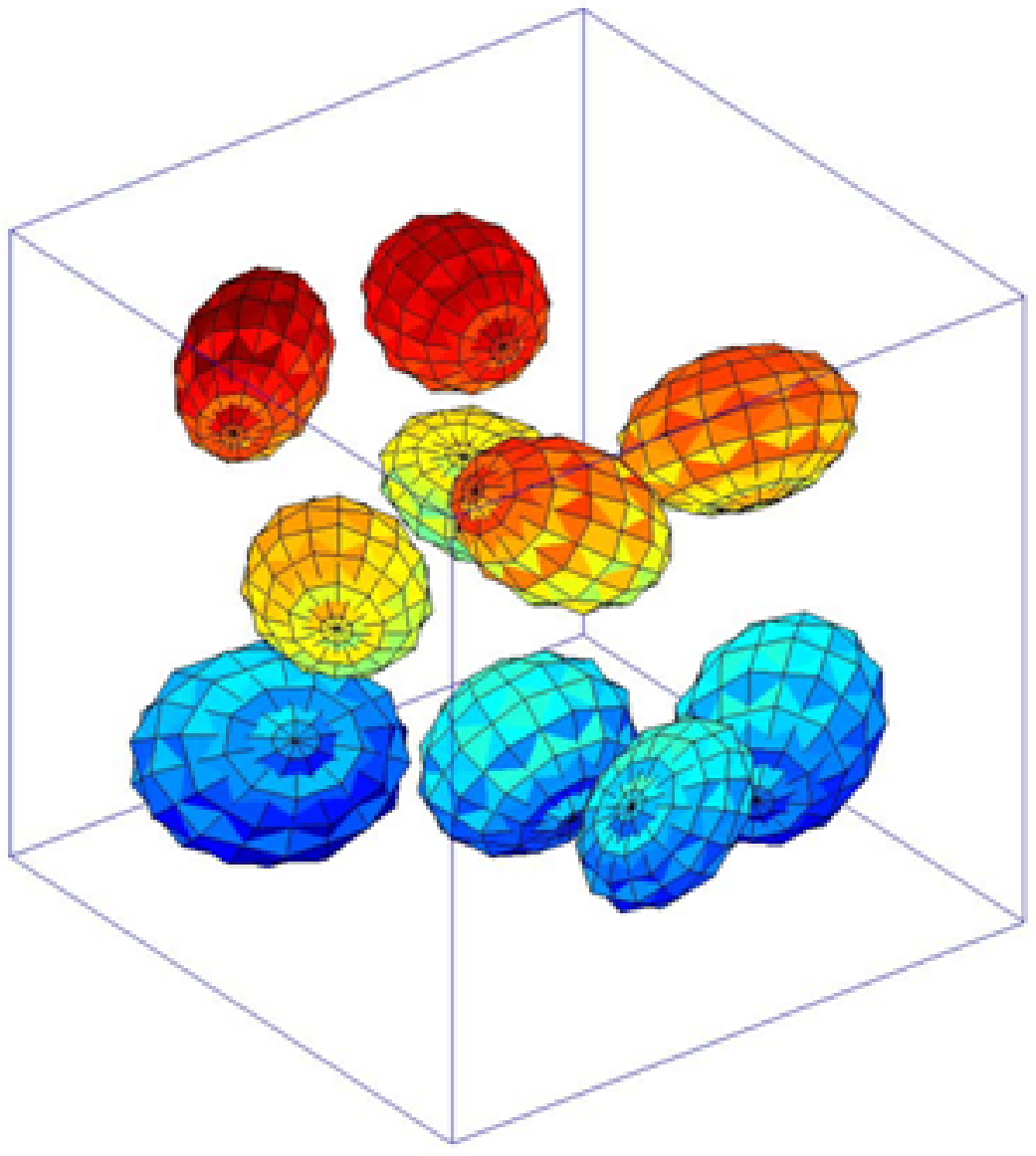}
\caption{(a) A semiconductor device $ \Omega $  with a great of number of quantum dots; (b)
the periodic cell
$Q$.}\label{fig1-1}
\end{center}
\end{figure}

We make the following assumptions:

${\bf({\rm A}_1)}$  Let $ {\bm \xi}=\varepsilon^{-1} \mathbf{x} $, and assume that the elements
$ a_{ij}({\bm \xi}) $, $\eta_{ij}({\bm \xi}) $ and $ \mu_{ij}({\bm \xi}) $
 of matrices $ A({\bm \xi}) $, $ \eta({\bm \xi})$ and
 $ \mu({\bm\xi})$; $ V_c({\bm \xi}) $ and $ V_{xc}(\mathbf{x},{\bm\xi}) $  are rapidly oscillating
 1-periodic real functions of space, respectively.

 ${\bf({\rm A}_2)}$ The matrices $(a_{ij}({\bm\xi}))$, $ (\eta_{ij}({\bm\xi})) $ and $ (\mu_{ij}({\bm\xi})) $ are real symmetric, and
satisfy the following uniform elliptic conditions in ${\bm\xi}$, i.e.
 \begin{equation*}
  \begin{array}{@{}l@{}}
 {\displaystyle   \alpha_0 |\mathbf{y}|^2 \le a_{ij}({\bm\xi})y_i y_j \le \alpha_1 |\mathbf{y}|^2,
   \quad
     \beta_0 |\mathbf{y}|^2 \le \eta_{ij}({\bm\xi}) y_i y_j \le \beta_1 |\mathbf{y}|^2,}\\[2mm]
  {\displaystyle   \gamma_0 |\mathbf{y}|^2 \le \mu_{ij}({\bm\xi})y_i y_j \le \gamma_1 |\mathbf{y}|^2,\quad |\mathbf{y}|^2=y_i y_i,}\\[2mm]
  {\displaystyle     0 < \alpha_0 \le \alpha_1, \quad 0<\beta_0\le \beta_1,\quad 0 < \gamma_0 \le \gamma_1,
   \quad \forall \mathbf{y}=(y_1, y_2, y_3)\in
   \mathbb{R}^3,}
  \end{array}
  \end{equation*}
where $\alpha_{0} $, $\alpha_{1}$, $ \beta_0 $, $ \beta_1 $, $ \gamma_0 $, $ \gamma_1 $ are constants
independent of $\varepsilon$.

 ${\bf({\rm A}_3)}$   $ a_{ij} $, $ \eta_{ij} $ , $ \mu_{ij} $,  $ V_c\in
L^{\infty}(\mathbb{R}^3) $, $ V_{xc}\in L^{\infty}(\Omega\times \mathbb{R}^3) $.

 ${\bf({\rm A}_4)}$ $\mathbf{f}\in
L^{2}(0,T;(L^{2}(\Omega))^{3}) $, $
\Psi_0\in \mathbb{H}^1_0(\Omega) $, $ {\bm\varphi}\in (H^1(\Omega))^3 $, $ {\bm\psi}\in (H^1(\Omega))^3 $,
where $ H_0^1(\Omega) $ and $ \mathbb{H}^1_0(\Omega) $ denote the Sobolev spaces of the real-valued and the complex-valued
functions, respectively.

Maxwell-Schr\"{o}dinger system with rapidly oscillating discontinuous coefficients
originates from the interaction of matter and electromagnetic fields in heterogeneous nanostructures(see, e.g.,
\cite{Sch}). There are a great number of applications in laser physics, quantum Hall effects,
superconductivity and semiconductor optics and transport phenomena in heterogeneous photoelectronic
devices (see, e.g., \cite{Ez, Har, Rei, Ul}). We would like to state that, in this paper we choose the
Schr\"{o}dinger equation (\ref{eq:1-3}) with the effective mass
approximation (EMA) and do not introduce the Kohn-Sham equation.
There are three main reasons. First, if we use Kohn-Sham equation
for the heterogeneous photoelectronic devices, it is time-consuming
and difficult, due to the limitation of the computing scaling.
Second, the most widely used techniques to calculate the electronic
levels in nanostructures are EMA and its extension the multiband $
k\cdot p $ method . They have been particularly successful in the
case of heterostructures (see, e.g., \cite{Del, Ul}). Thirdly, Cao et al. \cite{Cao-2} gave a
reasonable interpretation why EMA has the high accuracy for
calculating the band structures of semiconductor materials in the
vicinity of point $\Gamma$, from the viewpoint of mathematics. One
of key points of \cite{Cao-2} is to use an important result of
\cite{All-1}, namely EMA in physics is equivalent to the
homogenization method in mathematics under some assumptions.

The main difficult points to analyze and solve the problem (\ref{eq:1-6})
have the following: nonlinear, nonconvex coupled system and
rapidly oscillating discontinuous coefficients. This paper
focuses on discussing the multiscale approach and the convergence
for the problem (\ref{eq:1-6}). A direct numerical method such as the finite-difference time-domain (FDTD)
method or finite element (FE) method for solving the problem (\ref{eq:1-6})
cannot produce accurate numerical solutions unless a very fine mesh is required.
On the other hand, since the elements of the coefficient matrices
$ A(\frac{\displaystyle \mathbf{x}}{\displaystyle \varepsilon}) $
, $ \eta(\frac{\displaystyle \mathbf{x}}{\displaystyle \varepsilon}) $
and $ \mu(\frac{\displaystyle \mathbf{x}}{\displaystyle \varepsilon}) $
are discontinuous, the regularity of the solution for problem
(\ref{eq:1-6}) maybe will be quite low (see, e.g., \cite{Cos-1, Cos-2}).
It might be extremely difficult to derive the convergence results for FDTD or FEM.

We recall that the homogenization method gives the overall solution
behavior by incorporating the fluctuations due to the heterogeneities
(see, e.g., \cite{Ben, Ji, Ole}). There are numerous important studies for the homogenization
method of Maxwell's equations, it is impossible to mention
all contributions here. We refer to \cite{Ben, Ji, San, Wel-1, Wel-2}.
Other related studies have
been reported in \cite{Bos, Cao-3, Cao-4, Wel-3, Zhang-1, Zhang-2}.
For the homogenization method of Schr\"{o}dinger equation, we refer to
\cite{All-1, Ben, Cao-2}. It should be stated that the theoretical results
of the homogenization method for Schr\"{o}dinger-Poisson system or Maxwell-Schr\"{o}dinger
system with rapidly oscillating coefficients are very limited.
Recently, Zhang, Cao and Luo \cite{Zhang-L} developed the homogenization method
and the multiscale method for the stationary Schr\"{o}dinger-Poisson system
in heterogeneous nanostructures and derived the convergence results of
these methods. To the best of our knowledge, there are few theoretical
results for the homogenization method and the multiscale method of
Maxwell-Schr\"{o}dinger system.

The main objectives of this paper are to present the homogenization method and
the multiscale asymptotic methods for a kind of time-dependent Maxwell-Schr\"{o}dinger system
(\ref{eq:1-6}), to develop the associated numerical algorithms
and to derive the convergence results for the present method.
The new results obtained in this paper are concluded as the following:

(i) For a bounded smooth domain or a bounded polyhedral
convex domain $ \Omega\subset \mathbb{R}^3 $,  we
present the homogenization method and the multiscale asymptotic method for the Maxwell-Schr\"{o}dinger
system (\ref{eq:1-6}) and derive the associated convergence results, see Theorems~\ref{thm2-1} and ~\ref{thm3-1}
and Corollary~\ref{cor3-1}.

(ii) We develop the multiscale numerical algorithms based on the multiscale asymptotic expansions
(\ref{eq:3-17})-(\ref{eq:3-19}) of the solution for the problem (\ref{eq:1-6}) and obtain some
theoretical results, see Propositions~ \ref{prop4-1} and ~\ref{prop4-2}.

(iii) The numercial results of the multiscale asymptotic method for
the Maxwell-Schr\"{o}dinger system (\ref{eq:1-6}) without the exchange-correlation potential
and with the exchange-correlation potential are provided and the validity of
homogenization and multiscale method is confirmed.

The remainder of this paper is organized as follows. In section 2,
we present the homogenization method and the multiscale asymptotic
method for the time-dependent Maxwell-Schr\"{o}dinger system
(\ref{eq:1-6}) with rapidly oscillating discontinuous
coefficients and derive the convergence results. In section 3, we
use respectively the finite element method to discretize the cell
problems and the modified homogenized Maxwell-Schr\"{o}dinger
system, and obtain the convergence results. We use the
self-consistent iterative method (SCF) to solve the discrete system of
the modified homogenized Maxwell-Schr\"{o}dinger system. In order to
accelerate the convergence speed of the iterative algorithm, the
simple mixed method is applied. For more information, we refer to
\cite{Ha, Hu, Pul-1}. Finally, the numerical examples and some remarks are given.

\section{The homogenization method and the convergence}\label{sec-2}

In this section, we first introduce the results of well-posedness for
the time-dependent Maxwell-Schr\"{o}dinger system. Then we develop
the homogenization method for the problem (\ref{eq:1-6})
and give the convergence result of the homogenization method.

\subsection{The well-posedness of the problem (\ref{eq:1-6})}

Many authors have studied the existence of the solution for the time-dependent Schr\"{o}dinger-Maxwell system.
We recall some important studies about the problem.
Nakamitsu and Tsutsumi \cite{Nak} proved that the initial value problem for the Maxwell-Schr\"{o}dinger
system in the Lorentz gauge is globally well-posed in a space of smooth functions in dimensions one and two,
and locally well-posed in dimension three. Tsutsumi\cite{Tsu-1} showed, by constructing the
modified wave operator, that there exist global smooth solutions in the Coulomb gauge for a certain
class of scattered data as $ t\rightarrow +\infty $. The existence of global finite-energy solutions
was established for the initial value problem for the Maxwell-Schr\"{o}dinger system in the Coulomb, Lorentz
and temporal gauges by Guo, Nakamitsu and Strauss in \cite{Guo}.
Ginibre and Velo \cite{Gin} studied the theory of scattering for the Maxwell-Schr\"{o}dinger system
in space dimension 3, in the Coulomb gauge. The existence of the modified wave operators for the system
in 3+1 dimension space time was derived in the special case. Nakamura and Wada \cite{Nak-1, Nak-2}
investigated the time local and global well-posedness for the Maxwell-Schr\"{o}dinger equations
in Sobolev spaces in three spatial dimensions. One of the main results is that
the solutions exist time globally for large data. Recently Wada \cite{Wa} proved unique solvability of the
Maxwell-Schr\"{o}dinger equations in $ \mathbb{R}^{1+2} $ spacetime.
Benci and Fortunato \cite{Benci} investigated the existence of charged solitons for the nonlinear Maxwell-Schr\"{o}dinger
equations.

In summary, we would like to state that the above results are obtained based on the following assumptions:
(i) the initial problem (i.e. the Cauchy problem) of the Maxwell-Schr\"{o}dinger equations with constant coefficients; (ii)
no exchange-correlation potential function. Besides, they studied the well-posedness for the Maxwell-\\
Schr\"{o}dinger equations
with the vector potential and scalar potential
 $ \mathbf{A}-\varphi$  instead of the length gauge $-e\mathbf{E}\cdot\mathbf{x}$ used in this paper
 (see, e.g., \cite{Cohen, Oh}).
For the well-posedness of the solution of the initial-boundary value problem
for the Maxwell-Schr\"{o}dinger with rapidly oscillating discontinuous coefficients
seems to be open. The well-posedness of the problem (\ref{eq:1-6}) is not the issue of this paper.
Here and in the sequel, we assume that the problem (\ref{eq:1-6}) has one and only one weak solution.

\subsection{The homogenization method}

In this section, we will study the asymptotic behaviour of the solution for the problem (\ref{eq:1-6})
as $ \varepsilon\rightarrow 0 $, i.e. the homogenization method. As usual we introduce two variables:
$ \mathbf{x} $ and $ {\bm \xi=\varepsilon^{-1} \mathbf{x}} $.
For simplicity, we assume that the reference cell $ Q=(0,1)^3 $ without loss of generality. For
general cases, let $ Q=(0,l_1) \times(0,l_2) \times (0, l_3)
$, we refer to a classical book \cite{Ben}. For the coefficient matrices $ (a_{ij}({\bm \xi})) $,
$ (\eta_{ij}({\bm \xi})) $ and $ (\mu_{ij}({\bm \xi})) $, we define the following scalar
cell functions:
\begin{equation}\label{eq:2-1}
\left\{
\begin{array}{@{}l@{}}
{\displaystyle{\frac{\partial}{\partial \xi_i}}\big(
a_{ij}({\bm \xi})\displaystyle{\frac{\partial
\theta_k^a({\bm \xi})}{\partial \xi_j}}\big)=-\frac{
\partial}{\partial \xi_i}\big(a_{ik}({\bm \xi})\big), \quad {\bm \xi}\in Q,}\\[2.3mm]
{\displaystyle \theta_k^a({\bm \xi}) \,\,\, \hbox{is 1-periodic}\,\,\, \hbox{in}\,\,\,
{\bm \xi},}\\[2.1mm]
{\displaystyle \int_{Q} \theta_k^a({\bm\xi})\mathnormal{d}{\bm \xi}=0, \quad
k=1,2,3,}
\end{array}
\right.
\end{equation}
\begin{equation}\label{eq:2-2}
\left\{
\begin{array}{@{}l@{}}
{\displaystyle{\frac{\partial}{\partial \xi_i}}\big(
\eta_{ij}({\bm \xi})\displaystyle{\frac{\partial
\theta_k^\eta({\bm \xi})}{\partial \xi_j}}\big)=-\frac{
\partial}{\partial \xi_i}\big(\eta_{ik}({\bm \xi})\big), \quad {\bm \xi}\in Q,}\\[2.3mm]
{\displaystyle \theta_k^\eta({\bm \xi}) \,\,\, \hbox{is 1-periodic}\,\,\, \hbox{in}\,\,\,
{\bm \xi},}\\[2.1mm]
{\displaystyle \int_{Q} \theta_k^\eta({\bm\xi})\mathnormal{d}{\bm \xi}=0, \quad
k=1,2,3,}
\end{array}
\right.
\end{equation}
and
\begin{equation}\label{eq:2-3}
\left\{
\begin{array}{@{}l@{}}
{\displaystyle{\frac{\partial}{\partial \xi_i}}\big(
\mu_{ij}({\bm \xi})\displaystyle{\frac{\partial
\theta_k^\mu({\bm \xi})}{\partial \xi_j}}\big)=-\frac{
\partial}{\partial \xi_i}\big(\mu_{ik}({\bm \xi})\big), \quad {\bm \xi}\in Q,}\\[2.3mm]
{\displaystyle \theta_k^\mu({\bm \xi}) \,\,\, \hbox{is 1-periodic}\,\,\, \hbox{in}\,\,\,
{\bm \xi},}\\[2.1mm]
{\displaystyle \int_{Q} \theta_k^\mu({\bm\xi})\mathnormal{d}{\bm \xi}=0, \quad
k=1,2,3.}
\end{array}
\right.
\end{equation}

\begin{rem}\label{rem2-1}
Under the assumptions $({\rm A}_1)$--$({\rm A}_3)$, the existence and uniqueness of
the solutions for the cell problems (\ref{eq:2-1})-(\ref{eq:2-3}) can be established
based upon Lax-Milgram lemma.
\end{rem}

The homogenized coefficient matrices $ \widehat{A}=(\hat{a}_{ij}) $, $ \widehat{\eta}=(\hat{\eta}_{ij}) $ and
$ \widehat{\mu}=(\hat{\mu}_{ij}) $ are calculated by
\begin{equation}\label{eq:2-4}
\begin{array}{@{}l@{}}
{\displaystyle \hat{a}_{ij}=\int\limits_Q \big(a_{ij}({\bm\xi})+a_{ik}({\bm\xi})
\frac{\partial \theta_j^a({\bm\xi})}{\partial \xi_k}\big)\mathnormal{d}{\bm\xi},\quad
\hat{\eta}_{ij}=\int\limits_Q \big(\eta_{ij}({\bm\xi})+\eta_{ik}({\bm\xi})
\frac{\partial \theta_j^\eta({\bm\xi})}{\partial \xi_k}\big)\mathnormal{d}{\bm\xi},}\\[2mm]
{\displaystyle \qquad \qquad\qquad \hat{\mu}_{ij}=\int\limits_Q \big(\mu_{ij}({\bm\xi})+\mu_{ik}({\bm\xi})
\frac{\partial \theta_j^\mu({\bm\xi})}{\partial \xi_k}\big)\mathnormal{d}{\bm\xi},\quad i,j=1,2,3.}
\end{array}
\end{equation}

Hence, the homogenized Maxwell-Schr\"{o}dinger equations can formally be written as follows:
\begin{equation}\label{eq:2-5}
\left\{
\begin{array}{@{}l@{}}
{\displaystyle i\frac{\partial \Psi^0(\mathbf{x},t)}{\partial t}=
-\nabla \cdot \big(\widehat{A}\nabla
\Psi^0(\mathbf{x},t)\big)
+\big(\langle V_c\rangle-\mathbf{E}^0\cdot\widehat{\bm{\zeta}}
+V_{xc}[\rho^0]\big)
\Psi^0(\mathbf{x},t),}\\[2.3mm]
{\displaystyle \qquad \qquad \qquad \qquad (\mathbf{x},t)\in
\Omega\times(0,T),}\\[2.3mm]
{\displaystyle \widehat{\eta}\frac{\partial \mathbf{E}^0(\mathbf{x},t)}{\partial t}
=\mathbf{curl}\/ \mathbf{H}^0(\mathbf{x},t)+\mathbf{f}(\mathbf{x},t)
-\mathbf{J}_q^0,\,\, \nabla\cdot \mathbf{f}=0,\,\,(\mathbf{x},t)\in
\Omega\times(0,T),}\\[2.3mm]
{\displaystyle \widehat{\mu}
\frac{\partial \mathbf{H}^0(\mathbf{x},t)}{\partial t}=-\mathbf{curl}\/ \mathbf{E}^0(\mathbf{x},t)
,\,\, (\mathbf{x},t)\in
\Omega\times(0,T),}\\[2.3mm]
{\displaystyle \nabla\cdot \big(\widehat{\eta}\mathbf{E}^0(\mathbf{x},t)\big)=\rho^0(\mathbf{x},t),
\quad\nabla\cdot \big(\widehat{\mu}\mathbf{H}^0(\mathbf{x},t)\big)=0,
\,\, (\mathbf{x},t)\in
\Omega\times(0,T),}\\[2.3mm]
{\displaystyle \widehat{\bm{\zeta}}=-\mathbf{x},\,\, \rho^0=N{\vert\Psi^0\vert}^2
,\,\,\mathbf{J}_q^0=iN\big[(\overline{\Psi}^0) \widehat{A}\nabla\Psi^0-
\Psi^0 \widehat{A}\nabla\overline{\Psi}^0\big].}
\end{array}
\right.
\end{equation}

We take (\ref{eq:2-5}) to hold in $\Omega $ subject to the boundary conditions
\begin{equation}\label{eq:2-6}
{\displaystyle \Psi^0(\mathbf{x},t)=0,\quad \mathbf{E}^0(\mathbf{x},t)\times\mathbf{n}=0, \quad (\mathbf{x},t)\in
\partial \Omega\times(0,T).}
\end{equation}

For initial conditions are taken as
\begin{equation}\label{eq:2-7}
{\displaystyle \Psi^0(\mathbf{x},0)=\Psi_0(\mathbf{x}),\quad
\mathbf{E}^0(\mathbf{x},0)=\widehat{\eta}^{-1}{\bm\varphi}(\mathbf{x}),\quad
\mathbf{H}^0(\mathbf{x},0)=\widehat{\mu}^{-1}{\bm\psi}(\mathbf{x}).}
\end{equation}

\begin{rem}\label{rem2-2}
We observe that the homogenized problem (\ref{eq:2-5}) is a Maxwell-Schr\"{o}dinger system with
constant coefficients, it is much simpler to carry out the theoretical analysis and
numerical computations than those of the original problem
(\ref{eq:1-6}). This is the motivation of the study in this paper.
\end{rem}

Next we derive the convergence result of the homogenization method for the problem (\ref{eq:1-6}).
\begin{theorem}\label{thm2-1}
Let $\Omega \subset \mathbb{R}^3 $ be a
bounded smooth domain or a bounded polyhedral convex domain
with a periodic microstructure. Suppose that $ (\Psi^\varepsilon, \mathbf{E}^\varepsilon, \mathbf{H}^\varepsilon) $
be the unique weak solution of the original problem (\ref{eq:1-6}) without the exchange-correlation potential,
i.e. $ V_{xc}(\rho^\varepsilon)\equiv 0 $, and let
$ (\Psi^0, \mathbf{E}^0, \mathbf{H}^0) $ be the unique weak solution of the associated homogenized
problem (\ref{eq:2-5}). Under the assumptions $ (A_1)-(A_4) $, we have
\begin{equation}\label{eq:2-8}
\begin{array}{@{}l@{}}
{\displaystyle \Psi^\varepsilon\stackrel{\mathrm{w}^\ast}{\rightharpoonup} \Psi^0,\quad {\rm weakly \,\,star\,\, in}\quad L^\infty(0,T; \/\mathbb{H}_0^1(\Omega)),\quad {\rm as}\quad
\varepsilon\rightarrow 0,}\\[2mm]
{\displaystyle \mathbf{E}^\varepsilon\stackrel{\mathrm{w}^\ast}{\rightharpoonup} \mathbf{E}^0,\quad {\rm weakly \,\, star\,\, in}\quad L^\infty(0,T; \/(L^2(\Omega))^3),\quad {\rm as}\quad
\varepsilon\rightarrow 0,}\\[2mm]
{\displaystyle \mathbf{H}^\varepsilon\stackrel{\mathrm{w}^\ast}{\rightharpoonup} \mathbf{H}^0,\quad {\rm weakly \,\, star\,\, in}\quad L^\infty(0,T;\/ (L^2(\Omega))^3),\quad {\rm as}\quad
\varepsilon\rightarrow 0.}
\end{array}
\end{equation}
\end{theorem}

\begin{proof}
To begin, we prove that $\|\rho^\varepsilon\|_{H^{-1}(\Omega)}\leq C $, where
$ C $ is a constant independent of $ \varepsilon $. We observe
that $ \rho^\varepsilon $ in this paper is viewed as the carrier
density operator $ \mathcal{N} $ from Definition 2.12 of
\cite{Kai-1}. Under the assumptions $ (A_1)-(A_4) $, one can check
that $ \rho^\varepsilon $ satisfies all conditions of Corollary 5.5
of \cite{Kai-1}. Applying this corollary, we prove that
$\|\rho^\varepsilon\|_{H^{-1}(\Omega)}\leq C$,
where $ C $ is a constant independent of $ \varepsilon $.

We recall the equation $(\ref{eq:1-6})_4 $. Under the Coulomb gauge, we have
\begin{equation*}
{\displaystyle -\nabla\cdot \big(\eta(\frac{\mathbf{x}}{\varepsilon})
\nabla \phi^\varepsilon(\mathbf{x},t)\big)=\rho^\varepsilon(\mathbf{x},t),\quad \phi^\varepsilon\in H_0^1(\Omega),}
\end{equation*}
where $ \mathbf{E}^\varepsilon(\mathbf{x},t)=-\nabla \phi^\varepsilon(\mathbf{x},t) $,
$ \phi^\varepsilon(\mathbf{x},t)=\phi^\varepsilon(\rho^\varepsilon) $
and $ t\in (0,T) $ plays the role of a parameter. Using {\it a-priori} estimates of elliptic
equations (see \cite[p.~181]{Gil}), we get $\|\phi^\varepsilon(\rho^\varepsilon)\|_{H^1_0(\Omega)}
\leq C\|\rho^\varepsilon\|_{H^{-1}(\Omega)}\leq C $,
where $ C $ is a constant independent of $ \varepsilon $. Therefore, for any fixed $ t\in (0,T) $,
there is a subsequence, without confusion still denoted by $ \phi^\varepsilon(\rho^\varepsilon) $,
such that
\begin{equation}\label{eq:2-9}
{\displaystyle\phi^\varepsilon \stackrel{\mathrm{w}}{\rightharpoonup}\widehat{\phi}^0 \quad \hbox{weakly \,\,in}
\quad H^1_0(\Omega) \quad \hbox{as}\quad
\varepsilon\rightarrow 0.}
\end{equation}
Furthermore, for any fixed $ t\in (0,T) $, using the homogenization result of the elliptic equation
(see, e.g., \cite[p.~151-152]{Ji} or \cite[p.~29-30]{Ben}),
we get
\begin{equation}\label{eq:2-10}
\begin{array}{@{}l@{}}
{\displaystyle \mathbf{E}^\varepsilon\equiv -\nabla \phi^\varepsilon \stackrel{\mathrm{w}}{\rightharpoonup}-\nabla\widehat{\phi}^0
\stackrel{\mathrm{def}}{=}\widehat{\mathbf{E}}^0,\quad \hbox{weakly \,\, in}
\quad (L^2(\Omega))^3 \quad \hbox{as}\quad
\varepsilon\rightarrow 0,}\\[2mm]
{\displaystyle \eta(\frac{\mathbf{x}}{\varepsilon})\nabla \phi^\varepsilon
\stackrel{\mathrm{w}}{\rightharpoonup}\widehat{\eta}\nabla \widehat{\phi}^0,
 \quad \hbox{weakly \,\, in}
\quad (L^2(\Omega))^3 \quad \hbox{as}\quad
\varepsilon\rightarrow 0.}\\
\end{array}
\end{equation}

Thanks to $ (A_1) $ and $ (A_3) $, $ V_c $ is a 1-periodic function
in $ L^p(Q) $ and satisfies all conditions of Theorem 2.6 of
\cite[p.~33]{Cio}. Then we get
\begin{equation}\label{eq:2-11}
{\displaystyle V_c(\frac{x}{\varepsilon})\rightarrow
\langle V_c \rangle\,\,\, \hbox{weakly \,\,in}\,\,\, L^p(\Omega),\,\,\, \hbox{as}\,\,\,\varepsilon\rightarrow 0,\quad  1<p<+\infty.}
 \end{equation}

We recall $(\ref{eq:1-6})_1 $, and consider the following modified Sch\"{o}dinger
equation:
\begin{equation}\label{eq:2-12}
{\displaystyle i\frac{\partial \widehat{\Psi}^\varepsilon(\mathbf{x},t)}{\partial t}
=-\nabla\cdot \big(A(\frac{\mathbf{x}}{\varepsilon})\nabla
\widehat{\Psi}^\varepsilon(\mathbf{x},t)\big)
+\big(\langle V_c\rangle-\widehat{\mathbf{E}}^0\cdot \widehat{\bm \zeta}\big)
\widehat{\Psi}^\varepsilon(\mathbf{x},t).}
\end{equation}

Define $ (u,v)=\int_\Omega u \bar{v}\mathnormal{d}\mathbf{x} $, $
a^\varepsilon(u,v)=\int_\Omega a_{ij}(\frac{\displaystyle \mathbf{x}}{\displaystyle \varepsilon})
\frac{\displaystyle\partial u}{\displaystyle\partial x_j}\frac{\displaystyle\partial \bar{v}}{\displaystyle\partial x_i}\mathnormal{d}\mathbf{x} $,
where $ \bar{v} $ denotes the complex conjugate of $ v $. For simplicity, we assume that
$ \widehat{\Psi}^\varepsilon(\mathbf{x},0)=\Psi_0(\mathbf{x})\equiv 0 $ without loss of generality.
The variational form of (\ref{eq:2-12}) is the following:
\begin{equation}\label{eq:2-13}
{\displaystyle -i(\dot{\widehat{\Psi}^\varepsilon}, v)+a^\varepsilon(\widehat{\Psi}^\varepsilon, v)
+\big((\langle V_c\rangle-\widehat{\mathbf{E}}^0\cdot \widehat{\bm \zeta})\widehat{\Psi}^\varepsilon,v\big)=0,}
\end{equation}
where $ \dot{\widehat{\Psi}^\varepsilon} $ denotes the derivative of $ \widehat{\Psi}^\varepsilon $
with respect to $ t $. Taking $ v=\dot{\widehat{\Psi}^\varepsilon} $ in (\ref{eq:2-13}) and taking the real part,
we get
\begin{equation*}
{\displaystyle \frac{1}{2}\frac{d}{d t} a^\varepsilon(\widehat{\Psi}^\varepsilon)
=\frac{1}{2}\mathcal{R}e\frac{d}{d t}\big((\langle V_c\rangle-\widehat{\mathbf{E}}^0\cdot \widehat{\bm \zeta})\widehat{\Psi}^\varepsilon,
\widehat{\Psi}^\varepsilon\big)+\frac{1}{2}\mathcal{R}e\big((\dot{\widehat{\mathbf{E}}}^0\cdot \widehat{\bm \zeta})\widehat{\Psi}^\varepsilon,
\widehat{\Psi}^\varepsilon\big),}
\end{equation*}
and consequently
\begin{equation*}
{\displaystyle a^\varepsilon(\widehat{\Psi}^\varepsilon(t))=\mathcal{R}e\big((\langle V_c\rangle-\widehat{\mathbf{E}}^0(t)\cdot \widehat{\bm \zeta})\widehat{\Psi}^\varepsilon(t),
\widehat{\Psi}^\varepsilon(t)\big)+\mathcal{R}e\int_0^t\big((\dot{\widehat{\mathbf{E}}}^0\cdot\widehat{\bm \zeta})\widehat{\Psi}^\varepsilon,\widehat{\Psi}^\varepsilon\big)
\mathnormal{d}t,}
\end{equation*}
where $ \mathcal{R}e(u) $ denotes the real port of $ u $. Hence it follows that (one obtains
a preliminary estimate by taking $ v=\widehat{\Psi}^\varepsilon $ in (\ref{eq:2-13}))
\begin{equation*}
{\displaystyle \|\widehat{\Psi}^\varepsilon\|_{L^\infty(0,T; \/\mathbb{H}_0^1(\Omega))}\leq C,}
\end{equation*}
where $ C $ is a positive independent of $ \varepsilon $. Then we can extract a subsequence, still denoted by
$ \widehat{\Psi}^\varepsilon $, such that
\begin{equation}\label{eq:2-14}
{\displaystyle \widehat{\Psi}^\varepsilon\stackrel{\mathrm{w}^\ast}{\rightharpoonup} \widehat{\Psi}^0,\quad {\rm weakly\,\, star \,\,in}\quad L^\infty(0,T; \/\mathbb{H}_0^1(\Omega)),\quad {\rm as}\quad
\varepsilon\rightarrow 0.}
\end{equation}
Here $ \widehat{\Psi}^0\in L^\infty(0,T; \mathbb{H}_0^1(\Omega)) $ is the solution of the following Schr\"{o}dinger equation:
\begin{equation*}
\left\{
\begin{array}{@{}l@{}}
{\displaystyle
-i(\dot{\widehat{\Psi}^0}, v)+a^0(\widehat{\Psi}^0, v)
+\big((\langle V_c\rangle-\widehat{\mathbf{E}}^0\cdot \widehat{\bm \zeta})\widehat{\Psi}^0,v\big)=0,\quad
\forall v\in \mathbb{H}_0^1(\Omega),}\\[2mm]
{\displaystyle \widehat{\Psi}^0(x,0)=0,}
\end{array}
\right.
\end{equation*}
where $ a^0(u,v)=\int_\Omega \hat{a}_{ij}
\frac{\displaystyle\partial u}{\displaystyle\partial x_j}
\frac{\displaystyle\partial \bar{v}}{\displaystyle\partial x_i}\mathnormal{d}\mathbf{x} $,
and $ \widehat{A}=(\hat{a}_{ij}) $ is the homogenized coefficient matrix of $ A(\frac{\displaystyle \mathbf{x}}
{\displaystyle \varepsilon}) $. Following the lines of the proof of Theorem 11.4 in \cite[p.~211-214]{Cio} (see also
the proof of Theorem 12.6 in \cite[p.~231-234]{Cio}),
we can prove that
\begin{equation}\label{eq:2-15}
\begin{array}{@{}l@{}}
{\displaystyle \widehat{\Psi}^\varepsilon\rightarrow \widehat{\Psi}^0,\quad {\rm strongly \,\,\,in}\quad L^2(0,T; \/\mathbb{L}^2(\Omega)),\quad {\rm as}\quad
\varepsilon\rightarrow 0,}\\[2mm]
{\displaystyle A(\frac{\mathbf{x}}{\varepsilon})\nabla\widehat{\Psi}^\varepsilon
 \stackrel{\mathrm{w}}{\rightharpoonup}
 \widehat{A}\nabla\widehat{\Psi}^0,\quad {\rm weakly \,\,\,in}\quad L^2(0,T; \/(\mathbb{L}^2(\Omega))^3),\quad {\rm as}\quad
\varepsilon\rightarrow 0,}
\end{array}
\end{equation}
where $ \mathbb{L}^2(\Omega) $ denotes the Sobolev space of the complex-valued functions.

If there is not the exchange-correlation potential $ V_{xc}(\rho^\varepsilon) $ in $(\ref{eq:1-6})_1 $,
subtracting $(\ref{eq:1-6})_1 $ from (\ref{eq:2-12}) gives
\begin{equation}\label{eq:2-16}
\begin{array}{@{}l@{}}
{\displaystyle i \frac{\partial (\Psi^\varepsilon-\widehat{\Psi}^\varepsilon)}{\partial t}
=-\nabla\cdot \big(A(\frac{\mathbf{x}}{\varepsilon})\nabla (\Psi^\varepsilon-\widehat{\Psi}^\varepsilon)\big)
+\big(V_c(\frac{\mathbf{x}}{\varepsilon})-\mathbf{E}^\varepsilon\cdot \widehat{{\bm\zeta}}\big)
(\Psi^\varepsilon-\widehat{\Psi}^\varepsilon)}\\[2mm]
{\displaystyle \qquad \qquad \qquad \quad+\big((V_c(\frac{\mathbf{x}}{\varepsilon})-\langle
V_c \rangle)-(\mathbf{E}^\varepsilon-\widehat{\mathbf{E}}^0)\cdot \widehat{{\bm\zeta}}\big)\widehat{\Psi}^\varepsilon.}
\end{array}
\end{equation}

Setting $ u^\varepsilon=\Psi^\varepsilon-\widehat{\Psi}^\varepsilon $, the variational form of (\ref{eq:2-16})
is as follows:
\begin{equation}\label{eq:2-17}
\begin{array}{@{}l@{}}
{\displaystyle -i(\dot{u}^\varepsilon, u^\varepsilon)+a^\varepsilon(u^\varepsilon, u^\varepsilon)
+\big((V_c(\frac{\mathbf{x}}{\varepsilon})-\mathbf{E}^\varepsilon\cdot \widehat{{\bm\zeta}})u^\varepsilon,
u^\varepsilon\big)}\\[2mm]
{\displaystyle \quad =-\big([(V_c(\frac{\mathbf{x}}{\varepsilon})-\langle V_c \rangle)-(\mathbf{E}^\varepsilon-\widehat{\mathbf{E}}^0)\cdot \widehat{{\bm\zeta}}]
\widehat{\Psi}^\varepsilon, u^\varepsilon\big).}
\end{array}
\end{equation}

Thanks to $ (A_3) $, $ V_c $ is
bounded.  It follows from (\ref{eq:2-10}) that $ \|\mathbf{E}^\varepsilon\|_{(L^2(\Omega))^3}\leq \|\phi^\varepsilon\|_{H_0^1(\Omega)}\leq C $,
where $ C $ is a constant independent of $ \varepsilon $.
Then, there is a sufficiently large positive number $
\alpha>0 $ such that
\begin{equation}\label{eq:2-18}
{\displaystyle (\alpha+V_c(\frac{\mathbf{x}}{\varepsilon})-\mathbf{E}^\varepsilon\cdot \widehat{{\bm \zeta}})\geq \frac{\alpha}{2}>0.}
\end{equation}
Hence, we take the real part in (\ref{eq:2-17}) and get
\begin{equation}\label{eq:2-19}
\begin{array}{@{}l@{}}
{\displaystyle a^\varepsilon(u^\varepsilon, u^\varepsilon)+\big((\alpha+V_c(\frac{\mathbf{x}}{\varepsilon})-\mathbf{E}^\varepsilon\cdot \widehat{{\bm\zeta}})u^\varepsilon,
u^\varepsilon\big)-(\alpha u^\varepsilon, u^\varepsilon)}\\[2mm]
{\displaystyle \quad =-\mathcal{R}e\big([(V_c(\frac{\mathbf{x}}{\varepsilon})-\langle V_c \rangle)-(\mathbf{E}^\varepsilon-\widehat{\mathbf{E}}^0)\cdot \widehat{{\bm\zeta}}]
\widehat{\Psi}^\varepsilon, u^\varepsilon\big).}
\end{array}
\end{equation}

Due to the presence of the term $ -(\alpha u^\varepsilon, u^\varepsilon) $ on the left
side of (\ref{eq:2-19}), maybe the left side is
not a coercive bilinear form. To overcome this difficulty, we use
the trick of \cite[p.~89]{Monk}. To this end, we assume that $ w^\varepsilon $ is the solution of
the following problem:
\begin{equation}\label{eq:2-20}
\begin{array}{@{}l@{}}
{\displaystyle a^\varepsilon(w^\varepsilon, w^\varepsilon)+\big((\alpha+V_c(\frac{\mathbf{x}}{\varepsilon})-\mathbf{E}^\varepsilon\cdot \widehat{{\bm\zeta}})w^\varepsilon,
w^\varepsilon\big)+(\alpha w^\varepsilon, w^\varepsilon)}\\[2mm]
{\displaystyle \quad =-\mathcal{R}e\big([(V_c(\frac{\mathbf{x}}{\varepsilon})-\langle V_c \rangle)-(\mathbf{E}^\varepsilon-\widehat{\mathbf{E}}^0)\cdot \widehat{{\bm\zeta}}]
\widehat{\Psi}^\varepsilon, w^\varepsilon\big).}
\end{array}
\end{equation}

We thus obtain
\begin{equation*}
{\displaystyle \|w^\varepsilon\|_{L^\infty(0,T;\/ \mathbb{H}_0^1(\Omega))}\leq C\big\{
\|V_c(\frac{\mathbf{x}}{\varepsilon})-\langle V_c \rangle\|_{H^{-1}(\Omega)}
+\|\mathbf{E}^\varepsilon-\widehat{\mathbf{E}}^0\|_{L^2(0,T;\/ H^{-1}(\Omega))}\big\}.}
\end{equation*}

From (\ref{eq:2-10}) and (\ref{eq:2-11}), we have
\begin{equation}\label{eq:2-21}
{\displaystyle \|w^\varepsilon\|_{L^\infty(0,T; \/\mathbb{H}_0^1(\Omega))}
\rightarrow 0,\quad {\rm as }\quad \varepsilon\rightarrow 0.}
\end{equation}

Similarly to (4.18) of \cite[p.~91]{Monk}, we get
$(I+K)u^\varepsilon=w^\varepsilon $,
where $ I $ is an identity operator from $ \mathbb{L}^2(\Omega)\rightarrow \mathbb{L}^2(\Omega) $, and the operator $ K $
is a bounded and compact map from $ \mathbb{L}^2(\Omega)\rightarrow
\mathbb{L}^2(\Omega) $ shown as in (4.15) of \cite[p.~90]{Monk}.
Furthermore, we have $ \|u^\varepsilon\|_{\mathbb{L}^2(\Omega)}\leq C \|w^\varepsilon\|_{\mathbb{L}^2(\Omega)} $,
where $ C $ is a constant independent of $ \varepsilon $.

Combining (\ref{eq:2-19}) and (\ref{eq:2-21}) implies
\begin{equation*}
\begin{array}{@{}l@{}}
{\displaystyle \|u^\varepsilon\|_{L^\infty(0,T;\/ \mathbb{H}_0^1(\Omega))}\leq C\big\{
\|w^\varepsilon\|_{L^\infty(0,T; \/\mathbb{L}^2(\Omega))}+
\|V_c(\frac{\mathbf{x}}{\varepsilon})-\langle V_c \rangle\|_{H^{-1}(\Omega)}}\\[2mm]
{\displaystyle\quad+\|\mathbf{E}^\varepsilon-\widehat{\mathbf{E}}^0\|_{L^2(0,T;\/ H^{-1}(\Omega))}\big\}
\rightarrow 0,\quad {\rm as}\quad \varepsilon \rightarrow 0,}
\end{array}
\end{equation*}
and consequently
\begin{equation}\label{eq:2-22}
{\displaystyle \|\Psi^\varepsilon-\widehat{\Psi}^\varepsilon\|_{L^\infty(0,T; \/\mathbb{H}_0^1(\Omega))}
\rightarrow 0,\quad {\rm as }\quad \varepsilon\rightarrow 0.}
\end{equation}

We recall $ (\ref{eq:1-6})_5 $, and combining (\ref{eq:2-15}) and (\ref{eq:2-22}) gives
\begin{equation}\label{eq:2-23}
{\displaystyle \mathbf{J}_q^\varepsilon \stackrel{\mathrm{w}}{\rightharpoonup}
\widehat{\mathbf{J}}_q^0,\quad {\rm weakly \,\,in}\quad L^2(0,T; \/(L^2(\Omega))^3),\quad {\rm as}\quad
\varepsilon\rightarrow 0.}
\end{equation}

Using Theorem 4.1 in Chapter 7 of \cite{Duv} and following the lines of the proof of Theorem~4.5 in \cite[p.~666]{Ben}(see also
\cite[p.~125]{San}), we prove
\begin{equation}\label{eq:2-24}
\begin{array}{@{}l@{}}
{\displaystyle \mathbf{E}^\varepsilon \stackrel{\mathrm{w}^\ast}{\rightharpoonup}
\widehat{\mathbf{E}}^0,\quad {\rm weakly \,\, star\,\, in}\quad L^\infty(0,T; \/(L^2(\Omega))^3),\quad {\rm as}\quad
\varepsilon\rightarrow 0,}\\[2mm]
{\displaystyle \mathbf{H}^\varepsilon \stackrel{\mathrm{w}^\ast}{\rightharpoonup}
\widehat{\mathbf{H}}^0,\quad {\rm weakly \,\, star\,\, in}\quad L^\infty(0,T;\/ (L^2(\Omega))^3),\quad {\rm as}\quad
\varepsilon\rightarrow 0.}
\end{array}
\end{equation}

Let us turn to the proof of $ \widehat{\phi}^0=\phi^0(\rho^0) $. Combining (\ref{eq:2-15}) and
(\ref{eq:2-22}), for any fixed $ t\in (0,T) $, we have
\begin{equation*}
{\displaystyle \rho^\varepsilon\equiv N |\Psi^\varepsilon|^2 \stackrel{\mathrm{w}}{\rightharpoonup}
N |\widehat{\Psi}^0|^2\equiv \widehat{\rho}^0,\quad {\rm weakly \,\,\,in}\quad L^2(\Omega),\quad {\rm as}\quad
\varepsilon\rightarrow 0,}
\end{equation*}
and consequently
\begin{equation}\label{eq:2-25}
{\displaystyle \|\rho^\varepsilon-\widehat{\rho}^0\|_{H^{-1}(\Omega)}\rightarrow 0,
\quad {\rm as}\quad
\varepsilon\rightarrow 0.}
\end{equation}

For any fixed $ t\in (0,T) $, let $ \widetilde{\phi}^\varepsilon(\mathbf{x},t) $ be the solution of the following
elliptic equation:
\begin{equation}\label{eq:2-26}
\left\{
\begin{array}{lll}
{\displaystyle
-\frac{\partial}{\partial x_i}\big(\eta_{ij}(\frac{x}{\varepsilon})
\frac{\partial \widetilde{\phi}^\varepsilon(\mathbf{x},t)}{\partial x_j}\big)
=\widehat{\rho}^0(\mathbf{x},t),\quad \mathbf{x}\in \Omega,}\\[2mm]
{\displaystyle \widetilde{\phi}^\varepsilon(\mathbf{x},t)=0,
\quad \mathbf{x}\in \partial\Omega,}
\end{array}
\right.
\end{equation}
where $ t $ plays the role of a parameter. As usual we use the convergence result of the homogenization method
for the linear elliptic equations (see, e.g., Theorem~3.1 of
\cite{Ben} or Theorem~6.1 of \cite{Cio}) and obtain
\begin{equation}\label{eq:2-27}
\begin{array}{@{}l@{}}
{\displaystyle\widetilde{\phi}^\varepsilon\rightarrow \widehat{\phi}^0(\widehat{\rho}^0)\quad \hbox{weakly \,\,in}
\,\,\, H^1_0(\Omega),\,\,\,\hbox{as}\,\,\,
\varepsilon\rightarrow 0,}\\[2mm]
{\displaystyle\widetilde{\phi}^\varepsilon\rightarrow \widehat{\phi}^0(\widehat{\rho}^0)\quad \hbox{strongly \,\,in}
\,\,\, L^2(\Omega),\,\,\,\hbox{as}\,\,\,
\varepsilon\rightarrow 0.}
\end{array}
\end{equation}

Subtracting $(\ref{eq:1-6})_4 $ from (\ref{eq:2-26}) and using {\it a priori}
estimates for elliptic equations, we get
\begin{equation}\label{eq:2-28}
{\displaystyle\|\phi^\varepsilon(\rho^\varepsilon)-\widetilde{\phi}^\varepsilon\|_{H^1_0(\Omega)}
\leq C \|\rho^\varepsilon-\widehat{\rho}^0\|_{H^{-1}(\Omega)}\rightarrow 0,
\quad \hbox{as}\,\,\, \varepsilon\rightarrow 0.}
\end{equation}

For any fixed $ t\in (0,T) $, combining (\ref{eq:2-9}), (\ref{eq:2-25}) and (\ref{eq:2-27}) gives
$\widehat{\phi}^0(\mathbf{x},t)=\widehat{\phi}^0(\widehat{\rho}^0(\mathbf{x},t)) $.
On the other hand, we assume that the problem (\ref{eq:1-6}) without the exchange-correlation potential
and the associated homogenized problem (\ref{eq:2-5}) have the unique weak solutions, respectively.
Consequently, the convergence (\ref{eq:2-15}) takes place for the whole sequences.
Therefore, we get $ \phi^0=\phi^0(n^0) $. From this, we have
$ \widehat{\Psi}^0=\Psi^0 $, $ \widehat{\mathbf{E}}^0=\mathbf{E}^0 $ and $ \widehat{\mathbf{H}}^0=\mathbf{H}^0 $.
Therefore, we complete the proof of Theorem~\ref{thm2-1}.\qquad \end{proof}

\begin{corollary}\label{cor2-1}
If there is the exchange-correlation potential in (\ref{eq:1-6}),
which it is Lipschitz continuous and the corresponding Lipschitz
constant is sufficiently small, and other conditions are the same as Theorem~\ref{thm2-1},
then we can derive the similar
convergence results to those of Theorem~\ref{thm2-1}.
\end{corollary}

In fact, here we consider the following modified Sch\"{o}dinger
equation:
\begin{equation}\label{eq:2-29}
{\displaystyle i\frac{\partial \widehat{\Psi}^\varepsilon(\mathbf{x},t)}{\partial t}
=-\nabla\cdot \big(A(\frac{\mathbf{x}}{\varepsilon})\nabla
\widehat{\Psi}^\varepsilon(\mathbf{x},t)\big)
+\big(\langle V_c\rangle-\widehat{\mathbf{E}}^0\cdot \widehat{\bm \zeta}
+V_{xc}(\widehat{\rho}^\varepsilon)
\big)
\widehat{\Psi}^\varepsilon(\mathbf{x},t),}
\end{equation}
where $ \widehat{\rho}^\varepsilon=N|\widehat{\Psi}^\varepsilon|^2 $.
If the exchange-correlation potential in (\ref{eq:1-6}) is Lipschitz continuous and the corresponding Lipschitz
constant is sufficiently small, then we can prove that the self-consistent iterative method (SCF) is convergent.
Following the lines of the proofs of (\ref{eq:2-13})-(\ref{eq:2-28}), we can complete the proof
of Corollary~\ref{cor2-1}.

\begin{rem}\label{rem2-3}
If there is the generic exchange-correlation potential (see, e.g., \cite[p.~152-169]{Mar}),
then the convergence result of the
homogenization method for the problem
(\ref{eq:1-6}) is not known to authors yet.
\end{rem}

\section{The multiscale asymptotic method and the main convergence theorems}\label{sec-3}

Numerous numerical results have shown that, if $ \varepsilon>0 $ is not sufficiently small, the accuracy
of the homogenization method may not be satisfactory (see, e.g., \cite{Cao-1, Cao-3, Cao-4, Zhang-L, Zhang-1}).
Hence one hopes to seek the multiscale asymptotic methods and the
associated numerical algorithms in the real applications. In this section, we first formally present
the multiscale asymptotic expansions of the solution
of (\ref{eq:1-6}), and then we derive the convergence result of the multiscale method.

\subsection{The multiscale asymptotic expansions}

Let $ {\bm \xi}=\varepsilon^{-1}\mathbf{x} $, for the coefficient matrices $ A({\bm \xi})=(a_{ij}({\bm \xi})) $,
$ \eta({\bm \xi})=(\eta_{ij}({\bm \xi})) $
and $ \mu({\bm \xi})=(\mu_{ij}({\bm \xi})) $,
we will define three sets of cells functions:
$ \theta_k^a({\bm \xi}) $, $\theta_{kl}^a({\bm \xi})$;
$\theta_k^\eta({\bm \xi}) $, $ \theta_{kl}^\eta({\bm \xi}) $, $ \Theta_1^\eta({\bm \xi}) $, $ \Theta_2^\eta({\bm \xi})$;
$\theta_k^\mu({\bm \xi}) $, $\theta_{kl}^\mu({\bm \xi}) $, $\Theta_1^\mu({\bm \xi}) $, $ \Theta_2^\mu({\bm \xi}) $,
$ k, l=1,2,3 $,
where $ \theta_k^a({\bm \xi}) $, $\theta_{kl}^a({\bm \xi})$, $ \theta_k^\eta({\bm \xi}) $,
$\theta_{kl}^\eta({\bm \xi})$, $ \theta_k^\mu({\bm \xi}) $,
$\theta_{kl}^\mu({\bm \xi}) $ are scalar cell functions and
$\Theta_1^\eta({\bm \xi}) $, $ \Theta_2^\eta({\bm \xi})$ , $\Theta_1^\mu({\bm \xi}) $,
$ \Theta_2^\mu({\bm \xi})$ are matrix-valued cell functions defined in the
unit cell $ Q $. The scalar cells functions  $ \theta_k^a({\bm \xi}) $, $\theta_{kl}^a({\bm \xi})$,
$ \theta_k^\eta({\bm \xi}) $, $\theta_{kl}^\eta({\bm \xi})$, $
\theta_k^\mu({\bm \xi})$, $\theta_{kl}^\mu({\bm \xi}) $ are defined in turn
\begin{equation}\label{eq:3-1}
\left\{
\begin{array}{@{}l@{}}
{\displaystyle \frac{\partial}{\partial \xi_i}\big(
a_{ij}({\bm \xi})\frac{\partial \theta_k^a({\bm\xi})}{\partial \xi_j}\big)
=-\frac{\partial a_{ik}({\bm\xi})}{\partial \xi_i},\quad {\bm\xi}\in Q,}\\[2mm]
{\displaystyle\theta_k^a({\bm\xi})=0,\quad {\bm\xi}\in \partial Q,}
\end{array}
\right.
\end{equation}
\begin{equation}\label{eq:3-2}
\left\{
\begin{array}{@{}l@{}}
{\displaystyle \frac{\partial}{\partial \xi_i}\big(
a_{ij}({\bm\xi})\frac{\partial \theta_{kl}^a({\bm\xi})}{\partial \xi_j}\big)
=-\frac{\partial \big(a_{ik}({\bm\xi})\theta_l^a({\bm\xi})\big)}{\partial \xi_i}}\\[3mm]
{\displaystyle \qquad -a_{kj}({\bm\xi})\frac{\partial \theta_l^a({\bm\xi})}
{\partial \xi_j}-a_{kl}({\bm\xi})+\hat{a}_{kl},\quad {\bm\xi}\in Q,}\\
{\displaystyle \theta_{kl}^a({\bm\xi})=0,\quad {\bm\xi}\in \partial Q},
\end{array}
\right.
\end{equation}
\begin{equation}\label{eq:3-3}
\left\{
\begin{array}{@{}l@{}}
{\displaystyle \frac{\partial}{\partial \xi_i}\big(
\eta_{ij}({\bm\xi})\frac{\partial \theta_k^\eta({\bm\xi})}{\partial \xi_j}\big)
=-\frac{\partial \eta_{ik}({\bm\xi})}{\partial \xi_i},\quad {\bm\xi}\in Q,}\\[2mm]
\theta_k^\eta({\bm\xi})=0,\quad {\bm\xi}\in \partial Q,
\end{array}
\right.
\end{equation}
\begin{equation}\label{eq:3-4}
\left\{
\begin{array}{@{}l@{}}
{\displaystyle \frac{\partial}{\partial \xi_i}\big(
\eta_{ij}({\bm\xi})\frac{\partial \theta_{kl}^\eta({\bm\xi})}{\partial \xi_j}\big)
=-\frac{\partial \big(\eta_{ik}({\bm\xi})\theta_l^\eta({\bm\xi})\big)}{\partial \xi_i}}\\[3mm]
{\displaystyle \qquad -\eta_{kj}({\bm\xi})\frac{\partial \theta_l^\eta({\bm\xi})}
{\partial \xi_j}-\eta_{kl}({\bm\xi})+\hat{\eta}_{kl},\quad {\bm\xi}\in Q,}\\
{\displaystyle \theta_{kl}^\eta({\bm\xi})=0,\quad {\bm\xi}\in \partial Q},
\end{array}
\right.
\end{equation}
\begin{equation}\label{eq:3-5}
\left\{
\begin{array}{@{}l@{}}
{\displaystyle \frac{\partial}{\partial \xi_i}\big(
\mu_{ij}({\bm\xi})\frac{\partial \theta_k^\mu({\bm\xi})}{\partial \xi_j}\big)
=-\frac{\partial \mu_{ik}({\bm\xi})}{\partial \xi_i},\quad {\bm\xi}\in Q,}\\[2mm]
{\displaystyle \theta_k^\mu({\bm\xi})=0,\quad {\bm\xi}\in \partial Q,}
\end{array}
\right.
\end{equation}
and
\begin{equation}\label{eq:3-6}
\left\{
\begin{array}{@{}l@{}}
{\displaystyle \frac{\partial}{\partial \xi_i}\big(
\mu_{ij}({\bm\xi})\frac{\partial \theta_{kl}^\mu({\bm\xi})}{\partial \xi_j}\big)
=-\frac{\partial \big(\mu_{ik}({\bm\xi})\theta_l^\mu({\bm\xi})\big)}{\partial \xi_i}}\\
{\displaystyle \qquad -\mu_{kj}({\bm\xi})\frac{\partial \theta_l^\mu({\bm\xi})}
{\partial \xi_j}-\mu_{kl}({\bm\xi})+\hat{\mu}_{kl},\quad {\bm\xi}\in Q,}\\[3mm]
{\displaystyle \theta_{kl}^\mu({\bm\xi})=0,\quad {\bm\xi}\in \partial Q},
\end{array}
\right.
\end{equation}
where the homogenized coefficient matrices $ \widehat{A}=(\hat{a}_{kl}) $,
$ \widehat{\eta}=(\hat{\eta}_{kl}) $ and
$ \widehat{\mu}=(\hat{\mu}_{kl}) $ are similarly given in (\ref{eq:2-4}).

\begin{rem}\label{rem3-1}
Under the assumptions $({\rm A}_1)$--$({\rm A}_3)$, the existence and uniqueness of
the solutions for the cell problems (\ref{eq:3-1})-(\ref{eq:3-6}) can be established
based upon Lax-Milgram lemma. It should be mentioned the problems (\ref{eq:3-1})-(\ref{eq:3-6})
require the homogeneous Dirichlet's boundary conditions instead of the usual
periodic boundary conditions.
\end{rem}

Next we give the definitions of the matrix-valued cell functions $
\mathbf{\Theta}_1^\eta({\bm\xi})$, $ \mathbf{\Theta}_2^\eta({\bm\xi})$, $
\mathbf{\Theta}_1^\mu({\bm\xi})$ and $ \mathbf{\Theta}_2^\mu({\bm\xi})$.
Let $ \eta^{-1}({\bm\xi}) $ and $ \mu^{-1}({\bm\xi}) $ denote the inverse matrices of
$ \eta({\bm\xi}) $ and $ \mu({\bm\xi}) $, respectively.
We define $ \mathbf{\Theta}_{1,p}^\eta({\bm\xi})$, $
\mathbf{\Theta}_{1,p}^\mu({\bm\xi})$, $ p=1,2,3 $ in the following ways:
\begin{equation}\label{eq:3-7}
\left\{
\begin{array}{l@{}}
{\displaystyle\mathbf{curl}_{{\bm\xi}}(\eta^{-1}({\bm\xi})
\mathbf{curl}_{{\bm\xi}}\Theta^{\eta}_{1,p}({\bm\xi}))=
-\mathbf{curl}_{{\bm\xi}}(\eta^{-1}({\bm\xi})\mathbf{e}_{p}),\quad {\bm\xi} \in Q,}\\[2mm]
{\displaystyle\nabla_{{\bm\xi}}\cdot\Theta^{\eta}_{1,p}({\bm\xi})=0,\quad {\bm\xi} \in Q,}\\[2mm]
{\displaystyle\Theta^{\eta}_{1,p}({\bm\xi})\times {\bm\nu}=0,\quad {\bm\xi}
\in\partial Q,\quad p=1,2,3,}
\end{array}
\right.
\end{equation}
\begin{equation}\label{eq:3-8}
\left\{
\begin{array}{l@{}}
{\displaystyle\mathbf{curl}_{{\bm\xi}}(\mu^{-1}({\bm\xi})
\mathbf{curl}_{{\bm\xi}}\Theta^{\mu}_{1,p}({\bm\xi}))=
-\mathbf{curl}_{{\bm\xi}}(\mu^{-1}({\bm\xi})\mathbf{e}_{p}),\quad {\bm\xi} \in Q,}\\[2mm]
{\displaystyle\nabla_{{\bm\xi}}\cdot\Theta^{\mu}_{1,p}({\bm\xi})=0,\quad {\bm\xi} \in Q,}\\[2mm]
{\displaystyle\Theta^{\mu}_{1,p}({\bm\xi})\times {\bm\nu}=0,\quad {\bm\xi}
\in\partial Q,\quad p=1,2,3,}
\end{array}
\right.
\end{equation}
where $ \Theta_{1,p}^{\eta}({\bm\xi})$ and $ \Theta_{1,p}^{\mu}({\bm\xi})$,
$p=1,2,3 $ are the vector-valued functions,
${\bm\nu}=(\nu_1,\nu_2,\nu_3) $ is the outward unit normal to $ \partial Q $,
 $ \mathbf{e}_1=\{1,0,0\}^T$,
$\mathbf{e}_2=\{0,1,0\}^T$, $\mathbf{e}_3=\{0,0,1\}^T $, $
\mathbf{a}^T $ denotes the transpose of a vector $ \mathbf{a} $. Let
\begin{equation*}
{\displaystyle\Theta_{1}^{\eta}(\mathbf{\xi})=(\Theta^{\eta}_{1,1}(\mathbf{\xi}),
\Theta^{\eta}_{1,2}(\mathbf{\xi}),\Theta^{\eta}_{1,3}(\mathbf{\xi})),\quad
\Theta_{1}^{\mu}(\mathbf{\xi})=(\Theta^{\mu}_{1,1}(\mathbf{\xi}),
\Theta^{\mu}_{1,2}(\mathbf{\xi}),\Theta^{\mu}_{1,3}(\mathbf{\xi})).}
\end{equation*}

\begin{rem}\label{rem3-2}
The definitions of\/ $\Theta_{1,p}^\eta({\bm\xi})$,
$\Theta_{1,p}^\mu({\bm\xi})$, $ p=1,2,3 $
in {\rm(\ref{eq:3-7})} and {\rm(\ref{eq:3-8})} are similar to {\rm(4.128)} of
{\rm\cite[p.~663]{Ben}}.
However, the essential difference is that we take a perfect
conductor boundary condition instead of the periodic boundary
condition of\/~{\rm\cite{Ben}}. Similarly to\/ {\rm(4.128)} of {\rm\cite[p.~663]{Ben}},
under the assumptions $({\rm A}_1)$--$({\rm A}_3)$, the existence and uniqueness of
 problems {\rm(\ref{eq:3-7})} and {\rm(\ref{eq:3-8})} can be established based upon Lax-Milgram lemma.
\end{rem}

Following the idea of \cite{Cao-1}, we define the second-order vector-valued cell functions
$ \Theta_{2,p}^{\eta}({\bm\xi})$ and $ \Theta_{2,p}^{\mu}({\bm\xi})$ as follows:
\begin{equation}\label{eq:3-9}
\left\{
\begin{array}{l@{}}
{\displaystyle\mathbf{curl}_{{\bm\xi}}(\eta^{-1}({\bm\xi})\mathbf{curl}_{{\bm\xi}}
\Theta^{\eta}_{2,p}({\bm\xi}))=
-\mathbf{curl}_{{\bm\xi}}(\eta^{-1}({\bm\xi})\Theta^{\eta}_{1,p}({\bm\xi}))}\\[2mm]
{\displaystyle\quad-
\eta^{-1}({\bm\xi})\mathbf{curl}_{{\bm\xi}}\Theta^{\eta}_{1,p}({\bm\xi})
-\eta^{-1}({\bm\xi})\mathbf{e}_{p}+
\widehat{\eta}^{-1}\mathbf{e}_{p}+\nabla_{{\bm\xi}} \zeta_{2,p}^{\eta}({\bm\xi}),\quad {\bm\xi} \in Q,}\\[2mm]
{\displaystyle\nabla_{{\bm\xi}}\cdot\Theta^{\eta}_{2,p}({\bm\xi})=0,\quad {\bm\xi} \in Q,}\\[2mm]
{\displaystyle\Theta^{\eta}_{2,p}({\bm\xi})\times {\bm\nu}=0,\quad {\bm\xi} \in
\partial Q, \quad p=1,2,3,}
\end{array}
\right.
\end{equation}
and
\begin{equation}\label{eq:3-10}
\left\{
\begin{array}{l@{}}
{\displaystyle\mathbf{curl}_{{\bm\xi}}(\mu^{-1}({\bm\xi})\mathbf{curl}_{{\bm\xi}}
\Theta^{\mu}_{2,p}({\bm\xi}))=
-\mathbf{curl}_{{\bm\xi}}(\mu^{-1}({\bm\xi})\Theta^{\mu}_{1,p}({\bm\xi}))}\\[2mm]
{\displaystyle\quad-
\mu^{-1}({\bm\xi})\mathbf{curl}_{{\bm\xi}}\Theta^{\mu}_{1,p}({\bm\xi})
-\mu^{-1}({\bm\xi})\mathbf{e}_{p}+
\widehat{\mu}^{-1}\mathbf{e}_{p}+\nabla_{{\bm\xi}} \zeta_{2,p}^{\mu}({\bm\xi}),\quad {\bm\xi} \in Q,}\\[2mm]
{\displaystyle\nabla_{{\bm\xi}}\cdot\Theta^{\mu}_{2,p}({\bm\xi})=0,\quad {\bm\xi} \in Q,}\\[2mm]
{\displaystyle\Theta^{\mu}_{2,p}({\bm\xi})\times {\bm\nu}=0,\quad {\bm\xi} \in
\partial Q, \quad p=1,2,3.}
\end{array}
\right.
\end{equation}

By using (11.46) of \cite[p.~145]{Ben}, the homogenized coefficient
matrices $ \widehat{\eta}^{-1} $ and $ \widehat{\mu}^{-1} $ are calculated by
\begin{equation}\label{eq:3-11}
\begin{array}{l@{}}
{\displaystyle \widehat{\eta}^{-1}=\mathcal{M}\Bigl(\eta^{-1}({\bm\xi})
 +\eta^{-1}({\bm\xi})\,\mathbf{curl}_{{\bm\xi}}\,
\Theta_1^{\eta}({\bm\xi})\Bigr),}\\[2mm]
{\displaystyle\widehat{\mu}^{-1}=\mathcal{M}\Bigl(\mu^{-1}({\bm\xi})
 +\mu^{-1}({\bm\xi})\,\mathbf{curl}_{{\bm\xi}}\,
\Theta_1^{\mu}({\bm\xi})\Bigr),}
\end{array}
\end{equation}
where the matrix-valued cell functions $\Theta_1^\eta({\bm\xi})=(\Theta^{\eta}_{1,1}({\bm\xi}) $, $\Theta^{\eta}_{1,2}({\bm\xi}) $, $
\Theta^{\eta}_{1,3}({\bm\xi}))$
and
$\Theta_1^\mu({\bm\xi})=(\Theta^{\mu}_{1,1}({\bm\xi}),\Theta^{\mu}_{1,2}({\bm\xi}),
\Theta^{\mu}_{1,3}({\bm\xi}))$
are defined in (\ref{eq:3-7}) and (\ref{eq:3-8}), respectively,
$ \mathcal{M}v=\int_Q v({\bm\xi}) \mathnormal{d}{\bm\xi} $.
The functions $ \zeta_{2,p}^{\eta}({\bm\xi}) $ and $ \zeta_{2,p}^{\mu}({\bm\xi}) $,
$p=1,2,3 $ in (\ref{eq:3-9}) and (\ref{eq:3-10}) are respectively the solutions of the following elliptic
equations:
\begin{equation}\label{eq:3-12}
\left\{
\begin{array}{l@{}}
{\displaystyle-\Delta_{{\bm\xi}}
\zeta_{2,p}^{\eta}({\bm\xi})=\nabla_{{\bm\xi}}\cdot
\widetilde{G}^\eta({\bm\xi}), \quad {\bm\xi}\in Q,}\\[2mm]
{\displaystyle\zeta_{2,p}^{\eta}({\bm\xi})=0, \quad {\bm\xi}\in \partial Q,}
\end{array}
\right.
\end{equation}
and
\begin{equation}\label{eq:3-13}
\left\{
\begin{array}{l@{}}
{\displaystyle-\Delta_{{\bm\xi}}
\zeta_{2,p}^{\mu}({\bm\xi})=\nabla_{{\bm\xi}}\cdot
\widetilde{G}^\mu({\bm\xi}), \quad {\bm\xi}\in Q,}\\[2mm]
{\displaystyle\zeta_{2,p}^{\mu}({\bm\xi})=0, \quad {\bm\xi}\in \partial Q,}
\end{array}
\right.
\end{equation}
where $ \nabla_{{\bm\xi}}\cdot=div_{{\bm\xi}} $, and
\begin{equation}\label{eq:3-14}
\begin{array}{l@{}}
{\displaystyle\widetilde{G}^\eta({\bm\xi})=-\eta^{-1}({\bm\xi})\mathbf{curl}_{{\bm\xi}}\,\Theta^{\eta}_{1,p}(\xi)
-\eta^{-1}({\bm\xi})\mathbf{e}_{p}+\widehat{\eta}^{-1}\mathbf{e}_{p},}\\[2mm]
{\displaystyle\widetilde{G}^\mu({\bm\xi})=-\mu^{-1}({\bm\xi})\mathbf{curl}_{{\bm\xi}}\,\Theta^{\mu}_{1,p}({\bm\xi})
-\mu^{-1}({\bm\xi})\mathbf{e}_{p}+\widehat{\mu}^{-1}\mathbf{e}_{p}.}
\end{array}
\end{equation}
It can be verified that
\begin{equation}\label{eq:3-15}
{\displaystyle\nabla_{{\bm\xi}}\cdot(\widetilde{G}^\eta({\bm\xi})+\nabla_{{\bm\xi}}
\zeta_{2,p}^{\eta}({\bm\xi}))=0,\quad
\nabla_{{\bm\xi}}\cdot(\widetilde{G}^\mu({\bm\xi})+\nabla_{{\bm\xi}}
\zeta_{2,p}^{\mu}({\bm\xi}))=0,}
\end{equation}
and
\begin{equation}\label{eq:3-16}
{\displaystyle\zeta_{2,p}^{\eta},\quad \zeta_{2,p}^{\mu}\in H^2(Q)\cap H_0^1(Q).}
\end{equation}

Let the matrix-valued functions
$\Theta_{2}^\eta({\bm\xi})=(\Theta^{\eta}_{2,1}({\bm\xi}),\Theta^{\eta}_{2,2}
({\bm\xi}),\Theta^{\eta}_{2,3}({\bm\xi}))$ and
$\Theta_{2}^\mu({\bm\xi})=(\Theta^{\mu}_{2,1}({\bm\xi}),\Theta^{\mu}_{2,2}
({\bm\xi}),\Theta^{\mu}_{2,3}({\bm\xi}))$. Hence, we define the first-order and the second-order multiscale
asymptotic expansions of the solution for the problem
(\ref{eq:1-6}) as follows:
\begin{equation}\label{eq:3-17}
\begin{array}{l@{}}
{\displaystyle \Psi^{\varepsilon}_{1}(\mathbf{x},t)=\Psi^0(\mathbf{x},t)+\varepsilon \theta_k^a({\bm \xi})
\frac{\partial \Psi^0(\mathbf{x},t)}{\partial x_k},}\\[2mm]
{\displaystyle \Psi^{\varepsilon}_{2}(\mathbf{x},t)=\Psi^0(\mathbf{x},t)+\varepsilon
\theta_k^a({\bm \xi})
\frac{\partial \Psi^0(\mathbf{x},t)}{\partial x_k}+\varepsilon^2
\theta_{kl}^a({\bm \xi})
\frac{\partial^2 \Psi^0(\mathbf{x},t)}{\partial x_k \partial x_l},}
\end{array}
\end{equation}
\begin{equation}\label{eq:3-18}
\begin{array}{l@{}}
{\displaystyle \mathbf{E}^{\varepsilon,(1)}(\mathbf{x},t)=\mathbf{E}^{0}(\mathbf{x},t)
+\varepsilon \nabla\big(\theta_k^\eta({\bm\xi})E_{k}^0(\mathbf{x},t)\big)
-\varepsilon\mathbf{\Theta}_1^\mu({\bm\xi})
\widehat{\mu}\frac{\partial \mathbf{H}^0(\mathbf{x},t)}{\partial t},}\\[2mm]
{\displaystyle \mathbf{E}^{\varepsilon, (2)}(\mathbf{x},t)=\mathbf{E}^{0}(\mathbf{x},t)
+\varepsilon \nabla\Big(\theta_k^\eta({\bm\xi})E_{k}^0(\mathbf{x},t)+
\varepsilon \theta_{kl}^\eta({\bm\xi})\frac{\partial E_{k}^0(\mathbf{x},t)}
{\partial x_l}\Big)}\\[2mm]
{\displaystyle \quad -\varepsilon\mathbf{\Theta}_1^\mu({\bm\xi})
\widehat{\mu}\frac{\partial \mathbf{H}^0(\mathbf{x},t)}{\partial t}
-\varepsilon^2 \mathbf{\Theta}_2^\mu({\bm\xi})
\mathbf{curl_x}\big(\widehat{\mu}\frac{\partial \mathbf{H}^0(\mathbf{x},t)}{\partial t}\big),}
\end{array}
\end{equation}
\begin{equation}\label{eq:3-19}
\begin{array}{l@{}}
{\displaystyle \mathbf{H}^{\varepsilon, (1)}(\mathbf{x},t)=\mathbf{H}^{0}(\mathbf{x},t)
+\varepsilon \nabla\big(\theta_k^\mu({\bm\xi})H_{k}^0(\mathbf{x},t)\big)
+\varepsilon\mathbf{\Theta}_1^\eta({\bm\xi})
\widehat{\eta}\frac{\partial \mathbf{E}^0(\mathbf{x},t)}{\partial t},}\\[2mm]
{\displaystyle \mathbf{H}^{\varepsilon, (2)}(\mathbf{x},t)=\mathbf{H}^{0}(\mathbf{x},t)
+\varepsilon \nabla\Big(\theta_k^\mu({\bm\xi})H_{k}^0(\mathbf{x},t)+
\varepsilon \theta_{kl}^\mu({\bm\xi})\frac{\partial H_{k}^0(\mathbf{x},t)}
{\partial x_l}\Big)}\\[2mm]
{\displaystyle \quad +\varepsilon\mathbf{\Theta}_1^\eta({\bm\xi})
\widehat{\eta}\frac{\partial \mathbf{E}^0(\mathbf{x},t)}{\partial t}
+\varepsilon^2 \mathbf{\Theta}_2^\eta({\bm\xi})
\mathbf{curl_x}\big(\widehat{\eta}\frac{\partial \mathbf{E}^0(\mathbf{x},t)}{\partial t}\big),}
\end{array}
\end{equation}
where  $
(\Psi^0(\mathbf{x},t), \mathbf{E}^0(\mathbf{x},t), \mathbf{H}^0(\mathbf{x},t)) $ is the
solution of the homogenized problem
{\rm(\ref{eq:2-5})}, the scalar cells
functions $ \theta_k^a({\bm\xi}) $, $\theta_{kl}^a({\bm\xi})$,
$ \theta_k^\eta({\bm\xi}) $, $\theta_{kl}^\eta({\bm\xi})$, $
\theta_k^\mu({\bm\xi}) $ and $\theta_{kl}^\mu({\bm\xi}) $ are defined in
{\rm(\ref{eq:3-1})}-{\rm(\ref{eq:3-6})}, respectively; the
matrix-valued cell functions $ \mathbf{\Theta}_1^\eta({\bm\xi})$, $
\mathbf{\Theta}_2^\eta({\bm\xi})$, $ \mathbf{\Theta}_1^\mu({\bm\xi})$ and $
\mathbf{\Theta}_2^\mu({\bm\xi})$ have been defined in
{\rm(\ref{eq:3-7})}-{\rm(\ref{eq:3-8})},
{\rm(\ref{eq:3-9})}-{\rm(\ref{eq:3-10})}, respectively.
The homogenized coefficient matrices $ \widehat{A}=(\hat{a}_{ij}) $, $ \widehat{\eta}=(\hat{\eta}_{ij}) $
and $ \widehat{\mu}=(\hat{\mu}_{ij}) $ have been given
in (\ref{eq:2-4}).

In order to derive the convergence results for the multiscale
asymptotic expansions {\rm(\ref{eq:3-17})}-{\rm(\ref{eq:3-19})}, we
need to impose the conditions on the coefficient matrices
$ (a_{ij}({\bm \xi})) $,
$(\eta_{ij}({\bm\xi}))$ and $(\mu_{ij}({\bm\xi}))$.

$({\rm\bf H}_1)$ $ A({\bm \xi}) $, $\eta(\mathbf{\xi}) $ and  $\mu(\mathbf{\xi}) $ are all diagonal matrices,
 i.e.
 \begin{equation*}
 \begin{array}{l@{}}
{\displaystyle   A({\bm\xi})={\rm diag}(a_{11}({\bm\xi}), a_{22}({\bm\xi}), a_{33}({\bm\xi})), \quad
\eta({\bm\xi})={\rm diag}(\eta_{11}({\bm\xi}),\eta_{22}({\bm\xi}),\eta_{33}({\bm\xi})),}\\[2mm]
{\displaystyle \mu({\bm\xi})={\rm diag}(\mu_{11}({\bm\xi}),\mu_{22}({\bm\xi}),\mu_{33}({\bm\xi})).}
\end{array}
\end{equation*}

$({\rm \bf H}_2)$  $a_{kk}({\bm\xi})$, $\eta_{kk}({\bm\xi})$,  $\mu_{kk}({\bm\xi})$, $ k=1,2,3,$ are symmetric
with respect to the middleplane $\Delta_{k}$ of $ Q=(0,1)^{3} $,
where $ \Delta_k$,  $k=1,2 $, are illustrated in Figure \ref{f2}(a) in the
two dimensional case.

\begin{figure}[t!]
\centering
{\tiny(a)}\includegraphics[width=5.5cm,height=5.5cm]{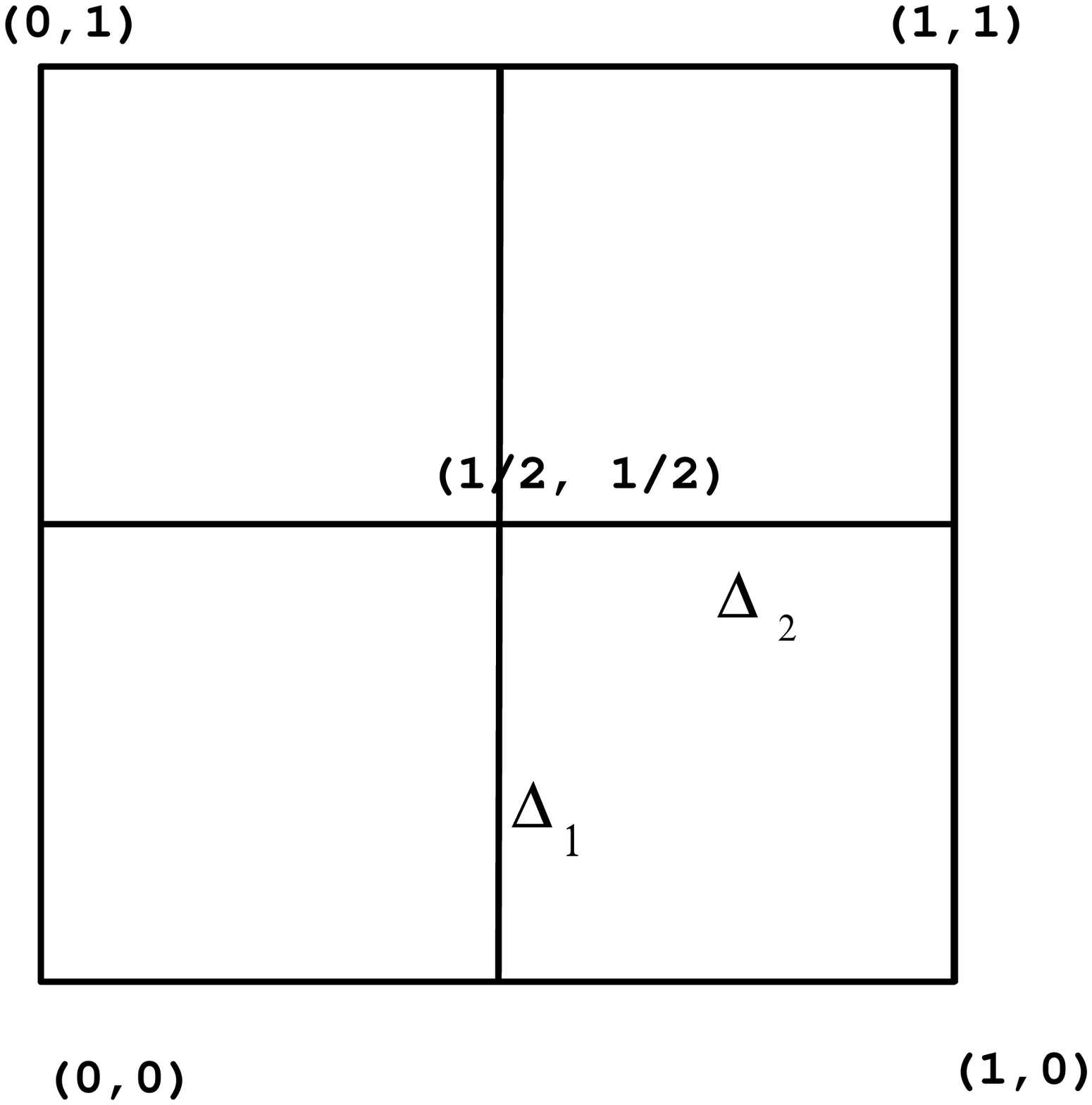}~{\tiny(b)}
\includegraphics[width=5.5cm,height=5.5cm]{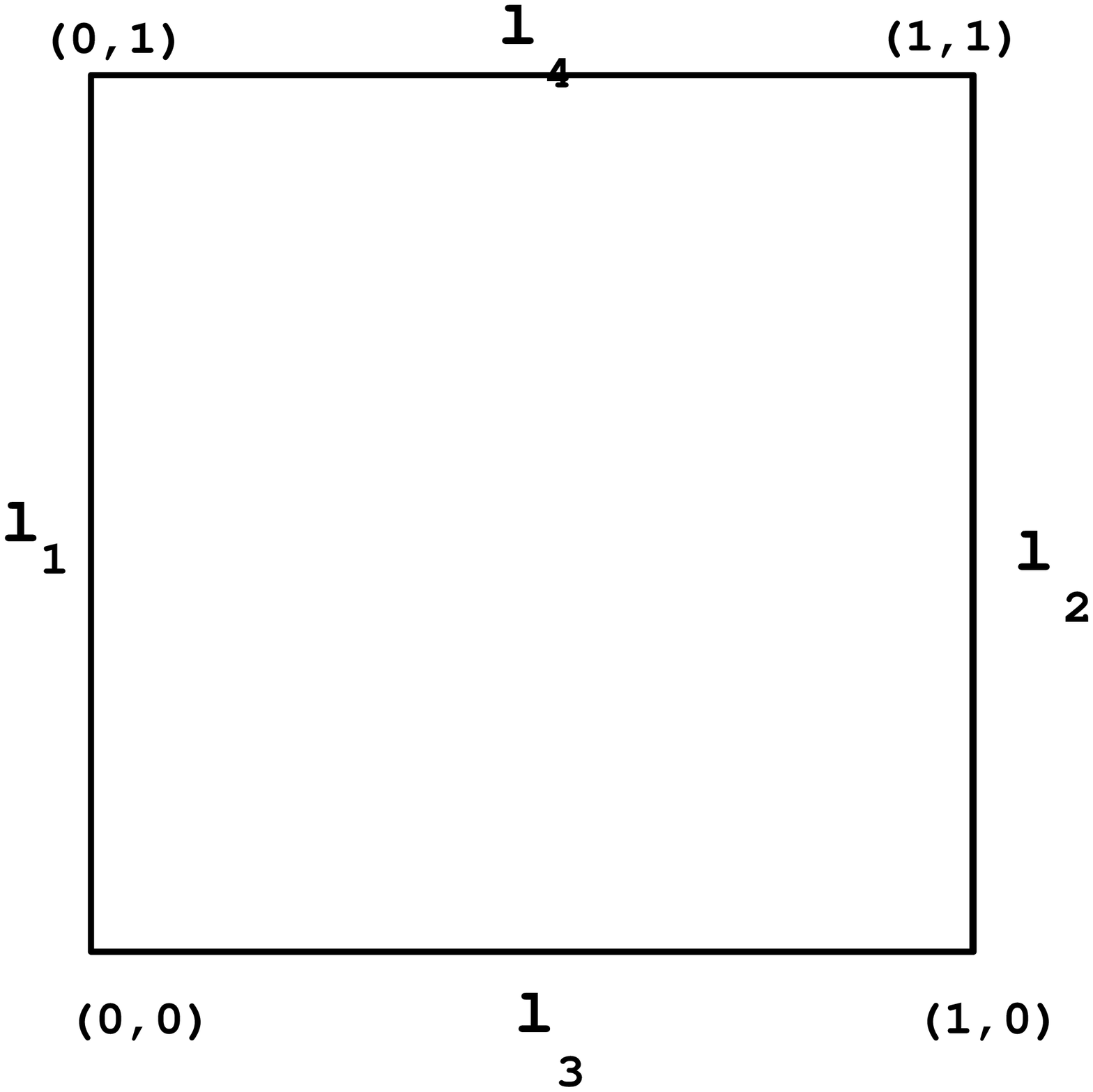}
\caption{{\rm(a)}~The symmetry of $ Q $.\/  {\rm(b)}~The sides of~$Q$.}\label{f2}
\end{figure}

\begin{rem}\label{rem3-3}
The condition~$ ({\rm H}_2) $ indicates that composite
materials satisfy geometric symmetric properties in a periodic
microstructure.
\end{rem}

\begin{lemma}\label{lem3-1}
(see Proposition~2.5 of \cite{Cao-3}, see also \cite{Cao-1})
Let the scalar cells functions $ \theta_k^a({\bm \xi}) $, $\theta_{kl}^a({\bm \xi})$,
$ \theta_k^\eta({\bm\xi}) $, $\theta_{kl}^\eta({\bm\xi})$, $
\theta_k^\mu({\bm\xi}) $ and $\theta_{kl}^\mu({\bm\xi}) $ be the solutions of
the cell problems {\rm(\ref{eq:3-1})}-{\rm(\ref{eq:3-6})}, respectively.
Under the assumptions {\rm$({\rm A}_1)$--$({\rm A}_3)$} and
{\rm$({\rm H}_1)$--$({\rm H}_2)$}, one can prove that
the normal derivatives $ \sigma_{\bm\xi}^a(\theta_k^a) $,
$ \sigma_{\bm\xi}^a(\theta_{kl}^a) $, $ \sigma_{\bm\xi}^\eta(\theta_k^\eta) $,
$ \sigma_{\bm\xi}^\eta(\theta_{kl}^\eta) $, $ \sigma_{\bm\xi}^\mu(\theta_k^\mu) $,
and $ \sigma_{\bm\xi}^\mu(\theta_{kl}^\mu) $, $ k, l=1,2,3 $ are continuous
on the boundary $ \partial Q $ of the reference cell $ Q $. Note that
$ \sigma_{\bm\xi}^a(u)\equiv\nu_i a_{ij}\frac{\displaystyle \partial u}
{\displaystyle \partial \xi_j} $,
$ \sigma_{\bm\xi}^\eta(u)\equiv\nu_i \eta_{ij}\frac{\displaystyle \partial u}
{\displaystyle \partial \xi_j} $ and
$ \sigma_{\bm\xi}^\mu(v)\equiv\nu_i \mu_{ij}\frac{\displaystyle \partial v}
{\displaystyle \partial \xi_j} $ , where $ {\bm\nu}=(\nu_1,\cdots, \nu_n) $ is
the outward unit normal to $ \partial Q $.
\end{lemma}

\begin{lemma}\label{lem3-2}
(See Proposition 2.1 of \cite{Cao-3})
Let\/ $ {\bm\Theta}_{1,p}^{\eta}({\bm\xi})$ and $
{\bm\Theta}_{1,p}^{\mu}({\bm\xi})$, $p=1,2,3 $, be the
solutions of the cell problems\/ {\rm(\ref{eq:3-7})} and
{\rm(\ref{eq:3-8})}, respectively. Under the assumptions\/  {\rm$({\rm
A}_1)$--$({\rm A}_3)$},\/ {\rm$({\rm H}_1)$--$({\rm H}_2)$}, it can
be proved that
\begin{equation}\label{eq:3-20}
{\displaystyle[\eta^{-1}({\bm\xi})\mathbf{curl_{\xi}}\Theta_{1,p}^{\eta}({\bm\xi})\times
{\bm\nu}]|_{\partial Q}=0,\quad
[\mu^{-1}({\bm\xi})\mathbf{curl_{\xi}}\Theta_{1,p}^{\mu}({\bm\xi})\times
{\bm\nu}]|_{\partial Q}=0,}
\end{equation}
where $ {\bm\nu}=(\nu_1,\nu_2,\nu_3) $ is the outward unit
normal on the boundary $ \partial Q $ of the reference cell
$ Q=(0,1)^3 $;\/ $ [v]|_{\partial Q} $ denotes the jump of a function $ v $ on $ \partial Q $.
\end{lemma}

\begin{lemma}\label{lem3-3}
(See Proposition 2.2 of \cite{Cao-3})
Let\/ $ {\bm\Theta}_{2,p}^{\eta}({\bm\xi})$ and $
{\bm\Theta}_{2,p}^{\mu}({\bm\xi})$, $p=1,2,3 $, be the
solutions of the cell problems\/ {\rm(\ref{eq:3-9})} and
{\rm(\ref{eq:3-10})}, respectively. Under the assumptions\/  {\rm$({\rm
A}_1)$--$({\rm A}_3)$},\/ {\rm$({\rm H}_1)$--$({\rm H}_2)$}, it can
be proved that
\begin{equation}\label{eq:3-21}
{\displaystyle[\eta^{-1}({\bm\xi})\mathbf{curl_{\xi}}\Theta_{2,p}^{\eta}({\bm\xi})\times
{\bm\nu}]|_{\partial Q}=0,\quad
[\mu^{-1}({\bm\xi})\mathbf{curl_{\xi}}\Theta_{2,p}^{\mu}({\bm\xi})\times
{\bm\nu}]|_{\partial Q}=0.}
\end{equation}
\end{lemma}

Next we give the main convergence theorems of the multiscale asymptotic method.
\begin{theorem}\label{thm3-1}
Suppose that\/ $ \Omega\subset \mathbb{R}^3 $ is the union of entire
periodic cells, i.e. $\overline{\Omega}=\bigcup_{\mathbf{z}\in
I_\varepsilon}\varepsilon(\mathbf{z}+\overline{Q}) $, where the index set
$I_\varepsilon=\{\mathbf{z}\in \mathbb{Z}^3,${\rm~}$\varepsilon(\mathbf{z}+\overline{Q})\subset
\overline{\Omega}\} $ and $ \varepsilon>0 $ is any fixed small parameter. Let\/
$(\Psi^\varepsilon(\mathbf{x},t), \mathbf{E}^{\varepsilon}(\mathbf{x},t),
\mathbf{H}^{\varepsilon}(\mathbf{x},t))$ be the solution of the
original problem {\rm(\ref{eq:1-6})}, and let\/
$(\Psi^\varepsilon_1(\mathbf{x},t), \mathbf{E}^{\varepsilon, (1)}(\mathbf{x},t),
\mathbf{H}^{\varepsilon, (1)}(\mathbf{x},t))$ and
$(\Psi^\varepsilon_2(\mathbf{x},t), \mathbf{E}^{\varepsilon, (2)}(\mathbf{x},t),
\mathbf{H}^{\varepsilon, (2)}(\mathbf{x},t))$ be the first-order
and the second-order multiscale asymptotic solutions defined in\/
{\rm(\ref{eq:3-17})}-{\rm(\ref{eq:3-19})}, respectively. Under the
assumptions\/ {\rm$({\rm A}_1)$--$({\rm A}_4)$} and\/ {\rm$({\rm
H}_1)$--$({\rm H}_2)$}, if\/
 $(\Psi^0, \mathbf{E}^0, \mathbf{H}^0)\in L^2(0,T; \/\mathbb{H}^3(\Omega)\times (H^{3}(\Omega))^6)\cap
 H^1(0,T;  $\\
 $\mathbb{H}^2(\Omega)\times (H^2(\Omega))^6) $, $ f\in H^1(0,T; (H^1(\Omega))^3) $,
 $ ({\bm\varphi}, {\bm\psi})\in (H^3(\Omega))^6 $,
 $ T<\infty $ and arbitrary, then we have
\begin{equation}\label{eq:3-22}
\begin{array}{l@{}}
{\displaystyle \|\Psi^\varepsilon-\Psi_s^\varepsilon\|_{L^2(0,T;\/ \mathbb{H}_0^1(\Omega))}
+\|\mathbf{E}^\varepsilon-\mathbf{E}^{\varepsilon, (s)}\|_{L^2(0,T;\/ (L^2(\Omega))^3)}}\\[2mm]
{\displaystyle\quad+\|\mathbf{H}^\varepsilon-\mathbf{H}^{\varepsilon, (s)}\|_{L^2(0,T; \/(L^2(\Omega))^3)}
\rightarrow 0,\quad {\rm as}\quad \varepsilon\rightarrow 0,\quad s=1,2.}
\end{array}
\end{equation}
\end{theorem}

\begin{proof}
Due to space limitations, here we only prove Theorem~\ref{thm3-1} for the case $ s=1 $.
The case $ s=2 $ is similar.

Thanks to Lemma~\ref{lem3-1}, from (\ref{eq:1-6}), (\ref{eq:2-5}),
(\ref{eq:3-1})-(\ref{eq:3-6}), (\ref{eq:3-17})-(\ref{eq:3-19}), by a tedious computation,
we get the following equality which holds in the sense of distributions:
\begin{equation}\label{eq:3-23}
\begin{array}{l@{}}
{\displaystyle -i\frac{\partial (\Psi^\varepsilon-\Psi_1^\varepsilon)}{\partial t}
-\nabla\cdot \big(A(\frac{\mathbf{x}}{\varepsilon})\nabla (\Psi^\varepsilon-\Psi_1^\varepsilon)\big)
+\big(V_c(\frac{\mathbf{x}}{\varepsilon})-\mathbf{E}^\varepsilon\cdot \widehat{\bm \zeta}\big)
(\Psi^\varepsilon-\Psi_1^\varepsilon)}\\[2mm]
{\displaystyle\quad =\big((\mathbf{E}^\varepsilon-\mathbf{E}^{\varepsilon,(1)})\cdot \widehat{\bm \zeta}\big)\Psi_1^\varepsilon
+\mathcal{F}^\varepsilon_1,}
\end{array}
\end{equation}
where
\begin{equation}\label{eq:3-24}
\begin{array}{l@{}}
{\displaystyle \mathcal{F}^\varepsilon_1=i\varepsilon \theta_k^a\frac{\partial^2 \Psi^0}{\partial t \partial x_k}
-\Big[\hat{a}_{lk}-a_{lk}({\bm \xi})-a_{lj}({\bm \xi})\frac{\partial \theta_k^a({\bm\xi})}{\partial \xi_j}
-\frac{\partial (a_{lj}({\bm \xi})\theta_k^a)}{\partial \xi_j}\Big]\frac{\partial^2 \Psi^0}{\partial x_l \partial x_k}}\\[2mm]
{\displaystyle\quad + \varepsilon a_{lj}({\bm \xi})\theta_k^a\frac{\partial^3 \Psi^0}{\partial x_l \partial j \partial x_k}
-\big(V_c({\bm\xi})-\langle V_c\rangle \big)\Psi^0-\varepsilon \big(V_c({\bm\xi})-\langle V_c\rangle \big)
\theta_k^a\frac{\partial \Psi^0}{\partial x_k}}\\[2mm]
{\displaystyle\quad+\varepsilon (\mathbf{E}^0\cdot \widehat{\bm \zeta})\theta_k^a\frac{\partial \Psi^0}{\partial x_k}
 +\varepsilon\Big(\nabla(\theta_k^\eta E_k^0)\cdot \widehat{\bm \zeta}-
\big(\Theta_1^\mu\widehat{\mu}\frac{\partial \mathbf{H}^0}{\partial t}\big)\cdot \widehat{\bm \zeta}\Big)\cdot
\big(\Psi^0+\varepsilon \theta_k^a\frac{\partial \Psi^0}{\partial x_k}\big).}
\end{array}
\end{equation}

Setting $ z^\varepsilon=\Psi^\varepsilon-\Psi^\varepsilon_1 $ and
following the lines of (\ref{eq:2-16})-(\ref{eq:2-19}), we get
\begin{equation}\label{eq:3-25}
\begin{array}{@{}l@{}}
{\displaystyle a^\varepsilon(z^\varepsilon, z^\varepsilon)+\big((\alpha+V_c(\frac{\mathbf{x}}{\varepsilon})-\mathbf{E}^\varepsilon\cdot \widehat{{\bm\zeta}})z^\varepsilon,
z^\varepsilon\big)-(\alpha z^\varepsilon, z^\varepsilon)}\\[2mm]
{\displaystyle \quad =\mathcal{R}e\Big\{\Big(\big((\mathbf{E}^\varepsilon-\mathbf{E}^{\varepsilon,(1)})\cdot \widehat{\bm \zeta}\big)\Psi_1^\varepsilon, z^\varepsilon\Big)\Big\}
+\mathcal{R}e\Big\{\big(\mathcal{F}^\varepsilon_1, z^\varepsilon\big)\Big\}.}
\end{array}
\end{equation}

Let
\begin{equation}\label{eq:3-26}
\begin{array}{l@{}}
{\displaystyle
g_1({\bm\xi},\mathbf{x},t)=-[\hat{a}_{lk}-a_{lk}({\bm \xi})-a_{lj}({\bm \xi})\frac{\partial \theta_k^a({\bm\xi})}{\partial \xi_j}
-\frac{\partial (a_{lj}({\bm \xi})\theta_k^a({\bm \xi}))}{\partial \xi_j}]\frac{\partial^2 \Psi^0(\mathbf{x},t)}{\partial x_l \partial x_k}
,}\\[2mm]
{\displaystyle g_2({\bm\xi},\mathbf{x},t)=-\big(V_c({\bm\xi})-\langle V_c\rangle \big)\Psi^0(\mathbf{x},t).}
\end{array}
\end{equation}

Under the assumptions of this theorem,
we can prove that the scalar functions $ g_j({\bm\xi},\mathbf{x},t)$, $j=1,2 $ are bounded and
measurable in $ ({\bm\xi},\mathbf{x},t) $, 1-periodic in $ {\bm\xi} $,
Lipschitz continuous with respect to $ (\mathbf{x},t) $
uniformly in $ {\bm\xi} $, and
\begin{equation}\label{eq:3-27}
{\displaystyle \int_Q g_j({\bm\xi},\mathbf{x},t)\mathnormal{d}{\bm\xi}=0,\quad j=1,2.}
\end{equation}
By applying Lemma~1.6 of \cite[p.~8]{Ole}, we get
\begin{equation}\label{eq:3-28}
{\displaystyle \left|\int_{\Omega} g_j(\frac{\displaystyle \mathbf{x}}{\displaystyle \varepsilon},\mathbf{x},t)
 v \mathnormal{d}\mathbf{x}\right|\leq C
\varepsilon \|v\|_{\mathbb{H}^1_0(\Omega)}, \quad \forall
v\in \mathbb{H}^1_0(\Omega),\quad j=1,2,}
\end{equation}
where $ C $ is a constant independent of $ \varepsilon $.

Similarly to (\ref{eq:2-20}), under the assumptions of this theorem, combining (\ref{eq:3-24})-(\ref{eq:3-28}) implies
\begin{equation*}
{\displaystyle \|\Psi^\varepsilon-\Psi^\varepsilon_1\|_{L^2(0,T; \/\mathbb{H}_0^1(\Omega))}\leq
C\big\{\varepsilon+\|\mathbf{E}^\varepsilon-\mathbf{E}^0\|_{L^2(0,T;\/ (H^{-1}(\Omega))^3)}\big\}.}
\end{equation*}

It follows from Theorem~\ref{thm2-1} that
\begin{equation*}
{\displaystyle \|\Psi^\varepsilon-\Psi^\varepsilon_1\|_{L^2(0,T; \/\mathbb{H}_0^1(\Omega))}\rightarrow 0,\quad {\rm as}\quad
\varepsilon\rightarrow 0.}
\end{equation*}

We recall (\ref{eq:1-6}) and (\ref{eq:3-17})-(\ref{eq:3-19}), and get
\begin{equation*}
\begin{array}{@{}l@{}}
{\displaystyle
\mathbf{J}_q^\varepsilon-\mathbf{J}_q^{\varepsilon, (1)}=iN\big[
(\overline{\Psi}^\varepsilon-\overline{\Psi}^\varepsilon_1) A(\frac{\mathbf{x}}
{\varepsilon})\nabla \Psi^\varepsilon+
\overline{\Psi}^\varepsilon_1 A(\frac{\mathbf{x}}
{\varepsilon})\nabla (\Psi^\varepsilon-\Psi^\varepsilon_1)}\\[2mm]
{\displaystyle \quad -
(\Psi^\varepsilon-\Psi^\varepsilon_1)A(\frac{\mathbf{x}}
{\varepsilon})\nabla \overline{\Psi}^\varepsilon-
\Psi^\varepsilon_1 A(\frac{\mathbf{x}}
{\varepsilon})\nabla (\overline{\Psi}^\varepsilon-\overline{\Psi}^\varepsilon_1)\big].}
\end{array}
\end{equation*}

Hence we have
\begin{equation}\label{eq:3-29}
{\displaystyle \big(\mathbf{J}^\varepsilon_q-\mathbf{J}^{\varepsilon, (1)}\big)\stackrel{\mathrm{w}}{\rightharpoonup} 0, \,\,
{\rm weakly \,\, in }\,\, L^2(0,T;\/ (L^2(\Omega))^3),\,\,\, {\rm as}\,\, \,\varepsilon\rightarrow 0}.
\end{equation}

Under the assumptions of this theorem, we verify that all conditions of Theorem~2.6 in \cite{Cao-4} can be satisfied.
Following the lines of the proof of Theorem~2.6 in \cite{Cao-4} and using (\ref{eq:3-29}), we prove
\begin{equation*}
\begin{array}{@{}l@{}}
{\displaystyle \int_0^T\int_{\Omega}\big\{
(\mathbf{E}^\varepsilon-\mathbf{E}^{\varepsilon, (1)})\cdot(\mathbf{E}^\varepsilon-\mathbf{E}^{\varepsilon, (1)})
+(\mathbf{H}^\varepsilon-\mathbf{H}^{\varepsilon, (1)})\cdot(\mathbf{H}^\varepsilon-\mathbf{H}^{\varepsilon, (1)})\big\}
\mathnormal{d}\mathbf{x}\mathnormal{d}t}\\[2mm]
{\displaystyle\qquad \qquad \qquad \qquad \qquad\leq C(T) \Big\{\varepsilon
+\langle (\mathbf{J}^\varepsilon_q
-\mathbf{J}^{\varepsilon, (1)}), (\mathbf{E}^\varepsilon-\mathbf{E}^{\varepsilon, (1)})\rangle\Big\},}
\end{array}
\end{equation*}
where $ \langle \mathbf{u}, \mathbf{v}\rangle =\int_0^T\int_{\Omega} \mathbf{u}\cdot \mathbf{v}\mathnormal{d}\mathbf{x}\mathnormal{d}t $,
and $ C(T) $ is a constant independent of~$ \varepsilon $, but dependent of $ T $.
We thus get
\begin{equation*}
{\displaystyle \|\mathbf{E}^\varepsilon-\mathbf{E}^{\varepsilon, (1)}\|_{L^2(0,T; \/(L^2(\Omega))^3)}
+\|\mathbf{H}^\varepsilon-\mathbf{H}^{\varepsilon, (1)}\|_{L^2(0,T; \/(L^2(\Omega))^3)}
\rightarrow 0,\quad {\rm as}\quad \varepsilon\rightarrow 0.}
\end{equation*}
Therefore, we complete the proof of Theorem~\ref{thm3-1}.
\qquad \end{proof}

\begin{corollary}\label{cor3-1}
Suppose that\/ $ \Omega\subset \mathbb{R}^3 $ is the union of entire
periodic cells, i.e. $\overline{\Omega}=\bigcup_{\mathbf{z}\in
I_\varepsilon}\varepsilon(\mathbf{z}+\overline{Q}) $, where the index set
$I_\varepsilon=\{\mathbf{z}\in \mathbb{Z}^3,${\rm~}$\varepsilon(\mathbf{z}+\overline{Q})\subset
\overline{\Omega}\} $ and $ \varepsilon>0 $ is any fixed small parameter. Let\/
$(\Psi^\varepsilon(\mathbf{x},t), \mathbf{E}^{\varepsilon}(\mathbf{x},t),
\mathbf{H}^{\varepsilon}(\mathbf{x},t))$ be the solution of the
original problem (\ref{eq:1-6}), and let\/
$(\Psi_1^\varepsilon(\mathbf{x},t), \mathbf{E}^{\varepsilon, (1)}(\mathbf{x},t),
\mathbf{H}^{\varepsilon, (1)}(\mathbf{x},t))$ and
$(\Psi_2^\varepsilon(\mathbf{x},t), \mathbf{E}^{\varepsilon, (2)}(\mathbf{x},t),
\mathbf{H}^{\varepsilon, (2)}(\mathbf{x},t))$ be the first-order
and the second-order multiscale asymptotic solutions defined in\/
{\rm(\ref{eq:3-17})}-{\rm(\ref{eq:3-19})}, respectively. Under
the assumptions\/ {\rm$({\rm A}_1)$--$({\rm A}_4)$} and\/ {\rm$({\rm
H}_1)$--$({\rm H}_2)$}, if\/
 $(\Psi^0, \mathbf{E}^0, \mathbf{H}^0)\in H^1(0,T; \mathbb{H}^3(\Omega)\times (H^{3}(\Omega))^6)\cap H^2(0,T; $\\
 $ \mathbb{H}^2(\Omega)\times (H^2(\Omega))^6)$
 $\cap H^3(0,T; \mathbb{H}^1(\Omega)\times (H^1(\Omega))^6) $, 
 $\mathbf{f}\in H^3(0,T;(H^{1}(\Omega))^3) $, $ \Psi_0(\mathbf{x})\equiv 0 $,
 $ {\bm \varphi}={\bm\psi}\equiv 0$,
then we have
\begin{equation}\label{eq:3-30}
\begin{array}{l}
{\displaystyle
\|\Psi^\varepsilon-\Psi_s^\varepsilon\|_{L^2(0,T; \/\mathbb{H}_0^1(\Omega))}\rightarrow 0,\quad s=1,2, }\\[2mm]
{\displaystyle \|\mathbf{E}_\varepsilon-\mathbf{E}_\varepsilon^{(s)}\|_{L^\infty(0,T; \/\mathbf{H}(\mathbf{curl};\Omega))}
+\|(\mathbf{E}_\varepsilon-\mathbf{E}_\varepsilon^{(s)})_t\|_{L^\infty(0,T; \/(L^3(\Omega))^3)}\rightarrow 0,}\\[2mm]
{\displaystyle \|\mathbf{H}_\varepsilon-\mathbf{H}_\varepsilon^{(s)}\|_{L^\infty(0,T;
\/\mathbf{H}(\mathbf{curl};\Omega))}
+\|(\mathbf{H}_\varepsilon-\mathbf{H}_\varepsilon^{(s)})_t\|_{L^\infty(0,T;\/
(L^3(\Omega))^3)}\rightarrow 0,}
\end{array}
\end{equation}
as $ \varepsilon \rightarrow 0 $, where $ v_t $ denotes the derivative of $ v $ with respect to $ t $, and  $ C(T) $ is a constant independent of~$ \varepsilon $,
but dependent of $ T $.
\end{corollary}

Following the lines of the proof of Theorem~2.7 in \cite{Cao-4}, we can complete the proof of Corollary~\ref{cor3-1}.

\begin{rem}\label{rem3-4}
We would like to emphasize that Lemmas\/~{\rm\ref{lem3-1}},
{\rm\ref{lem3-2}} and {\rm\ref{lem3-3}} are key for the
convergence analysis of the multiscale asymptotic expansions defined
in {\rm(\ref{eq:3-17})}-{\rm(\ref{eq:3-19})}, since they allow us to
obtain the strong convergence results.
\end{rem}

\begin{rem}\label{rem3-5}
By using Corollary~\ref{cor2-1} and the fact that $ \|\rho^\varepsilon-\hat{\rho}^0\|_{L^2(0,T; H^{-1}(\Omega))}\rightarrow 0 $
as $ \varepsilon\rightarrow 0 $, if the exchange-correlation potential in (\ref{eq:1-6}) is Lipschitz continuous and the corresponding Lipschitz
constant is sufficiently small, then we can obtain the similar convergence results of Theorem~\ref{thm3-1} and Corollary~\ref{cor3-1}.
However, if there is the generic exchange-correlation potential (see, e.g., \cite[p.~152-169]{Mar}),
then the convergence result of the
multiscale asymptotic method for the problem
(\ref{eq:1-6}) is not known to authors.
\end{rem}

\begin{rem}\label{rem3-6}
It should be stated that the derived convergence results in
Theorem~\ref{thm3-1}, Corollary~\ref{cor3-1} and Remark~\ref{rem3-5} are valid provided that the very strict regularity of
the solution  $(\Psi^0, \mathbf{E}^0, \mathbf{H}^0) $ of the associated homogenized problem (\ref{eq:2-5}) is satisfied.
So far, it seems to be open and challenging. Even so, the formal multiscale
asymptotic expansion is particularly useful for developing efficient
numerical methods. The numerical results presented in
Section~\ref{sec-5} strongly support the assessment. In particular, the second-order multiscale
solution is necessary and essential in some cases.
\end{rem}

\section{Multiscale numerical algorithms}\label{sec-4}

In this section, we first summarize the main steps for the homogenization method and the multiscale
method presented in the previous sections for the Maxwell-Schr\"{o}dinger
system (\ref{eq:1-6}) with rapidly oscillating discontinuous coefficients, and then we
provide the associated numerical algorithms.

According to (\ref{eq:3-17})-(\ref{eq:3-19}), the multiscale method for solving the Maxwell-\\
Schr\"{o}dinger system
(\ref{eq:1-6}) is composed of the following steps:

{\it Step }1.\, Compute the scalar cell functions $ \theta_k^a({\bm \xi}) $, $ \theta_{kl}^a({\bm \xi}) $;
$ \theta_k^\eta({\bm \xi}) $, $ \theta_{kl}^\eta({\bm \xi}) $;
$ \theta_k^\mu({\bm \xi}) $, $ \theta_{kl}^\mu({\bm \xi}) $,
$ k,l=1,2,3 $ defined in
(\ref{eq:3-1})-(\ref{eq:3-6}) and the matrix-valued cell functions
${\bm\Theta}_1^\eta({\bm \xi}) $, $ {\bm\Theta}_2^\eta({\bm \xi}) $,
${\bm\Theta}_1^\mu({\bm \xi}) $, $ {\bm\Theta}_2^\mu({\bm \xi}) $ defined in
(\ref{eq:3-7})-(\ref{eq:3-10}), respectively. Then we get the approximations of
the homogenized coefficient matrices $ \widehat{A}=(\hat{a}_{ij}) $, $ \widehat{\eta}=(\hat{\eta}_{ij}) $
and $ \widehat{\mu}=(\hat{\mu}_{ij}) $.

{\it Step} 2. \, Solve the homogenized
Maxwell-Schr\"{o}dinger system (\ref{eq:2-5}) with constant coefficients over a whole domain $ \Omega\times (0,T) $
in a coarse mesh and at a larger time step, where $ T>0 $ is arbitrary.

{\it Step} 3. \, Apply higher-order difference quotients to
compute the derivatives of the solution $ (\Psi^0(\mathbf{x},t), \mathbf{E}^0(\mathbf{x},t),
\mathbf{H}^0(\mathbf{x},t)) $ for the homogenized
Maxwell-Schr\"{o}dinger system (\ref{eq:2-5})  based on (\ref{eq:3-17})-(\ref{eq:3-19}), respectively.
For the detailed formulas, see \cite{Cao-4, Zhang-L}.

At Step 1, we employ the adaptive finite element method to solve the
boundary value problems (\ref{eq:3-1})-(\ref{eq:3-6}) with discontinuous coefficients
in the unit cell $ Q $ for computing the scalar cell functions $ \theta_k^a({\bm \xi}) $, $ \theta_{kl}^a({\bm \xi}) $;
$ \theta_k^\eta({\bm \xi}) $, $ \theta_{kl}^\eta({\bm \xi}) $;
$ \theta_k^\mu({\bm \xi}) $, $ \theta_{kl}^\mu({\bm \xi}) $,
$ k,l=1,2,3 $. For more information of the algorithm and the convergence, we refer to \cite{Zhang-L}.
Then we use the adaptive edge element method
to solve the time-harmonic Maxwell's equations (\ref{eq:3-7})-(\ref{eq:3-10}) with discontinuous coefficients in the unit cell
$ Q $ for computing the matrix-valued cell functions
${\bm\Theta}_1^\eta({\bm \xi}) $, $ {\bm\Theta}_2^\eta({\bm \xi}) $,
${\bm\Theta}_1^\mu({\bm \xi}) $, $ {\bm\Theta}_2^\mu({\bm \xi}) $.
For the details and the convergence, we refer the interested reader to
\cite{Cao-3, Cao-4, Ne-1}.

Next we focus on discussing the algorithm for solving the homogenized Maxwell-Schr\"{o}dinger
equations with constant coefficients in $ \Omega\times (0,T) $ at Step 2,
which is a nonlinear and nonconvex coupled system with constant coefficients.
We recall some important studies about the problem.
The most widely used computational method is the time-domain finite difference(FDTD) method since it is simple and easy to implement.
For example, Lorin, Chelkowski and Bandrauk \cite{Lor-1} used the finite difference method for Schr\"{o}dinger equation and
the FDTD method for Maxwell's equations to simulate intense ultrashort laser pulses interaction with a 3$\mathrm{D} $ $\mathrm{H}_2^{+}$ gas.
Pierantoni, Mencarelli and Rozzi \cite{Pi} applied the transmission line matrix method(TLM) for Maxwell's equations and the FDTD method
for the Schr\"{o}dinger equation to simulate a carbon nanotube between
two metallic electrodes. Ahmed and Li \cite{Ah-1} used the FDTD method for the Maxwell-Schr\"{o}dinger system
to simulate plasmonics nanodevices. Sato and Yabana \cite{Sa} combined the FDTD method for propagation of macroscopic electromagnetic fields
and time-dependent density functional theory (TDDFT) for quantum dynamics of electrons to simulate  interactions
between an intense laser field and a solid. For other numerical methods of the Maxwell-Schr\"{o}dinger system, we refer to
\cite{Lop, Lor-2, Oh, Tur}.

In this section, we first use the finite element method in space to discretize
the homogenized Maxwell-Schr\"{o}dinger system to the nonlinear ordinary differential system.
Then we apply the midpoint scheme to discretize the system to the nonlinear discrete system.
Finally, we combine self-consistent method (SCF) and the simple
mixed method to solve the nonlinear discrete system (see, e.g., \cite{Ha, Hu, Pul-1}).

We would like to state that we only can get the approximations of the homogenized coefficient
matrices $ \widehat{A}=(\hat{a}_{ij}) $, $ \widehat{\eta}=(\hat{\eta}_{ij}) $ and
$ \widehat{\mu}=(\hat{\mu}_{ij}) $ in the real simulation. $ \widehat{A}^{h_0} $,
$ \widehat{\eta}^{h_0} $ and $ \widehat{\mu}^{h_0} $
denote respectively the approximate values of $ \widehat{A} $, $ \widehat{\eta} $ and $ \widehat{\mu} $,
where $ h_0 $ is the mesh size for
computing the scalar cell functions $ \theta_k^a({\bm \xi}) $, $ \theta_{kl}^a({\bm \xi}) $;
$ \theta_k^\eta({\bm \xi}) $, $ \theta_{kl}^\eta({\bm \xi}) $;
$ \theta_k^\mu({\bm \xi}) $, $ \theta_{kl}^\mu({\bm \xi}) $,
$ k,l=1,2,3 $, and the matrix-valued cell functions $ {\bm\Theta}_1^\eta({\bm \xi}) $
, $ {\bm\Theta}_2^\eta({\bm \xi}) $, $ {\bm\Theta}_1^\mu({\bm \xi}) $
and $ {\bm\Theta}_2^\mu({\bm \xi}) $. It can be proved that

\begin{proposition}\label{prop4-1}
Let $ \widehat{A}=(\hat{a}_{ij}) $, $ \widehat{\eta}=(\hat{\eta}_{ij}) $ and $ \widehat{\mu}=(\hat{\mu}_{ij}) $
be the homogenized coefficient matrices
calculated by (\ref{eq:2-4}) and let $ \widehat{A}^{h_0}=(\hat{a}_{ij}^{h_0}) $,
$ \widehat{\eta}^{h_0}=(\hat{\eta}_{ij}^{h_0}) $ and $ \widehat{\mu}=(\hat{\mu}_{ij}^{h_0}) $
be the corresponding finite element approximations, respectively.
Suppose the mesh size $ h_0>0 $ is sufficiently small, then we have
\begin{equation}\label{eq:4-1}
\begin{array}{l@{}}
{\displaystyle \hat{a}_{ij}^{h_0}=\hat{a}_{ji}^{h_0},\quad \hat{\eta}_{ij}^{h_0}=\hat{\eta}_{ji}^{h_0},\quad
\hat{\mu}_{ij}^{h_0}=\hat{\mu}_{ji}^{h_0},
\quad \forall i,j=1,2,3, }\\[2mm]
{\displaystyle \max\limits_{i,j}|\hat{a}_{ij}-\hat{a}_{ij}^{h_0}|\leq C h_0^2,\quad
\max\limits_{i,j}|\hat{\eta}_{ij}-\hat{\eta}_{ij}^{h_0}|\leq C h_0^2,
\quad
\max\limits_{i,j}|\hat{\mu}_{ij}-\hat{\mu}_{ij}^{h_0}|\leq C h_0^2,}\\[2mm]
{\displaystyle \bar{\alpha}_0 |\mathbf{y}|^2\leq \hat{a}_{ij}^{h_0}y_i y_j \leq \bar{\alpha}_1 |\mathbf{y}|^2,
\quad \bar{\beta}_0 |\mathbf{y}|^2\leq \hat{\eta}_{ij}^{h_0}y_i y_j \leq \bar{\beta}_1 |\mathbf{y}|^2,}\\[2mm]
{\displaystyle \bar{\gamma}_0 |\mathbf{y}|^2\leq \hat{\mu}_{ij}^{h_0}y_i y_j \leq \bar{\gamma}_1 |\mathbf{y}|^2,
\quad
\forall \mathbf{y}=(y_1, y_2, y_3)\in \mathbb{R}^3,\quad
|\mathbf{y}|^2=y_i y_i,}
\end{array}
\end{equation}
where $ C $, $ \bar{\alpha}_0 $, $ \bar{\alpha}_1 $, $ \bar{\beta}_0 $,
$ \bar{\beta}_1 $ ,$ \bar{\gamma}_0 $, $ \bar{\gamma}_1 $ are constants independent of $ \varepsilon $, $ h_0 $;
$ h_0 $ is the final mesh size of the adaptive finite elements for computing the cell functions.
\end{proposition}

\begin{proof}
Following the lines of the proofs of Proposition 4.3 of \cite{Cao-1} and Proposition 3.3 of \cite{Cao-3},
we can complete the proof of Proposition~\ref{prop4-1}.
\qquad
\end{proof}

In the real simulation, we will solve the modified homogenized Maxwell-Schr\"{o}dinger equations are given by
\begin{equation}\label{eq:4-2}
\left\{
\begin{array}{@{}l@{}}
{\displaystyle i\frac{\partial \Psi^{0, h_0}}{\partial t}=
-\nabla \cdot \big(\widehat{A}^{h_0}\nabla
\Psi^{0,h_0}\big)
+\big(\langle V_c\rangle-\mathbf{E}^{0,h_0}\cdot\widehat{\bm{\zeta}}
+V_{xc}[\rho^{0,h_0}]\big)
\Psi^{0,h_0},}\\[2.3mm]
{\displaystyle \qquad \qquad \qquad \qquad (\mathbf{x},t)\in
\Omega\times(0,T),}\\[2.3mm]
{\displaystyle \hat{\eta}^{h_0}\frac{\partial \mathbf{E}^{0,h_0}}{\partial t}
=\mathbf{curl}\/ \mathbf{H}^{0,h_0}+\mathbf{f}
-\mathbf{J}_q^{0,h_0},\,\, \nabla\cdot \mathbf{f}=0,\,\,(\mathbf{x},t)\in
\Omega\times(0,T),}\\[2.3mm]
{\displaystyle \hat{\mu}^{h_0}
\frac{\partial \mathbf{H}^{0,h_0}}{\partial t}=-\mathbf{curl}\/ \mathbf{E}^{0,h_0}
,\,\, (\mathbf{x},t)\in
\Omega\times(0,T),}\\[2.3mm]
{\displaystyle \nabla\cdot \big(\hat{\eta}^{h_0}\mathbf{E}^{0,h_0}\big)=\rho^{0,h_0},
\quad\nabla\cdot \big(\hat{\mu}^{h_0}\mathbf{H}^{0,h_0}\big)=0,
\,\, (\mathbf{x},t)\in
\Omega\times(0,T),}\\[2mm]
{\displaystyle  \rho^{0,h_0}=N{\vert\Psi^{0,h_0}\vert}^2
,\,\,\mathbf{J}_q^{0,h_0}=iN\big[(\overline{\Psi}^{0,h_0}) \widehat{A}^{h_0}\nabla\Psi^{0,h_0}-
\Psi^{0,h_0} \widehat{A}^{h_0}\nabla\overline{\Psi}^{0,h_0}\big].}
\end{array}
\right.
\end{equation}

Next we will analyze the difference between $ (\Psi^{0,h_0}, \mathbf{E}^{0,h_0}, \mathbf{H}^{0,h_0}) $
and $ (\Psi^{0}, \mathbf{E}^{0}, \mathbf{H}^{0}) $.
\begin{proposition}\label{prop4-2}
Let $ (\Psi^{0}, \mathbf{E}^{0}, \mathbf{H}^{0}) $ and $ (\Psi^{0,h_0}, \mathbf{E}^{0,h_0}, \mathbf{H}^{0,h_0}) $
be the solutions of the homogenized Maxwell-Schr\"{o}dinger system (\ref{eq:2-5}) without the exchange-correlation
potential and the associated modified Maxwell-Schr\"{o}dinger
system (\ref{eq:4-2}), respectively. If the mesh size $ h_0>0 $ is sufficiently small, then we prove
\begin{equation}\label{eq:4-3}
\begin{array}{@{}l@{}}
{\displaystyle \|\Psi^{0,h_0}-\Psi^0\|_{L^2(0,T; \/\mathbb{H}_0^1(\Omega))}
+\|\mathbf{E}^{0,h_0}-\mathbf{E}^{0}\|_{L^\infty(0,T; \/\mathbf{H}(\mathbf{curl}; \Omega))}}\\[2mm]
{\displaystyle \qquad+\|\mathbf{H}^{0,h_0}-\mathbf{H}^{0}\|_{L^\infty(0,T; \/\mathbf{H}(\mathbf{curl}; \Omega))}
\rightarrow 0,\quad {\rm as}\quad h_0\rightarrow 0.}
\end{array}
\end{equation}
\end{proposition}

\begin{proof}
From $ (\ref{eq:4-3})_4 $, setting $ \mathbf{E}^{0,h_0}=-\nabla \phi^{0,h_0} $,
we have
\begin{equation}\label{eq:4-4}
\left\{
\begin{array}{@{}l@{}}
{\displaystyle -\nabla\cdot \big(\hat{\eta}^{h_0}\nabla \phi^{0,h_0}\big)=\rho^{0,h_0},\quad
x\in \Omega,}\\[2mm]
{\displaystyle \phi^{0,h_0}=0,\quad x\in \partial \Omega.}
\end{array}
\right.
\end{equation}

Thanks to Proposition~\ref{prop4-1}, we get $ \|\phi^{0,h_0}\|_{H_0^1(\Omega)}\leq C
\|\rho^{0,h_0}\|_{H^{-1}(\Omega)} $, where $ C $ is a constant independent of $ h_0 $.
Therefore, for any fixed $ t\in (0,T) $,
there is a subsequence, without confusion still denoted by $ \phi^{0,h_0} $,
such that
\begin{equation}\label{eq:4-5}
{\displaystyle\phi^{0,h_0} \stackrel{\mathrm{w}}{\rightharpoonup}\widetilde{\phi}^0 \quad \hbox{weakly\,\,\,in}
\quad H^1_0(\Omega) \quad \hbox{as}\quad
h_0\rightarrow 0.}
\end{equation}

Furthermore, we have
\begin{equation}\label{eq:4-6}
{\displaystyle \mathbf{E}^{0,h_0}=-\nabla \phi^{0,h_0}
\stackrel{\mathrm{w}}{\rightharpoonup}-\nabla\widetilde{\phi}^0=\widetilde{\mathbf{E}}^0 \quad \hbox{weakly\,\,\,in}
\quad (L^2(\Omega))^3 \quad \hbox{as}\quad
h_0\rightarrow 0.}
\end{equation}

Let $ (\widetilde{\Psi}^0(\mathbf{x},t), \widetilde{\mathbf{E}}^0(\mathbf{x},t), \widetilde{\mathbf{H}}^0(\mathbf{x},t)) $ be the solution of the following Maxwell-Schr\"{o}dinger equations:
\begin{equation}\label{eq:4-7}
\left\{
\begin{array}{@{}l@{}}
{\displaystyle i\frac{\partial \widetilde{\Psi}^0}{\partial t}=
-\nabla \cdot \big(\widehat{A}\nabla
\widetilde{\Psi}^0\big)
+\big(\langle V_c\rangle-\widetilde{\mathbf{E}}^0\cdot\widehat{\bm{\zeta}}\big)
\widetilde{\Psi}^0,\quad(\mathbf{x},t)\in\Omega\times(0,T).}\\[2.3mm]
{\displaystyle \widehat{\eta}\frac{\partial \widetilde{\mathbf{E}}^0}{\partial t}
=\mathbf{curl}\/ \widetilde{\mathbf{H}}^0+\mathbf{f}
-\widetilde{\mathbf{J}}_q^0,\,\, \nabla\cdot \mathbf{f}=0,\,\,(\mathbf{x},t)\in
\Omega\times(0,T),}\\[2.3mm]
{\displaystyle \widehat{\mu}
\frac{\partial \widetilde{\mathbf{H}}^0}{\partial t}=-\mathbf{curl}\/ \widetilde{\mathbf{E}}^0
,\,\, (\mathbf{x},t)\in
\Omega\times(0,T),}\\[2.3mm]
{\displaystyle \nabla\cdot \big(\widehat{\eta}\widetilde{\mathbf{E}}^0\big)=\tilde{\rho}^0,
\quad\nabla\cdot \big(\widehat{\mu}\widetilde{\mathbf{H}}^0\big)=0,
\,\, (\mathbf{x},t)\in
\Omega\times(0,T),}
\end{array}
\right.
\end{equation}
where $ \widetilde{\rho}^0=N|\widetilde{\Psi}^0|^2 $ and
$ \widetilde{\mathbf{J}}_q^0=
iN\big[\overline{\widetilde{\Psi}^0} \widehat{A}\nabla\widetilde{\Psi}^0-
\widetilde{\Psi}^0 \widehat{A}\nabla\overline{\widetilde{\Psi}^0}\big] $.

Subtracting $(\ref{eq:4-7})_1 $ from $ (\ref{eq:4-2})_1 $, we get
\begin{equation}\label{eq:4-8}
\begin{array}{@{}l@{}}
{\displaystyle i\frac{\partial (\widetilde{\Psi}^0-\Psi^{0,h_0})}{\partial t}
=-\nabla\cdot \big(\widehat{A}\nabla(\widetilde{\Psi}^0-\Psi^{0,h_0})\big)
+\big(\langle V_c\rangle-\widetilde{\mathbf{E}}^0\cdot \widehat{\bm \zeta}\big)
(\widetilde{\Psi}^0-\Psi^{0,h_0})}\\[2mm]
{\displaystyle\quad -\nabla\cdot\big((\widehat{A}-\widehat{A}^{h_0})\nabla
\Psi^{0,h_0}\big)-\big((\widetilde{\mathbf{E}}^0-\mathbf{E}^{0,h_0})\cdot\widehat{\bm \zeta}\big)
\Psi^{0,h_0},\quad(\mathbf{x},t)\in\Omega\times(0,T).}
\end{array}
\end{equation}

Following the lines of (\ref{eq:2-16})-(\ref{eq:2-22}) and using (\ref{eq:4-6}) and Proposition~\ref{prop4-1},
we prove
\begin{equation}\label{eq:4-9}
\begin{array}{@{}l@{}}
{\displaystyle \|\Psi^{0,h_0}-\widetilde{\Psi}^0\|_{L^2(0,T; \/\mathbb{H}_0^1(\Omega))}
\leq C \big\{h_0^2+\|\widetilde{\mathbf{E}}^0-\mathbf{E}^{0,h_0}\|_{L^2(0,T; \/(H^{-1}(\Omega))^3)}\big\}
\rightarrow 0,}
\end{array}
\end{equation}
as $ h_0\rightarrow 0 $.

Furthermore, we have
\begin{equation}\label{eq:4-10}
\begin{array}{@{}l@{}}
{\displaystyle \mathbf{J}_q^{0,h_0}\stackrel{\mathrm{w}}{\rightharpoonup}\widetilde{\mathbf{J}}_q^0 , \,\,
{\rm weakly \,\, in }\,\, L^2(0,T;\/ (L^2(\Omega))^3),\,\,\, {\rm as}\,\, \,h_0\rightarrow 0.}
\end{array}
\end{equation}

Subtracting (\ref{eq:4-7}) from (\ref{eq:4-2}) gives
\begin{equation}\label{eq:4-11}
\left\{
\begin{array}{@{}l@{}}
{\displaystyle \widehat{\eta}\frac{\partial (\widetilde{\mathbf{E}}^0-\mathbf{E}^{0,h_0})}{\partial t}
+\big(\widehat{\eta}-\widehat{\eta}^{h_0}\big)\frac{\partial \mathbf{E}^{0,h_0}}{\partial t}
=\mathbf{curl}\/ \big(\widetilde{\mathbf{H}}^0-\mathbf{H}^{0,h_0}\big)}\\[2.3mm]
{\displaystyle\qquad \qquad\qquad \qquad-\big(\widetilde{\mathbf{J}}_q^0-\mathbf{J}_q^{0,h_0}\big),\,\,(\mathbf{x},t)\in
\Omega\times(0,T),}\\[2.3mm]
{\displaystyle \widehat{\mu}
\frac{\partial (\widetilde{\mathbf{H}}^0-\mathbf{H}^{0,h_0})}{\partial t}
+\big(\widehat{\mu}-\widehat{\mu}^{h_0}\big)\frac{\partial \mathbf{H}^{0,h_0}}{\partial t}
=-\mathbf{curl}\/ \big(\widetilde{\mathbf{E}}^0-\mathbf{E}^{0,h_0}\big),}\\[2.3mm]
{\displaystyle \qquad \qquad \qquad \qquad \qquad \qquad(\mathbf{x},t)\in
\Omega\times(0,T).}
\end{array}
\right.
\end{equation}

Multiplying $ (\ref{eq:4-11})_1 $ by $
(\widetilde{\mathbf{E}}^0-\mathbf{E}^{0,h_0}) $
and $ (\ref{eq:4-11})_2 $ by $
(\widetilde{\mathbf{H}}^0-\mathbf{H}^{0,h_0}) $,
and integrating on $ (0,t)\times \Omega $, we obtain
\begin{equation}\label{eq:4-12}
\begin{array}{@{}l@{}}
{\displaystyle \int_0^t \int_{\Omega}
\widehat{\eta}\frac{\partial
(\widetilde{\mathbf{E}}^0-\mathbf{E}^{0,h_0})}{
\partial \tau}\cdot (\widetilde{\mathbf{E}}^0-\mathbf{E}^{0,h_0})
\mathnormal{d}\mathbf{x}\mathnormal{d}\tau}\\[2.2mm]
{\displaystyle\quad+\int_0^t \int_{\Omega}
\widehat{\mu}\frac{\partial (\widetilde{\mathbf{H}}^0-\mathbf{H}^{0,h_0})}
{\partial\tau}\cdot (\widetilde{\mathbf{H}}^0-\mathbf{H}^{0,h_0})
\mathnormal{d}\mathbf{x}\mathnormal{d}\tau}\\[2.2mm]
{\displaystyle \quad=-\int_0^t \int_{\Omega}
(\widehat{\eta}-\widehat{\eta}^{h_0})\frac{\partial \mathbf{E}^{0,h_0}}
{\partial t}\cdot (\widetilde{\mathbf{E}}^0-\mathbf{E}^{0,h_0})
\mathnormal{d}\mathbf{x}\mathnormal{d}\tau}\\[2.2mm]
{\displaystyle \quad-\int_0^t \int_{\Omega}
(\widehat{\mu}-\widehat{\mu}^{h_0})\frac{\partial \mathbf{H}^{0,h_0}}
{\partial t}\cdot (\widetilde{\mathbf{H}}^0-\mathbf{H}^{0,h_0})
\mathnormal{d}\mathbf{x}\mathnormal{d}\tau}\\[2.2mm]
{\displaystyle \quad-\int_0^t \int_{\Omega}
(\widetilde{\mathbf{J}}_q^0-\mathbf{J}_q^{0,h_0})\cdot (\widetilde{\mathbf{E}}^0-\mathbf{E}^{0,h_0})
\mathnormal{d}\mathbf{x}\mathnormal{d}\tau.}
\end{array}
\end{equation}

It follows from (\ref{eq:4-10}) and Proposition~\ref{prop4-1} that
\begin{equation*}
{\displaystyle \|\mathbf{E}^{0,h_0}-\widetilde{\mathbf{E}}^0\|_{L^\infty(0,T; \/(L^2(\Omega))^3)}
+\|\mathbf{H}^{0,h_0}-\widetilde{\mathbf{H}}^0\|_{L^\infty(0,T; \/(L^2(\Omega))^3)}
\rightarrow 0,\,\,\, {\rm as}\,\,\, h_0\rightarrow 0.}
\end{equation*}

Following the lines of the proof of Corollary~\ref{cor3-1}, we can prove
\begin{equation*}
{\displaystyle \|\mathbf{E}^{0,h_0}-\widetilde{\mathbf{E}}^0\|_{L^\infty(0,T; \/\mathbf{H}(\mathbf{curl}; \Omega))}
+\|\mathbf{H}^{0,h_0}-\widetilde{\mathbf{H}}^0\|_{L^\infty(0,T;\/ \mathbf{H}(\mathbf{curl}; \Omega))}
\rightarrow 0,\,\,\, {\rm as}\,\,\, h_0\rightarrow 0.}
\end{equation*}

We recall (\ref{eq:4-4}), and define $ \widetilde{\phi}^{0,h_0}(x) $ is the solution of the following
elliptic equation:
\begin{equation}\label{eq:4-13}
\left\{
\begin{array}{@{}l@{}}
{\displaystyle -\nabla\cdot \big(\hat{\eta}^{h_0}\nabla \widetilde{\phi}^{0,h_0}\big)=\widetilde{\rho}^0,\quad
x\in \Omega,}\\[2mm]
{\displaystyle \widetilde{\phi}^{0,h_0}=0,\quad x\in \partial \Omega,}
\end{array}
\right.
\end{equation}
where $ \widetilde{\rho}^0=N|\widetilde{\Psi}^0|^2 $. Using Proposition~\ref{prop4-1} and the fact
that $ \|\rho^{0,h_0}-\widetilde{\rho}^0\|_{H^{-1}(\Omega)}\rightarrow 0 $ as $ h_0\rightarrow 0 $,
we obtain
\begin{equation}\label{eq:4-14}
{\displaystyle\|\widetilde{\phi}^{0,h_0}-\phi^{0,h_0}\|_{H^1_0(\Omega)}\leq C
\|\rho^{0,h_0}-\widetilde{\rho}^0\|_{H^{-1}(\Omega)}\rightarrow 0, \quad \hbox{as}\quad
h_0\rightarrow 0.}
\end{equation}
Combining (\ref{eq:4-5}) and (\ref{eq:4-14}) implies
$ \widetilde{\phi}^0(\mathbf{x},t)=\widetilde{\phi}^0(\widetilde{\rho}^0) $. Using the uniqueness of the solution of the Poisson equation
$(\ref{eq:2-5})_4 $
for the homogenized Maxwell-Schr\"{o}dinger system without the exchange-correlation potential,
the convergence (\ref{eq:4-4}) takes place for the whole sequences. Therefore,
we get $ \phi^0=\phi^0(\rho^0) $. From this, we obtain
\begin{equation}\label{eq:4-15}
\begin{array}{l}
{\displaystyle \widetilde{\Psi}^0=\Psi^0,\,\,\, \widetilde{\mathbf{E}}^0=\mathbf{E}^0,\,\,\,
\widetilde{\mathbf{H}}^0=\mathbf{H}^0.}
\end{array}
\end{equation}
Therefore, we complete the proof of Proposition~\ref{prop4-2}.
\qquad \end{proof}

\begin{rem}\label{rem4-1}
Similarly, if the exchange-correlation potential in (\ref{eq:1-6}) is Lipschitz continuous and the corresponding Lipschitz
constant is sufficiently small, then we can obtain the similar convergence results of Proposition~\ref{prop4-2}.
However, for the generic exchange-correlation potential (see, e.g., \cite[p.~152-169]{Mar}), it seems to be open.
\end{rem}

In this paper, we solve the following homogenized Maxwell-Schr\"{o}dinger equations with constant coefficients
instead of (\ref{eq:4-2}):
\begin{equation}\label{eq:4-16}
\left\{
\begin{array}{@{}l@{}}
{\displaystyle i\frac{\partial \Psi^{0, h_0}}{\partial t}=
-\nabla \cdot \big(\widehat{A}^{h_0}\nabla
\Psi^{0,h_0}\big)
+\big(\langle V_c\rangle-\mathbf{E}^{0,h_0}\cdot\widehat{\bm{\zeta}}
+V_{xc}[\rho^{0,h_0}]\big)
\Psi^{0,h_0},}\\[2.2mm]
{\displaystyle \qquad \qquad \qquad \qquad (\mathbf{x},t)\in
\Omega\times(0,T),}\\[2.2mm]
{\displaystyle \widehat{\eta}^{h_0}\frac{\partial^2 \mathbf{E}^{0,h_0}}{\partial t^2}
+\mathbf{curl}\/ \big((\widehat{\mu}^{h_0})^{-1}\mathbf{curl}\/ \mathbf{E}^{0,h_0}\big)=\mathbf{F}
-\frac{\partial \mathbf{J}_q^{0,h_0}}{\partial t},}\\[2.2mm]
{\displaystyle \mathbf{F}=\frac{\partial \mathbf{f}}{\partial t},\quad\nabla\cdot \mathbf{F}=0,\quad(\mathbf{x},t)\in
\Omega\times(0,T),}\\[2.2mm]
{\displaystyle  \rho^{0,h_0}=N{\vert\Psi^{0,h_0}\vert}^2
,\,\,\mathbf{J}_q^{0,h_0}=iN\big[(\overline{\Psi}^{0,h_0}) \widehat{A}^{h_0}\nabla\Psi^{0,h_0}-
\Psi^{0,h_0} \widehat{A}^{h_0}\nabla\overline{\Psi}^{0,h_0}\big],}\\[2.2mm]
{\displaystyle \Psi^{0,h_0}(\mathbf{x},t)=0,\quad \mathbf{E}^{0,h_0}(\mathbf{x},t)\times \mathbf{n}=0,\quad
(\mathbf{x},t)\in \partial \Omega \times (0,T),}\\[2.2mm]
{\displaystyle \Psi^{0,h_0}(\mathbf{x},0)=\Psi_0(\mathbf{x}),\quad
\mathbf{E}^{0, h_0}(\mathbf{x}, 0)=\mathbf{E}_0(\mathbf{x}),\quad
\frac{\partial \mathbf{E}^{0,h_0}(\mathbf{x},0)}{\partial t}=\mathbf{E}_1(\mathbf{x}),}
\end{array}
\right.
\end{equation}
where $ \mathbf{E}_1(\mathbf{x})=\big\{(\widehat{\eta}^{h_0})^{-1}\big(\mathbf{curl}\/ \mathbf{H}^{0,h_0}+\mathbf{f}
-\mathbf{J}_q^{0,h_0}\big)\big\}|_{t=0} $.

The variational form of  (\ref{eq:4-16}) is written as
\begin{equation}\label{eq:4-17}
\left\{
\begin{array}{l@{}}
{\displaystyle (i\frac{\partial \Psi^{0,h_0}}{\partial t},\varphi)
=a(\mathbf{E}^{0,h_0};\Psi^{0,h_0},\varphi),\quad \forall  \varphi=\varphi_1+i\varphi_2,\quad i^2=-1,\quad  \varphi_1, \varphi_2\in
H^1_0(\Omega),} \\[2mm]
{\displaystyle \langle \widehat{\eta}^{h_0}\frac{\partial^{2}\mathbf{E}^{0,h_0}}{\partial t^2},\mathbf{v}\rangle 
+b(\mathbf{E}^{0, h_0},\mathbf{v})
=\langle \mathbf{F},\mathbf{v}\rangle-\langle\frac{\partial \mathbf{J}_q^{0, h_0}}{\partial t},\mathbf{v}\rangle,\quad \forall
\mathbf{v}\in \mathbf{H}_0(\mathbf{curl};\Omega),}
\end{array}
\right.
\end{equation}
where
\begin{equation*}
\begin{array}{l@{}}
{\displaystyle a(\mathbf{E}^{0,h_0}; \Psi^{0,h_0}, \varphi)=\int_\Omega\big\{\widehat{A}^{h_0}\nabla
\Psi^{0,h_0}\cdot\nabla\bar{\varphi}+(\langle V_c \rangle-\mathbf{E}^{0,h_0}\cdot\widehat{\bm{\zeta}}
+V_{xc}[\rho^{0,h_0}])\Psi^{0, h_0}\bar{\varphi} \big\} \mathnormal{d}\mathbf{x},}\\[2mm]
{\displaystyle b(\mathbf{u},\mathbf{v})=\int_\Omega
(\widehat{\mu}^{h_0})^{-1}\mathbf{curl\/u}\cdot\mathbf{curl}\/\mathbf{v} \mathnormal{d}\mathbf{x},\quad
 (\psi, \varphi)=\int_\Omega \psi \bar{\varphi} \mathnormal{d}\mathbf{x},\quad
 \langle \mathbf{u}, \mathbf{v}\rangle=\int_\Omega
\mathbf{u}\cdot \mathbf{v} \mathnormal{d}\mathbf{x}.}
\end{array}
\end{equation*}

Let $\tau_{h}=\{e\}$ be a regular family of tetrahedrons of a whole
domain $\Omega$ and $h=\underset {e}{max}\{h_{e}\}$.
We define the linear finite element space of $H^1_0(\Omega)$,
\begin{equation}\label{eq:4-18}
{\displaystyle\mathcal{U}_h(\Omega)=\{u_h\in C(\overline{\Omega}):\,\,\, u_h|_{e}\in
P_1,\,\,\, u_h|_{\partial \Omega}=0\},}
\end{equation}
and the finite element space of $ \mathbf{H}_{0}(\mathbf{curl};\Omega)$
consisting of degree$-k$ edge elements by
\begin{equation}\label{eq:4-19}
{\displaystyle\mathcal{W}_h(\Omega)=\{\mathbf{w}_h\in \mathbf{H}(\mathbf{curl};\Omega):
\,\,\, \mathbf{w}_h|_e\in R_k,\,\,\mathbf{w}_h\times \mathbf{n}=0 \,\, \hbox{on}\,\,
\partial \Omega \},}
\end{equation}
where $ \mathbf{n} $ is the outward unit normal to the boundary $
\partial \Omega $ and $ R_k $ is defined in (5.32) of (\cite[p.~128]{Monk}).

As usual, the complex function $ \Psi_h^{0,h_0} $ is decomposed into two parts:
$ \Psi_h^{0,h_0}=\Psi_{h,R}^{0,h_0}+i\Psi_{h, I}^{0,h_0} $, where $ i^2=-1 $, $ \Psi_{h,R}^{0,h_0} $ and
$\Psi_{h, I}^{0,h_0} $ are the real and the imaginary part of $ \Psi_h^{0,h_0} $, respectively.
The semi-discrete scheme for solving the problem
(\ref{eq:4-16}) is as follows:
Find $ \Psi_{h, R}^{0, h_0},  \Psi_{h, I}^{0, h_0}\in L^2(0,T; \/\mathcal{U}_h(\Omega))$ and
$\mathbf{E}_h^{0,h_0}\in L^2(0,T; \mathcal{W}_h(\Omega))$
with $\Psi_{h, R}^{0, h_0}(0)=\Psi^R_{0 h},\, \Psi_{h, I}^{0, h_0}(0)=\Psi^I_{0 h}\in \mathcal{U}_h(\Omega)$,
$\mathbf{E}_h^{0, h_0}(0)=\mathbf{E}_{0 h}\in \mathcal{W}_h(\Omega)$,
$\partial_t \mathbf{E}_h^{0,h_0}(0)=\mathbf{E}_{1 h}\in
\mathcal{W}_h(\Omega)$ such that
\begin{equation}\label{eq:4-20}
\left\{
\begin{array}{l@{}}
{\displaystyle (i\frac{d \Psi_h^{0,h_0}}{d t}, u_h+i v_h)=a(\mathbf{E}_{h}^{0,h_0};
\Psi_h^{0, h_0}, u_h+i v_h),\quad \forall  u_h,\, v_h \in \mathcal{U}_h(\Omega),} \\[2mm]
{\displaystyle \langle\frac{d^2}{d t^2}{\mathbf{E}}_h^{0, h_0}(\mathbf{x},t),
\mathbf{w}_{h}\rangle+b({\mathbf{E}}_h^{0, h_0}(\mathbf{x},t), \mathbf{w}_h)
=\langle\mathbf{F},\mathbf{w}_h\rangle
-\langle\frac{d\mathbf{J}_q^{0,h_0}}{dt},\mathbf{w}_{h}\rangle,}\\[2mm]
{\displaystyle \qquad \qquad  \forall
\mathbf{w}_h\in \mathcal{W}_h(\Omega),\,\, t\in (0,T)},
\end{array}
\right.
\end{equation}
where $a(\mathbf{E}; \Psi, \varphi) $ and  $ b(\mathbf{u}, \mathbf{v}) $ have been defined above.
$\Psi_{0 h}^R $ and $\Psi_{0 h}^I $  are respectively the projection of $ \Psi_0^R $ and $ \Psi_0^I $
in the subspace $ \mathcal{U}_h(\Omega) $;
$ \mathbf{E}_{0 h} $ and $\mathbf{E}_{1 h} $ are the projections of
$ \mathbf{E}_0 $ and $\mathbf{E}_1 $ in the subspace $ \mathcal{W}_h(\Omega) $,
respectively.

\tikzstyle{long} = [rectangle, draw, fill=red!20, 
    text width=8em, text centered, rounded corners, minimum height=2em]
\tikzstyle{decision} = [ rectangle, draw, fill=red!20, 
    text width=16em, text centered,  rounded corners, node distance=3cm, inner sep=0pt]
\tikzstyle{block} = [rectangle, draw, fill=blue!20, 
    text width=8em, text centered, rounded corners, minimum height=2em]
\tikzstyle{line} = [draw, -latex']

\begin{figure}[!htbp]
\begin{tikzpicture}[node distance = 2cm, auto]
    \node[block](input){given $(\mathbf{E}_h^{n+1})^{(0)}$};
     \node[block,below of = input, node distance = 1.5cm](given){get $(\mathbf{E}_h^{n+1})^{(k)}$};
       \node [block, below of=given] (solve) { solve Schr{o}dinger equation through inner iteration};

    \node [block, below of=solve] (solve1) { compute $ \mathbf{J}_{q}^{n+1}$};
      \node [block, below of=solve1] (solve2) { solve Maxwell equation, get $(\mathbf{E}_h^{n+1})^{(k+1)}$};
        \node [block, left of=solve2, node distance=4cm] (update) { $k = k+1$};
         \node [decision, below of=solve2, node distance = 2cm] (decide) {$\|(\mathbf{E}_h^{n+1})^{(k+1)}-(\mathbf{E}_h^{n+1})^{(k)}\| > tol$};
         \node [block, below of=decide] (next) {turn to the next time step};

    \path [line] (input) -- (given);
    \path [line] (given) -- (solve);
       \path [line] (solve) -- (solve1);
    \path [line] (solve1) -- (solve2);
     \path [line] (solve2) -- (decide);
    \path [line] (decide) -| node [near start] {no} (update);
    \path [line] (update) |- (given);
    \path [line] (decide) -- node {yes}(next);
\end{tikzpicture}
\caption{The flowchart of the exterior-circle iteration}\label{fig4-1}
\end{figure}

\tikzstyle{long} = [rectangle, draw, fill=red!20,
    text width=8em, text centered, rounded corners, minimum height=2em]
\tikzstyle{decision} = [ rectangle, draw, fill=red!20,
    text width=16em, text centered,  rounded corners, node distance=3cm, inner sep=0pt]
\tikzstyle{block} = [rectangle, draw, fill=blue!20,
    text width=8em, text centered, rounded corners, minimum height=2em]
\tikzstyle{line} = [draw, -latex']

\begin{figure}[htb]
\begin{tikzpicture}[node distance = 2cm, auto]
     \node[block](input){given $(\rho_{h}^{n+1})^{(0)}$};
     \node[block, below of = input, node distance = 1.5cm](given){get $(\rho_{h}^{n+1})^{(k)}$};
       \node [block, below of=given, node distance = 1.5cm] (solve) { compute  $V_{xc}$};

    \node [block, below of=solve] (solve1) { solve the linearized schrodinger equation};
      \node [block, below of=solve1] (solve2) { compute $(\rho_{h}^{n+1})^{(k+1)}$};
        \node [block, left of=solve2, node distance=4cm] (update) { $k = k+1$};
         \node [decision, below of=solve2, node distance = 1.5cm] (decide) {$\|(\rho_{h}^{n+1})^{(k+1)}-(\rho_{h}^{n+1})^{(k)}\| > tol$};
         \node [block, below of=decide] (next) {return to the outer iteration};

    \path [line] (input) -- (given);
    \path [line] (given) -- (solve);
       \path [line] (solve) -- (solve1);
    \path [line] (solve1) -- (solve2);
     \path [line] (solve2) -- (decide);
    \path [line] (decide) -| node [near start] {no} (update);
    \path [line] (update) |- (given);
    \path [line] (decide) -- node {yes}(next);
\end{tikzpicture}
\caption{The flowchart of the interior-circle iteration}\label{fig4-2}
\end{figure}

\tikzstyle{long} = [rectangle, draw, fill=red!20,
    text width=8em, text centered, rounded corners, minimum height=2em]
\tikzstyle{decision} = [ rectangle, draw, fill=red!20,
    text width=16em, text centered,  rounded corners, node distance=3cm, inner sep=0pt]
\tikzstyle{block} = [rectangle, draw, fill=blue!20,
    text width=8em, text centered, rounded corners, minimum height=2em]
\tikzstyle{line} = [draw, -latex']

For the  semi-discrete system (\ref{eq:4-20}), we employ the Crank-Nicolson scheme
to discretize it and get the nonlinear full-discrete system. Then we use the exterior-circle
and interior-circle iterative methods to solve it, respectively. The computational procedure is
briefly described. First, we apply the exterior iterative method to solve the full-discrete
system of (\ref{eq:4-20}). Second, for each time step, we combine the self-consistent iterative method (SCF)
and the simple mixed method to solve the discrete system of the time-dependent Schr\"{o}dinger equation.
The detailed procedures are displayed in Figs.~\ref{fig4-1} and ~\ref{fig4-2}, respectively.

\begin{rem}
As for the convergence and stable analysis of the above numerical algorithms, it is a hard task. Due to space limitations,
we will study these problems in another paper.
\end{rem}

\section{Numerical tests}\label{sec-5}

To validate the developed multiscale algorithm in this paper, we
present numerical simulations for the following case studies.

\begin{exam}\label{exam5-1}
We consider the following Maxwell-Schr\"{o}dinger system
with rapidly oscillating discontinuous coefficients:
\begin{equation}\label{eq:5-1}
\left\{
\begin{array}{@{}l@{}}
{\displaystyle i\frac{\partial \Psi^\varepsilon}{\partial t}=
-\nabla \cdot \big(A(\frac{\mathbf{x}}{\varepsilon})\nabla
\Psi^\varepsilon\big)
+\big(V_c(\frac{\mathbf{x}}{\varepsilon})-\mathbf{E}^\varepsilon\cdot \widehat{\bm\zeta}
+V_{xc}[\rho^\varepsilon]\big)
\Psi^\varepsilon,\,\, (\mathbf{x},t)\in
\Omega\times(0,T),}\\[2.2mm]
{\displaystyle \eta(\frac{\mathbf{x}}{\varepsilon})\frac{\partial^2 \mathbf{E}^\varepsilon}{\partial t^2}
+\mathbf{curl}\/ \big((\mu(\frac{\mathbf{x}}{\varepsilon}))^{-1}\mathbf{curl}\/ \mathbf{E}^\varepsilon\big)=\mathbf{F}
-\frac{\partial \mathbf{J}_q^\varepsilon}{\partial t},\,\,\,\nabla\cdot \mathbf{F}=0,\,\,\,(\mathbf{x},t)\in
\Omega\times(0,T),}\\[2.2mm]
{\displaystyle  \widehat{\bm\zeta}=-\mathbf{x},\,\,\rho^\varepsilon=N{\vert\Psi^\varepsilon\vert}^2
,\,\,\mathbf{J}_q^\varepsilon=iN\big[(\overline{\Psi}^\varepsilon) A(\frac{\mathbf{x}}{\varepsilon})\nabla\Psi^\varepsilon-
\Psi^\varepsilon  A(\frac{\mathbf{x}}{\varepsilon})\nabla\overline{\Psi}^\varepsilon\big],}\\[2.2mm]
{\displaystyle \Psi^\varepsilon(\mathbf{x},t)=0,\quad \mathbf{E}^\varepsilon(\mathbf{x},t)\times \mathbf{n}=0,\quad
(\mathbf{x},t)\in \partial \Omega \times (0,T),}\\[2.2mm]
{\displaystyle \Psi^\varepsilon(\mathbf{x},0)=\Psi_0(\mathbf{x}),\quad
\mathbf{E}^\varepsilon(\mathbf{x}, 0)=\mathbf{E}_0(\mathbf{x}),\quad
\frac{\partial \mathbf{E}^\varepsilon(\mathbf{x},0)}{\partial t}=\mathbf{E}_1(\mathbf{x}).}
\end{array}
\right.
\end{equation}
In this example, there is not the exchange-correlation potential in (\ref{eq:5-1}), i.e. $ V_{xc} = 0 $.
and we take the matrix $ \eta(\frac{\displaystyle \mathbf{x}}{\displaystyle \varepsilon})\equiv I_3 $, where
$ I_3 $ is an $ 3\times 3 $ identity matrix. A whole domain $\Omega $ and the unit cell $ Q $ are shown in Fig.~\ref{fig5-1}:(a) and (b), respectively.
We take $ \varepsilon=\frac{\displaystyle 1}{\displaystyle 8} $, $N=10$, $T = 0.5$,
 $\mathbf{f}(\mathbf{x},t)=(f_1(\mathbf{x},t), f_2(\mathbf{x},t), f_3(\mathbf{x},t))^T $,
where $ f_1=1000(1-\cos(\pi t))(y^{2}+1), f_2=1000(1-\cos(\pi t))(z^{2}+1), f_3= 1000(1-\cos(\pi t))(x^{2}+1)$, $ \mathbf{E}_0(\mathbf{x})=0 $, $ \mathbf{E}_1(\mathbf{x})=0 $.
Let $ A(\frac{\displaystyle \mathbf{x}}{\displaystyle \varepsilon})=(a_{ij}(\frac{\displaystyle \mathbf{x}}{\displaystyle \varepsilon})) $,
$ \mu(\frac{\displaystyle \mathbf{x}}{\displaystyle \varepsilon})=(\mu_{ij}(\frac{\displaystyle \mathbf{x}}{\displaystyle \varepsilon})) $.
 Here $ \delta_{ij} $ is the Kronecker symbol.
\begin{gather*}
V_{c}(\frac{\displaystyle \mathbf{x}}{\displaystyle \varepsilon})=\left\{
\begin{array}{l}
 0, \,  \hbox{in each cube}\\
 1, \,  \hbox{others}
\end{array}
\right.
\end{gather*}
\begin{gather*}
\mbox{\bf Case 5.1.1.} \quad a_{ij}(\frac{\displaystyle \mathbf{x}}{\displaystyle \varepsilon})
=\left\{
\begin{array}{l}
 0.1 \delta_{ij}, \,  \rm{in \, each \, cube}\\
 \delta_{ij}, \,  \rm{others}
\end{array}
\right.
\quad \mu_{ij}(\frac{\displaystyle \mathbf{x}}{\displaystyle \varepsilon})
=\left\{
\begin{array}{l}
 \delta_{ij}, \,  \rm{in \,each\, cube}\\
 0.01\delta_{ij}, \,  \rm{others}
\end{array}
\right.\nonumber \\
\mbox{\bf Case 5.1.2. }\quad a_{ij}(\frac{\displaystyle \mathbf{x}}{\displaystyle \varepsilon})
=\left\{
\begin{array}{l}
 0.05\delta_{ij}, \,  \rm{in\, each \,cube}\\
 \delta_{ij}, \,  \rm{others}
\end{array}
\right. \quad \mu_{ij}(\frac{\displaystyle \mathbf{x}}{\displaystyle \varepsilon})
=\left\{
\begin{array}{l}
 \delta_{ij}, \,  \rm{in \,each \,cube}\\
 0.005\delta_{ij}, \,  \rm{others}
\end{array}
\right.\nonumber \\
\mbox{\bf Case 5.1.3.}\quad a_{ij}(\frac{\displaystyle \mathbf{x}}{\displaystyle \varepsilon})
=\left\{
\begin{array}{l}
 0.02\delta_{ij}, \,  \rm{in \,each \,cube}\\
 \delta_{ij}, \,  \rm{others}
\end{array}
\right. \quad \mu_{ij}(\frac{\displaystyle \mathbf{x}}{\displaystyle \varepsilon})
=\left\{
\begin{array}{l}
 \delta_{ij}, \,  \rm{in \,each \,cube}\\
 0.0025\delta_{ij}, \,  \rm{others}.
\end{array}
\right.\nonumber \\
\end{gather*}
\end{exam}
\begin{figure}[t!]
\centering
{\tiny((a))}\includegraphics[width=4.8cm,height=4.8cm]{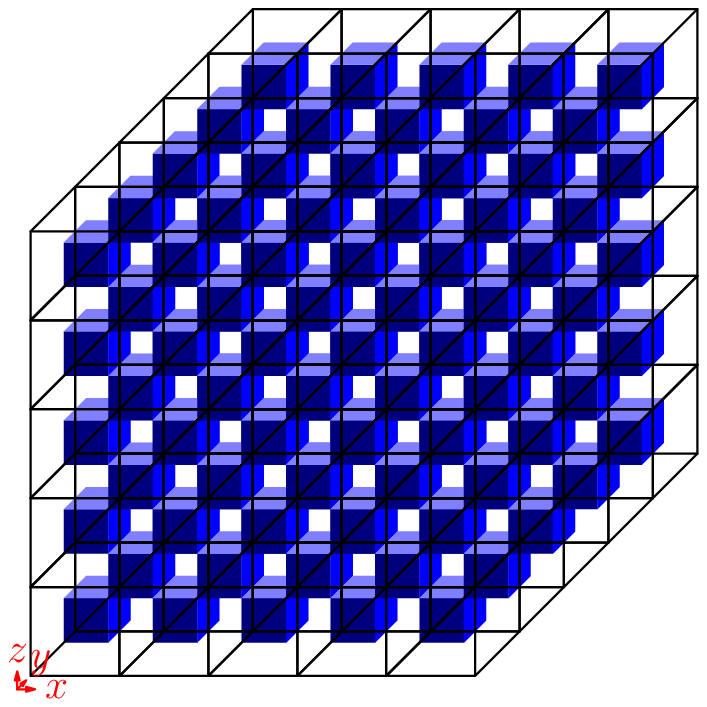}~{\tiny((b))}
\includegraphics[width=4.8cm,height=4.8cm]{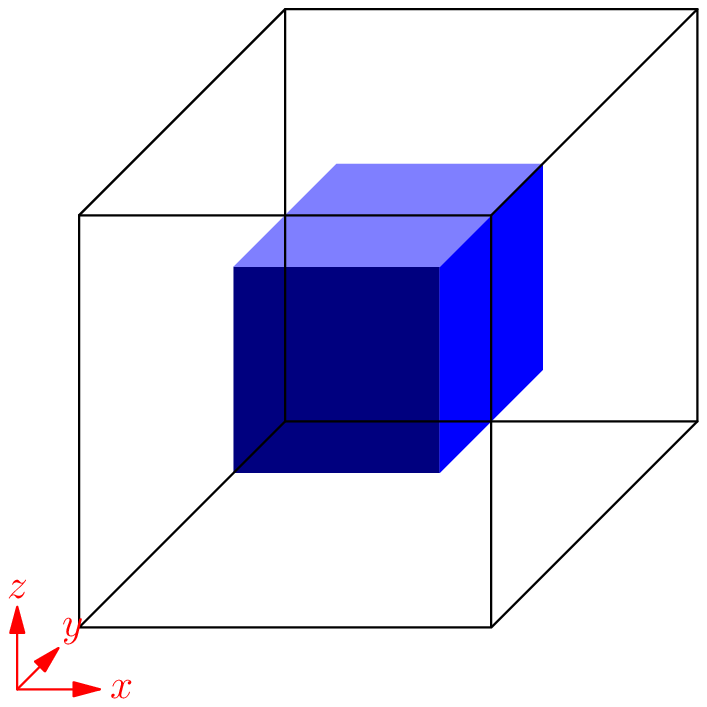}
\caption{{\rm(a)}~A whole domain $ \Omega $ of heterogenous materials
with a periodic microstructure;\/ {\rm(b)}~the reference cell $Q=(0,1)^3$.}\label{fig5-1}
\end{figure}

To determine the initial wave function $ \Psi_0(\mathbf{x}) $, we need to solve the time-independent Schr\"{o}dinger equation.
We take the wave function of the ground state as the initial wave function
$ \Psi_0(\mathbf{x}) $ of the Maxwell-Schr\"{o}dinger system (\ref{eq:5-1}).
For more details, see \cite{Zhang-L}.

In order to demonstrate the numerical accuracy of the present method,
the exact solution $ (\Psi^\varepsilon(\mathbf{x},t), \mathbf{E}^\varepsilon(\mathbf{x},t))$
of Maxwell-Schr\"{o}dinger system (\ref{eq:5-1})
must be available. Since the elements of the coefficient matrices $
(a_{ij}(\frac{\displaystyle \mathbf{x}}{\displaystyle \varepsilon})) $ and
$ (\mu_{ij}(\frac{\displaystyle \mathbf{x}}{\displaystyle \varepsilon})) $ are discontinuous,
in general, it is an extremely difficult task or even impossible to
obtain the exact solution. Here, we replace the exact solution $ (\Psi^\varepsilon(\mathbf{x},t), \mathbf{E}^\varepsilon(\mathbf{x},t))$
by the numerical solution in a very fine mesh and at a small time step. It should be emphasized that this step is not
necessary in the real applications. The computational costs for solving the Schr\"{o}dinger equation and the Maxwell's equations are listed in Table~\ref{tab5-1} and
Table ~\ref{tab5-2}, respectively. The time step $ \Delta t=0.0025 $.

\begin{table}[htb]
\caption{Comparison of computational costs for solving the Schr\"{o}dinger equation}\label{tab5-1}
\begin{center}
\begin{tabular}{|c|c|c|c|}\hline
 & original problem &  cell problems & homogenized equation \\
\hline  Dof & 2478213 & 180135 & 76410 \\
\hline  Number of elements & 13573655  & 1034688 & 403590 \\
 \hline
\end{tabular}
\end{center}
\end{table}
\begin{table}[htb]
\caption{Comparison of computational costs for solving the Maxwell's equations}\label{tab5-2}
\begin{center}
\begin{tabular}{|c|c|c|c|}\hline
 & original problem &  cell problems & homogenized equations \\
\hline  Dof & 16125013 & 1229702 & 486888 \\
\hline  Number of elements & 13573655 & 1034688 & 403590 \\
 \hline
\end{tabular}
\end{center}
\end{table}

For simplicity, without confusion $\rho^{\varepsilon}= N\vert\Psi^{\varepsilon}\vert^2$ and
$\mathbf{E}^{\varepsilon}$ denote the numerical solutions of the density function
and the electric field  for the Maxwell-Schr\"{o}dinger system (\ref{eq:5-1}) in a fine mesh and at a time step
$ \Delta t=0.0025 $, respectively,
which are regarded as the reference solutions of the problem (\ref{eq:5-1}).
Here $\rho^{0}=N\vert\Psi^{0}\vert^2$ and
$\mathbf{E}^{0}$ are respectively the numerical solutions of the density function
and the electric field for the associated homogenized Maxwell-Schr\"{o}dinger system in a coarse mesh
and at a time step $ \Delta t =0.0025 $.
$\rho_1^{\varepsilon}=N\vert\Psi_1^{\varepsilon}\vert^2$ and
$\rho_2^{\varepsilon}=N\vert\Psi_2^{\varepsilon}\vert^2$ are the first-order and the second-order multiscale
solutions of the density function, respectively.  $\mathbf{E}^{\varepsilon,(1)}$  and
$\mathbf{E}^{\varepsilon,(2)} $ are the first-order and the
second-order multiscale solutions of  the electric field, respectively. Set $e_0=\rho^\varepsilon-\rho^0$, $
e_1=\rho^{\varepsilon}-\rho_1^{\varepsilon}$, $
e_2=\rho^{\varepsilon}-\rho_2^{\varepsilon}$,
$\mathbf{e}_0=\mathbf{E}^{\varepsilon}-\mathbf{E}^{0}$,
$\mathbf{e}_1=\mathbf{E}^{\varepsilon}-\mathbf{E}^{\varepsilon, (1)}$,
$\mathbf{e}_2=\mathbf{E}^{\varepsilon}-\mathbf{E}^{\varepsilon, (2)}$. For convenience, we introduce
the following notation. $\|\rho\|_0 =\|\rho\|_{L^{2}(0,T; \/L^2(\Omega))} $,
$ \|\rho\|_1 = \|\rho\|_{L^{2}(0,T; \/H^1(\Omega))} $,
$\|\mathbf{E}\|_{(0)}=\|\mathbf{E}\|_{L^{2}(0,T; \/(L^2(\Omega))^3)} $,
$\|\mathbf{E}\|_{(1)}=\|\mathbf{E}\|_{L^{2}(0,T; \/\mathbf{H}(\mathbf{curl};\Omega))} $.

The computational results for the density function and the electric field in Example~\ref{exam5-1} are illustrated
in Table~\ref{tab5-3} and Table~\ref{tab5-4}, respectively.

\begin{table}[htb]
\caption{The computational results for the density function in Example~\ref{exam5-1}}\label{tab5-3}
\begin{center}
\begin{tabular}{|c|c|c|c|c|c|c|}
   \hline
    &$\frac{\|e_0\|_{0}}{\|\rho^{\varepsilon}\|_{0}}$
    & $\frac{\|e_1\|_{0}}{\|\rho^{\varepsilon}\|_{0}}$
    & $\frac{\|e_2\|_{0}}{\|\rho^{\varepsilon}\|_{0}}$
    & $\frac{\|e_0\|_{1}}{\|\rho^{\varepsilon}\|_{1}}$
    & $\frac{\|e_1\|_{1}}{\|\rho^{\varepsilon}\|_{1}}$
    & $\frac{\|e_2\|_{1}}{\|\rho^{\varepsilon}\|_{1}}$ \\
   \hline
    Case 5.1.1 & 0.021350 & 0.012342 & 0.005139 & 0.224995 & 0.125434 & 0.036561  \\
    \hline
    Case 5.1.2 & 0.031112 & 0.024987 &  0.005740 & 0.314168 & 0.248844 & 0.041702  \\
    \hline
    Case 5.1.3 & 0.073949 & 0.071766 & 0.012427& 0.536015 & 0.516680 & 0.081628  \\
    \hline
  \end{tabular}
  \end{center}
\end{table}
\begin{table}[htb]
\caption{The computational results for the electric field in Example~\ref{exam5-1}}\label{tab5-4}
\begin{center}
\begin{tabular}{|c|c|c|c|c|c|c|}
   \hline
    &$\frac{\|\mathbf{e}_0\|_{(0)}}{\|\mathbf{E}^{\varepsilon}\|_{(0)}}$
    &$\frac{\|\mathbf{e}_1\|_{(0)}}{\|\mathbf{E}^{\varepsilon}\|_{(0)}}$
    &$\frac{\|\mathbf{e}_2\|_{(0)}}{\|\mathbf{E}^{\varepsilon}\|_{(0)}}$
    &$\frac{\|\mathbf{e}_0\|_{(1)}}{\|\mathbf{E}^{\varepsilon}\|_{(1)}}$
    &$\frac{\|\mathbf{e}_1\|_{(1)}}{\|\mathbf{E}^{\varepsilon}\|_{(1)}}$
    &$\frac{\|\mathbf{e}_2\|_{(1)}}{\|\mathbf{E}^{\varepsilon}\|_{(1)}}$\\
    \hline
    Case 5.1.1 & 0.124387 & 0.079315 & 0.039245  & 1.147215  &1.036499  & 0.793641  \\
    \hline
    Case 5.1.2 & 0.166034  & 0.159700  &0.067750  & 1.595115  & 1.239791  & 0.850183 \\
    \hline
    Case 5.1.3 & 0.325069 & 0.321475 & 0.129570  & 2.719812& 2.213347 & 0.889329 \\
    \hline
  \end{tabular}
  \end{center}
\end{table}

The evolution of the relative errors of the density function
in $L^2(\Omega) $-norm and in
$H^1(\Omega) $-norm with respect to time t in Case 5.1.1 is displayed in Fig.~\ref{fig5-2}.
The evolution of the relative errors of the electric field
in $(L^2(\Omega))^3$-norm and in $\mathbf{H}(\mathbf{curl};\Omega) $-norm with
respect to time t in Case 5.1.2 is illustrated in Fig.~\ref{fig5-3}.

\begin{figure}
\begin{center}
\includegraphics[width=6cm,height=6cm] {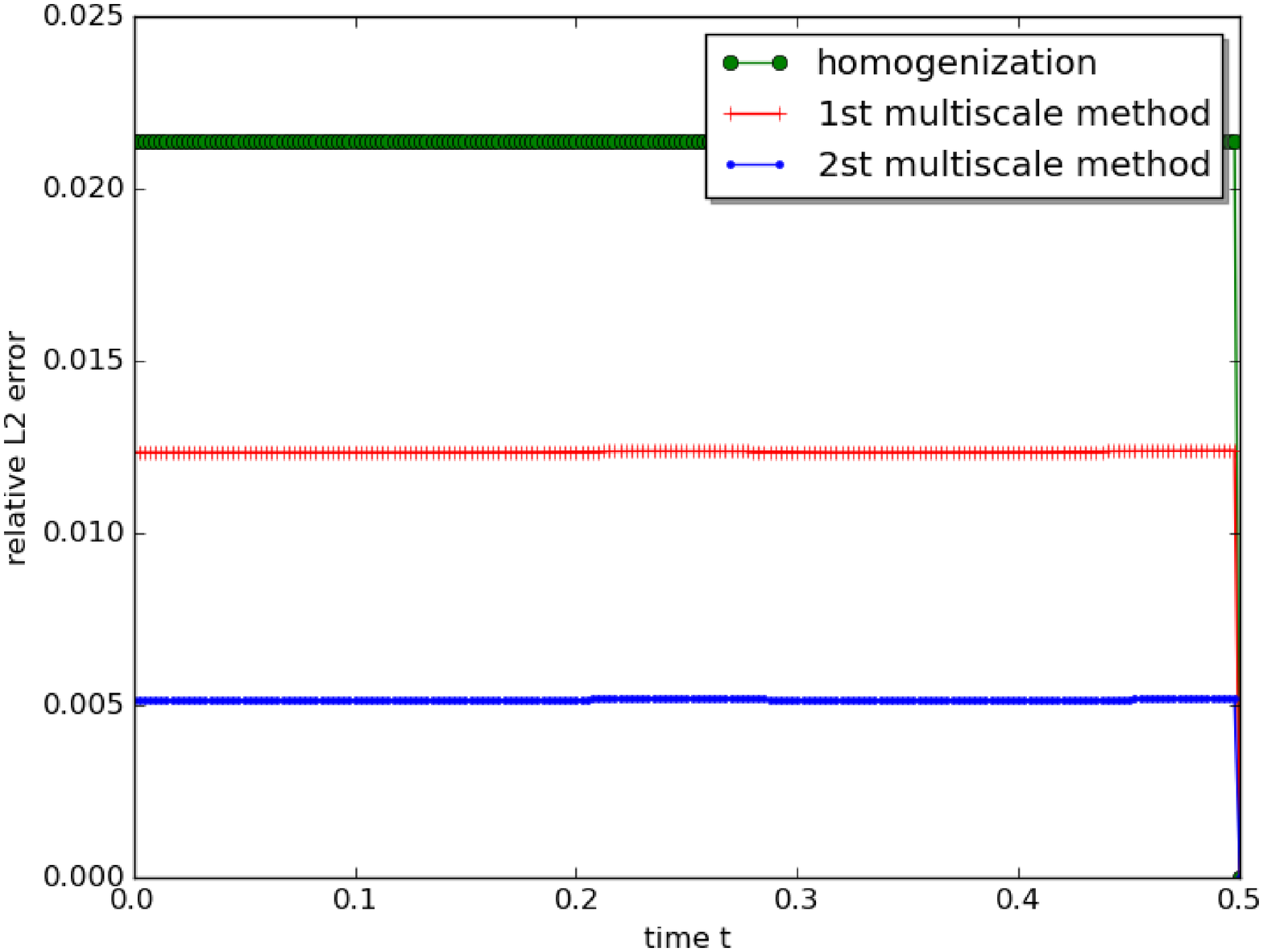}
\tiny{}\includegraphics[width=6cm,height=6cm] {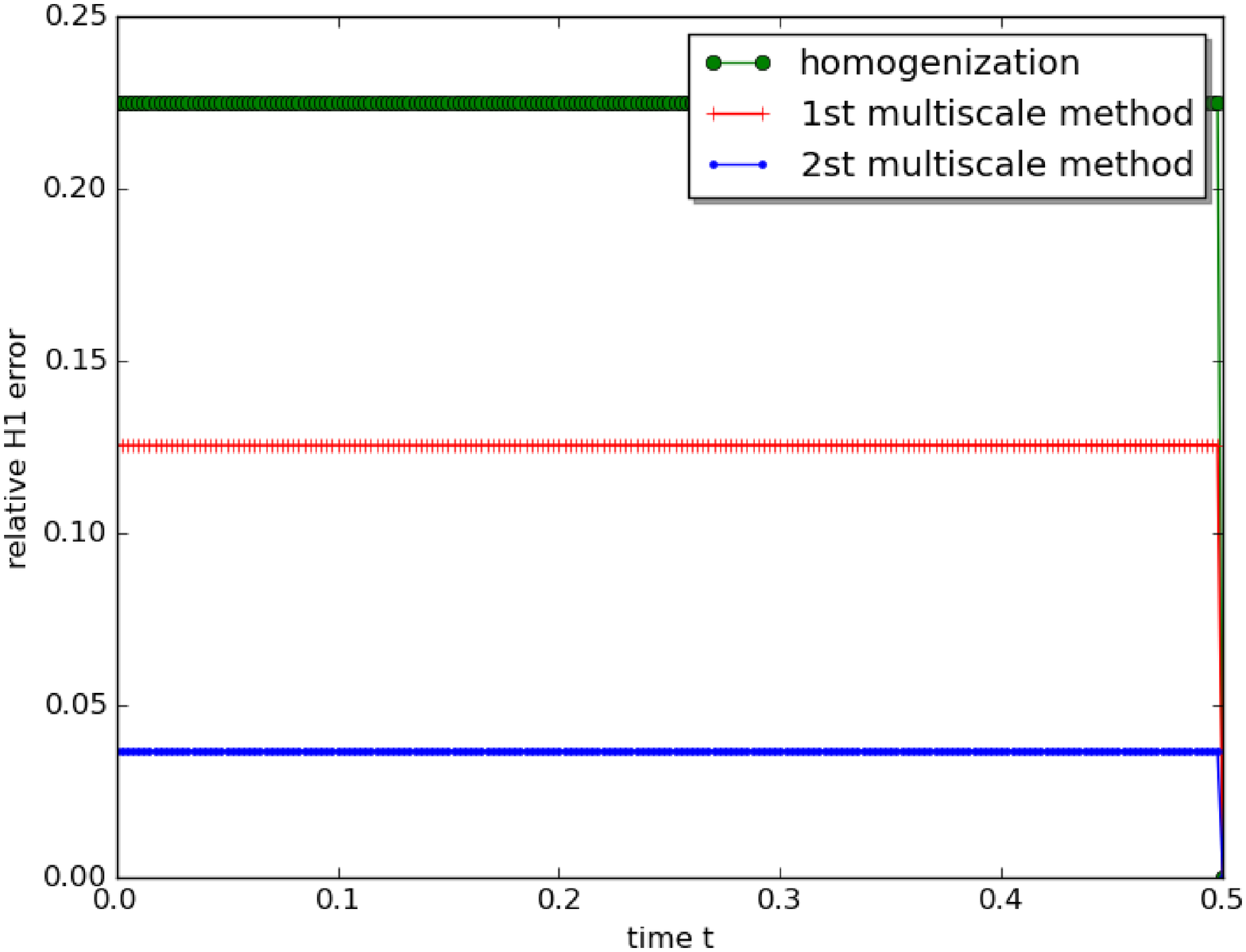}
\caption{Evolution of relative errors of the density function in the$\;
L^2(\Omega) $-norm (left) and in the $\;H^1(\Omega) $-norm
(right) in Case 5.1.1.}\label{fig5-2}
\end{center}
\end{figure}

\begin{figure}
\begin{center}
\includegraphics[width=6cm,height=6cm] {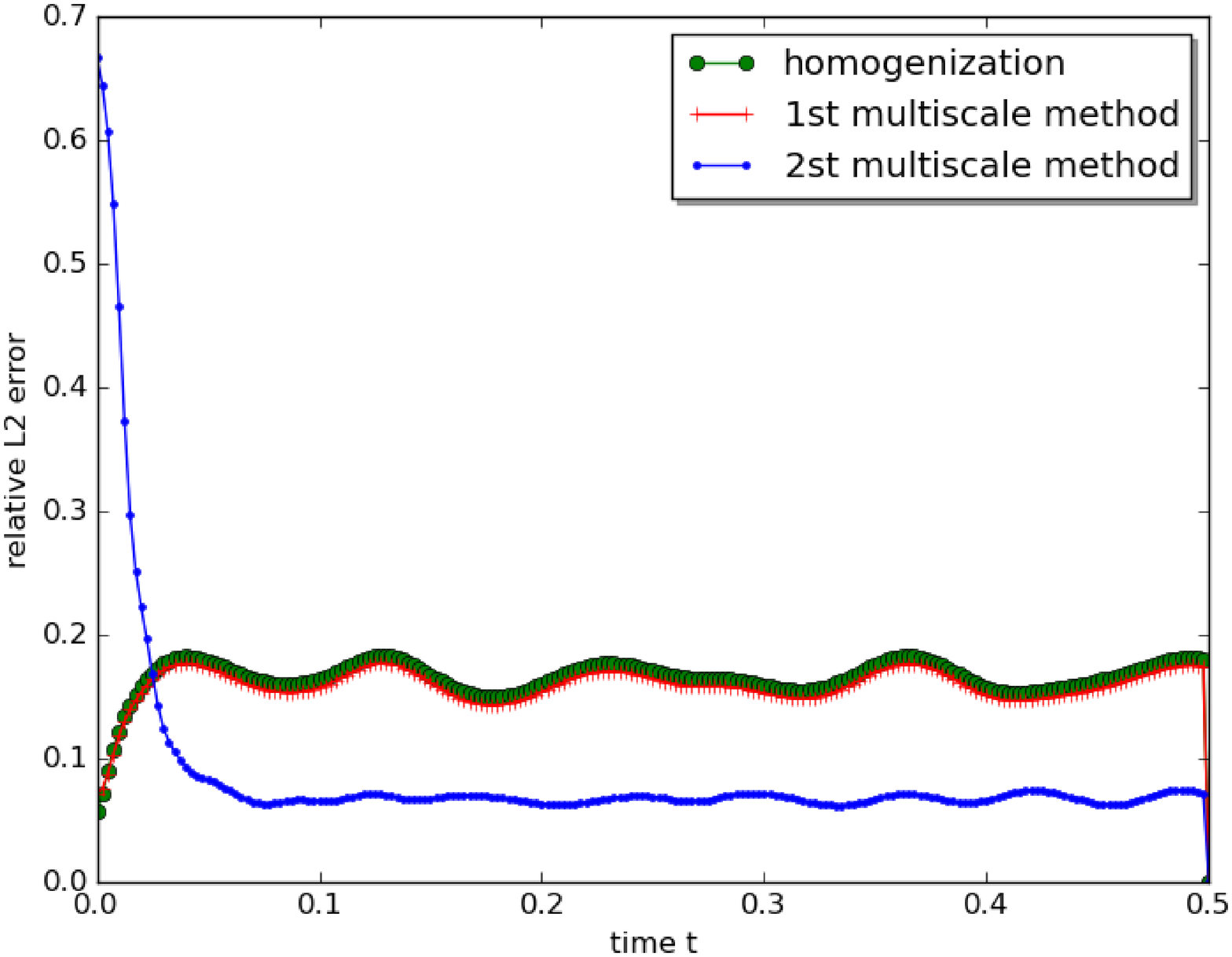}
\tiny{}\includegraphics[width=6cm,height=6cm] {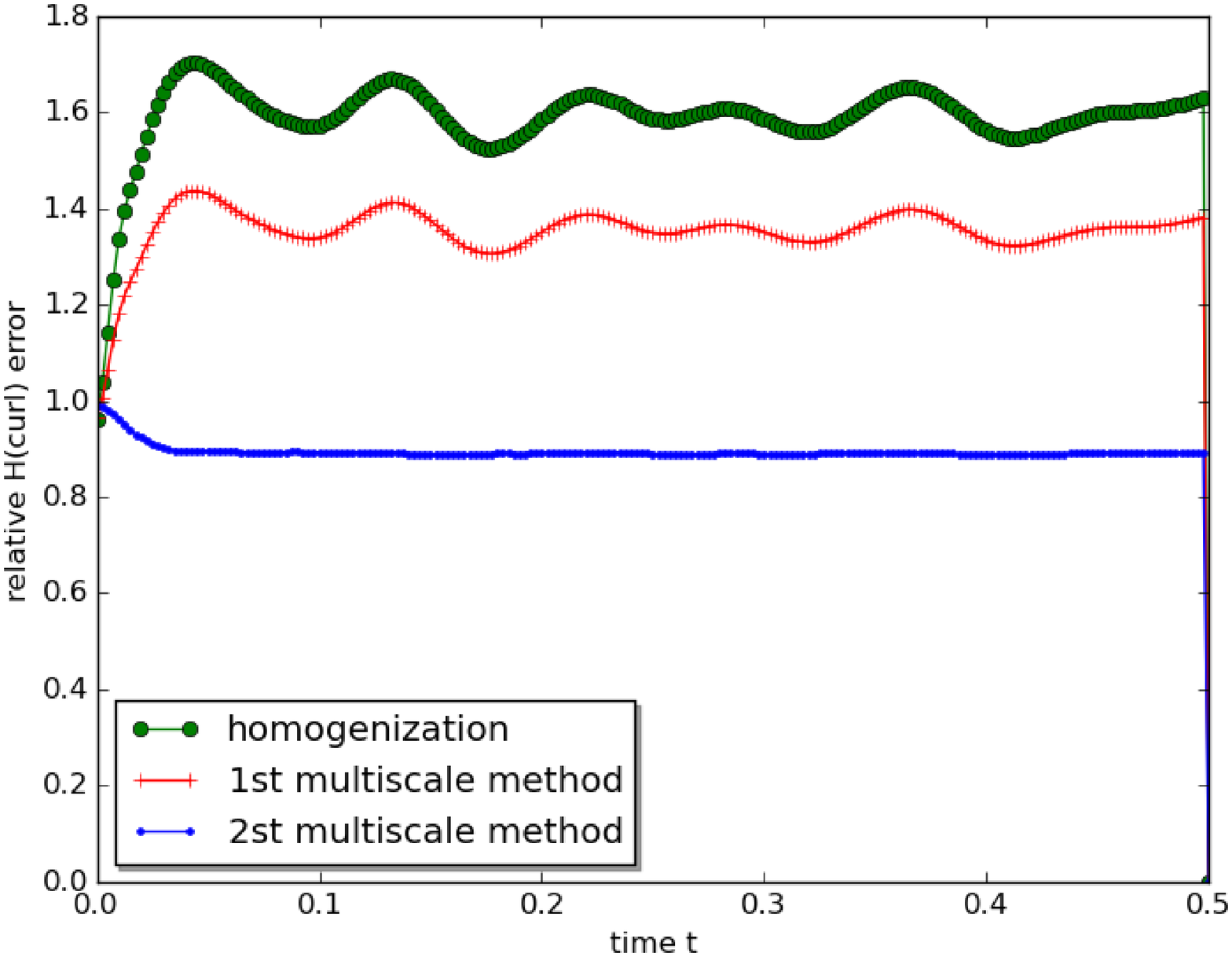}
\caption{Evolution of relative errors of the electric field in the$\;
(L^2(\Omega))^3 $-norm (left) and in the $\;\mathbf{H}(\mathbf{curl};\Omega) $-norm
(right) in Case 5.1.2.}\label{fig5-3}
\end{center}
\end{figure}

The computational results based on the homogenization method, the first-order and the second-order
multiscale methods for the density function and the electric field on the intersection $ x_3=0.45 $ and at
$ T=0.3 $ in Case 5.1.1, 5.1.2 and 5.1.3 are shown in Figs.~\ref{fig5-4}, \ref{fig5-5} and \ref{fig5-6},
respectively.

\begin{figure}
\begin{center}
\tiny{(a)}\includegraphics[width=5.5cm,height=4cm] {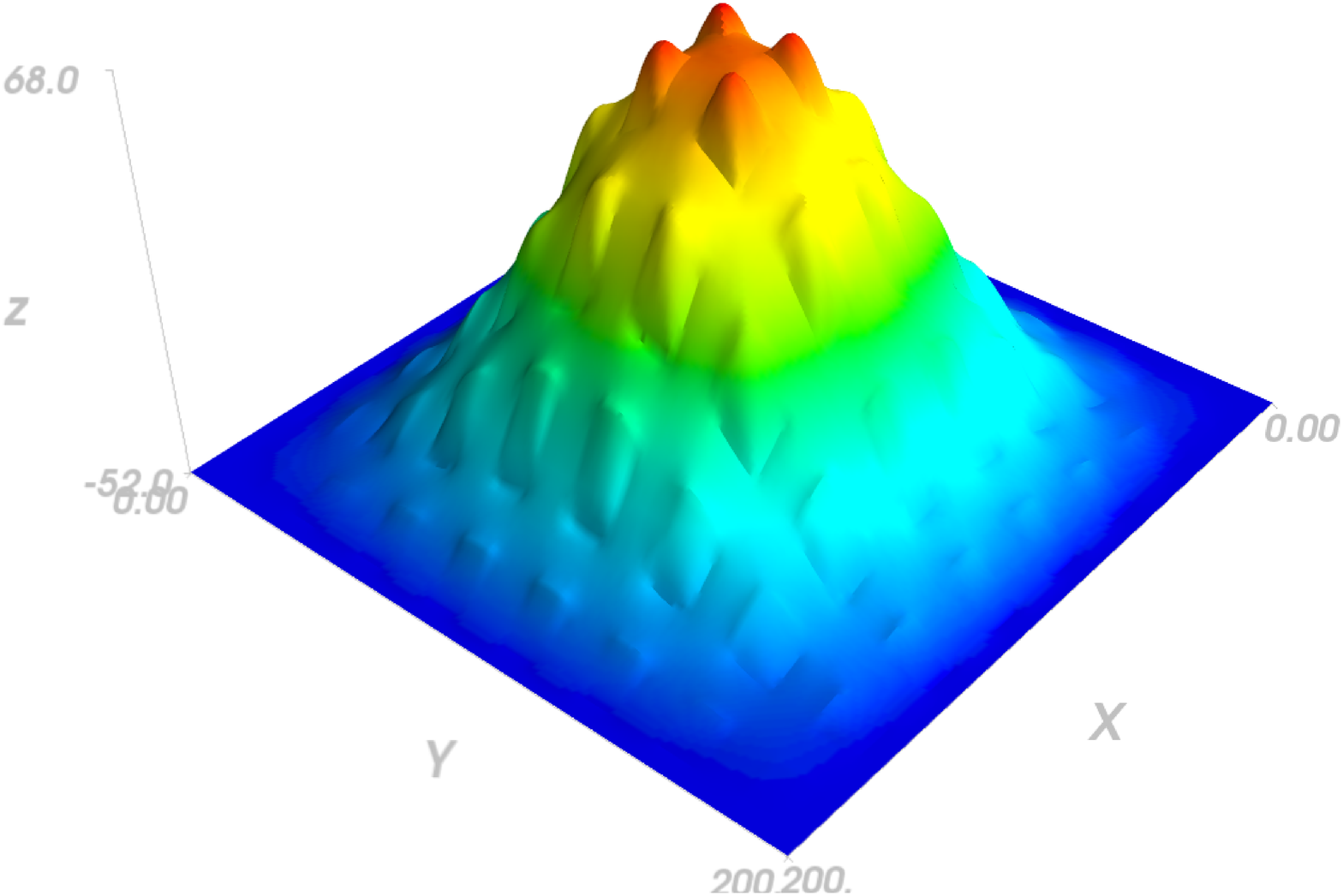}
\tiny{(b)}\includegraphics[width=5.5cm,height=4cm] {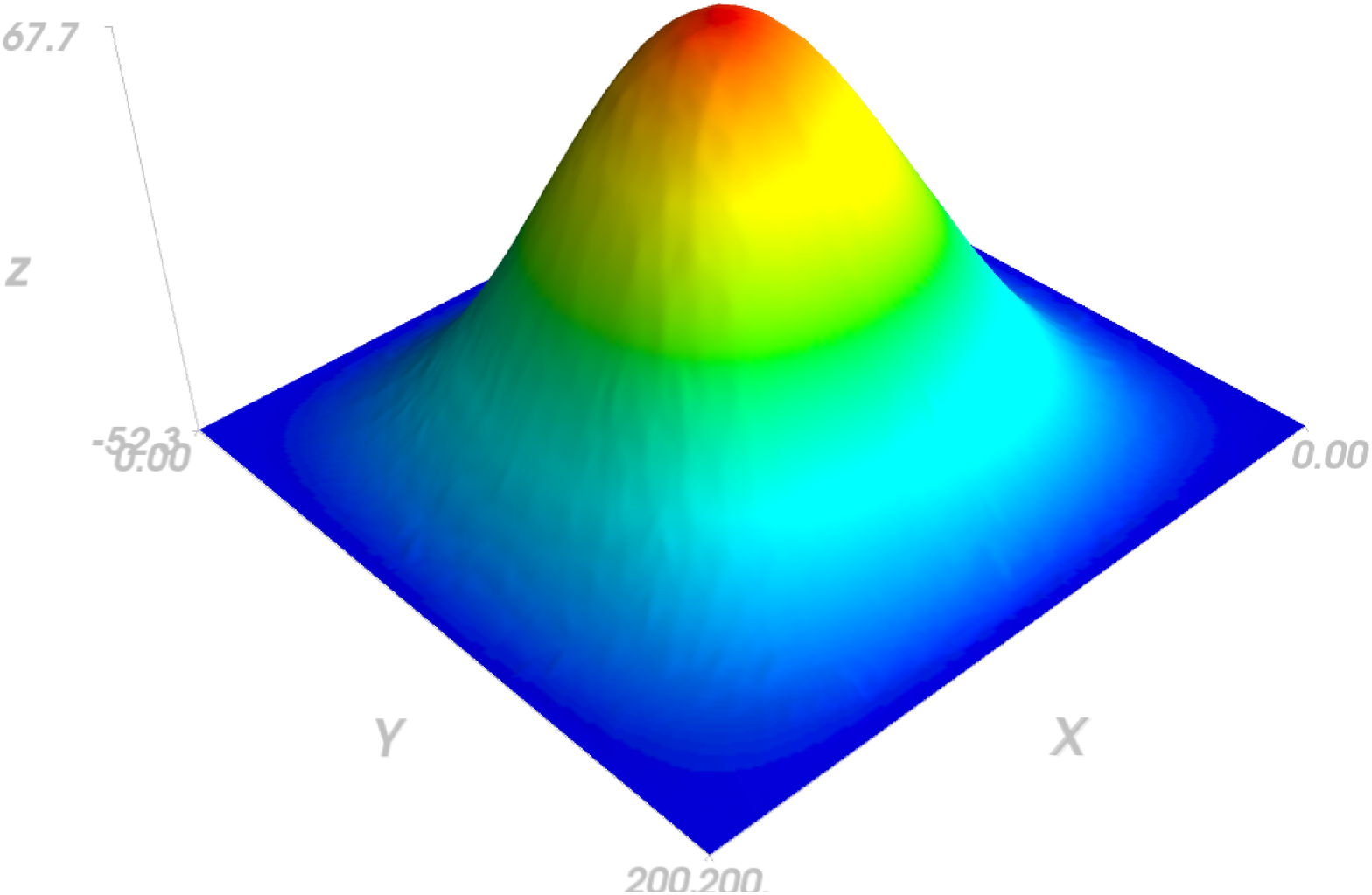}
\tiny{(c)}\includegraphics[width=5.5cm,height=4cm] {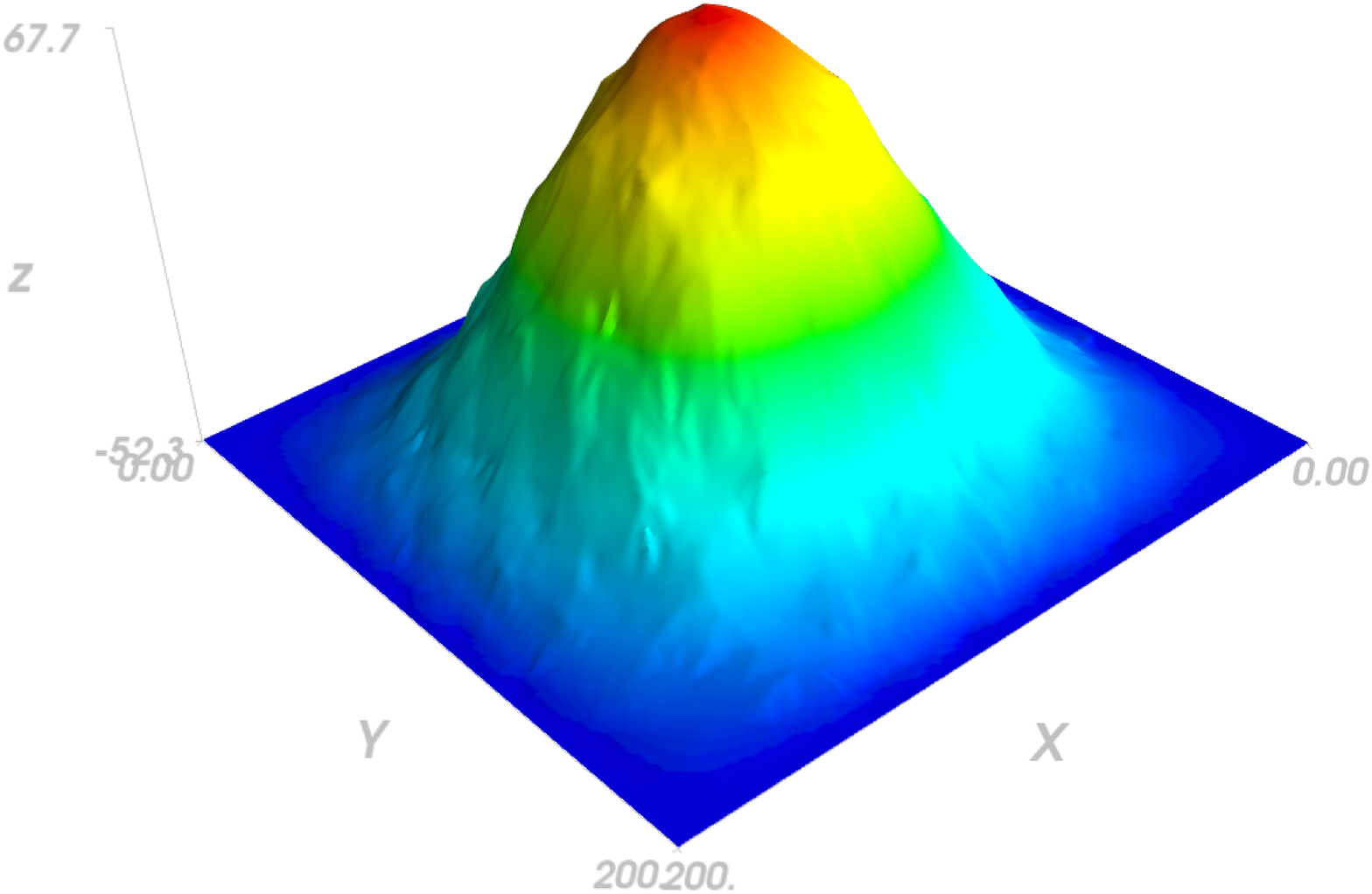}
\tiny{(d)}\includegraphics[width=5.5cm,height=4cm] {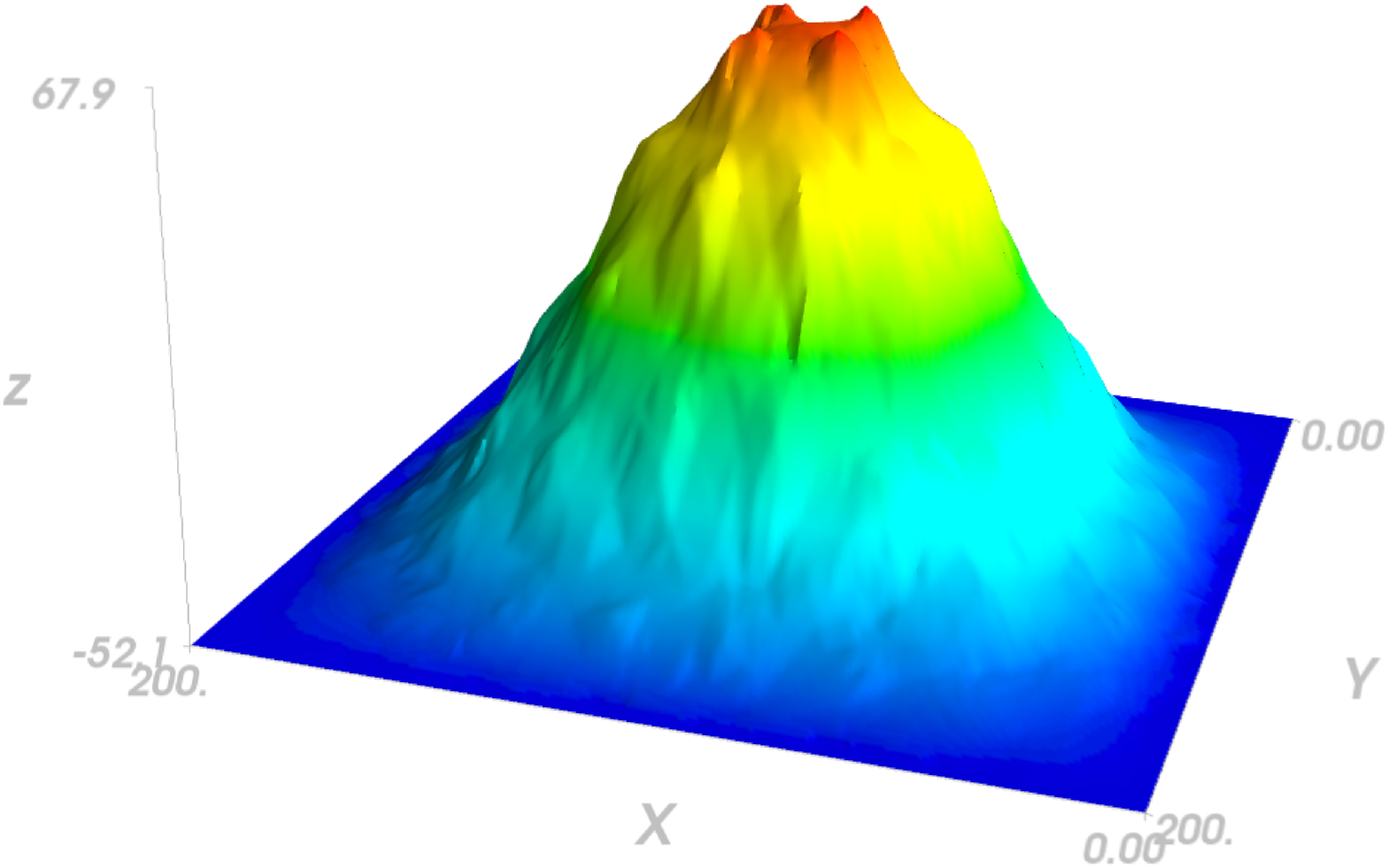}
\tiny{(e)}\includegraphics[width=5.5cm,height=4cm] {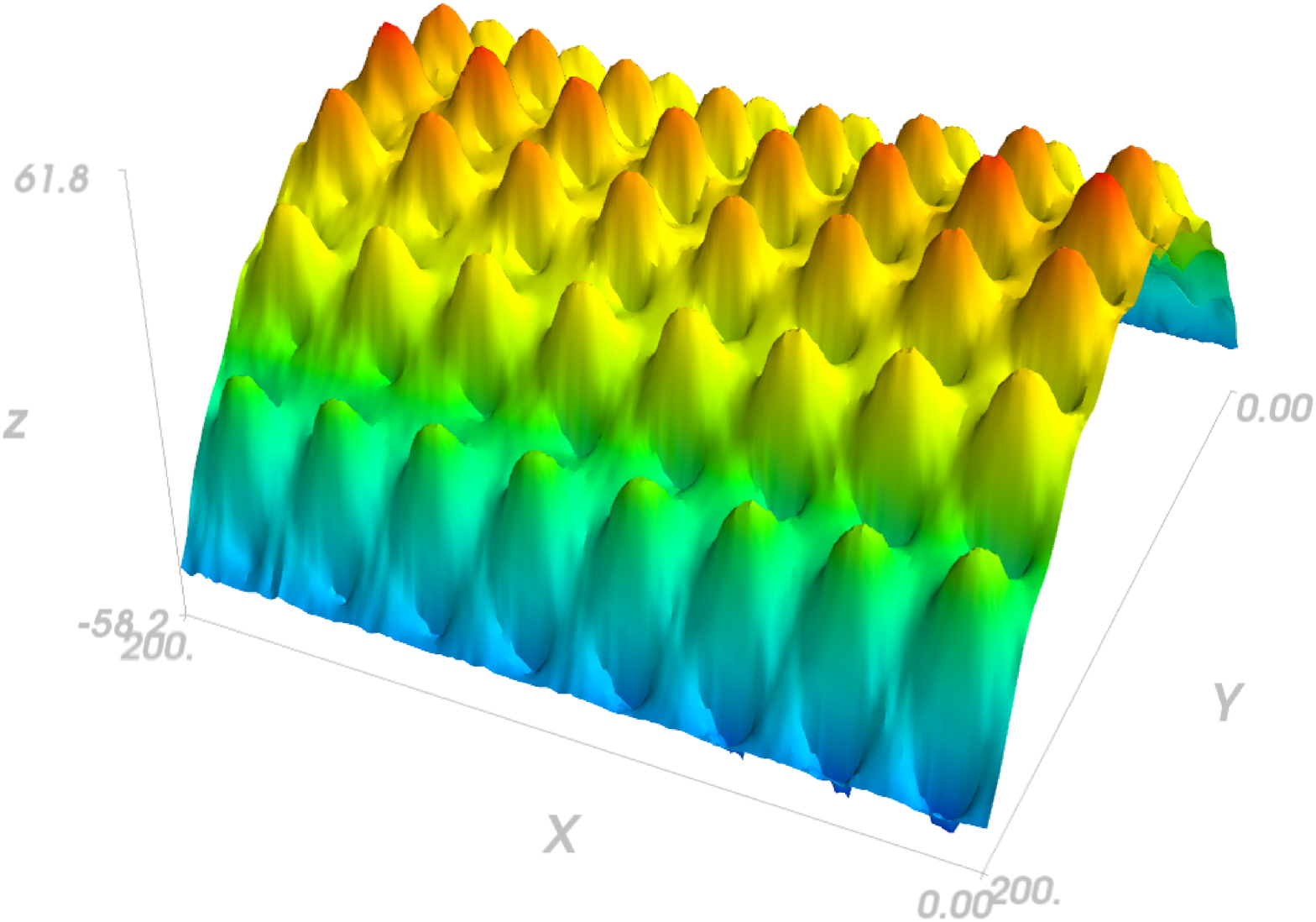}
\tiny{(f)}\includegraphics[width=5.5cm,height=4cm] {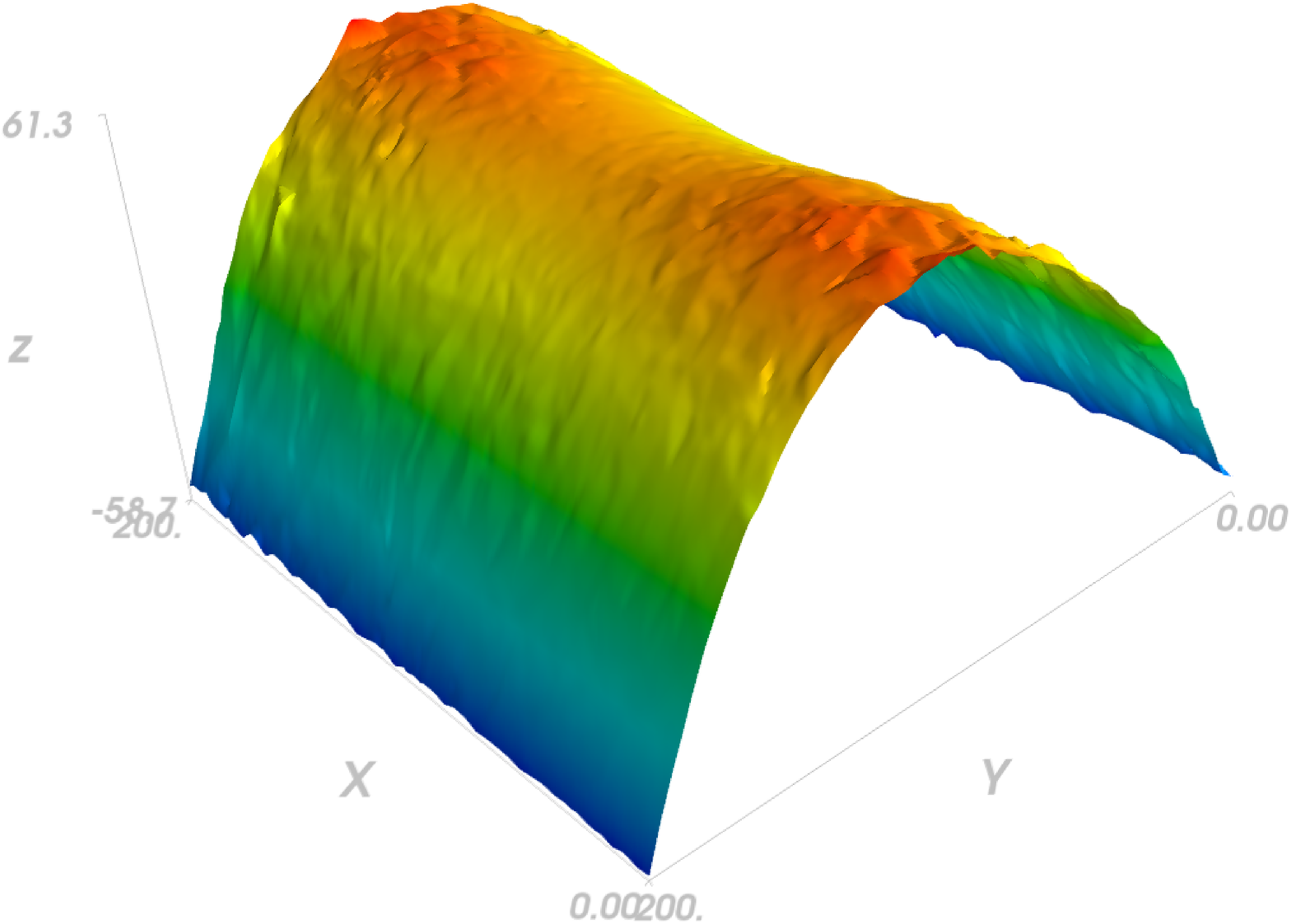}
\tiny{(g)}\includegraphics[width=5.5cm,height=4cm] {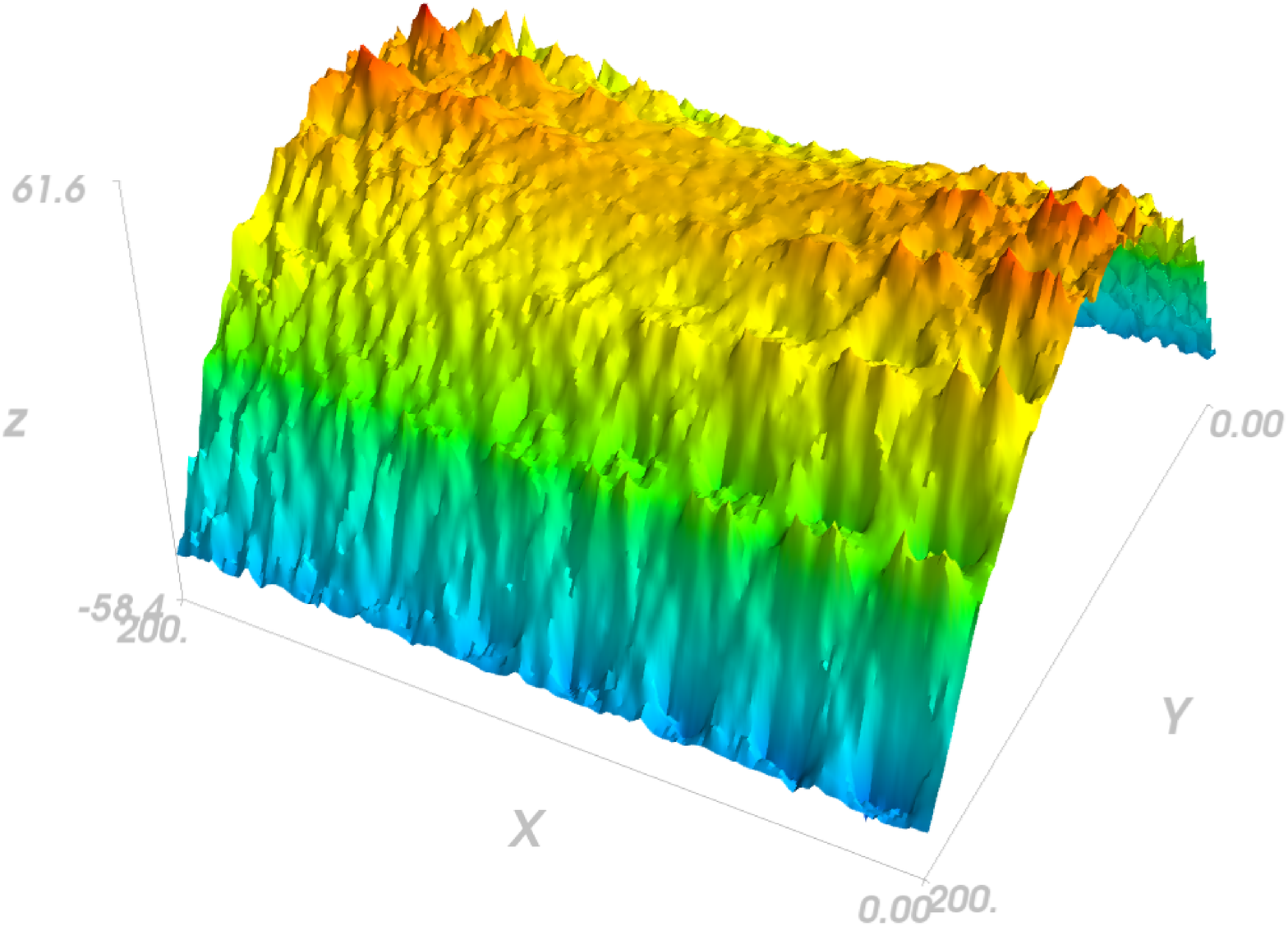}
\tiny{(h)}\includegraphics[width=5.5cm,height=4cm] {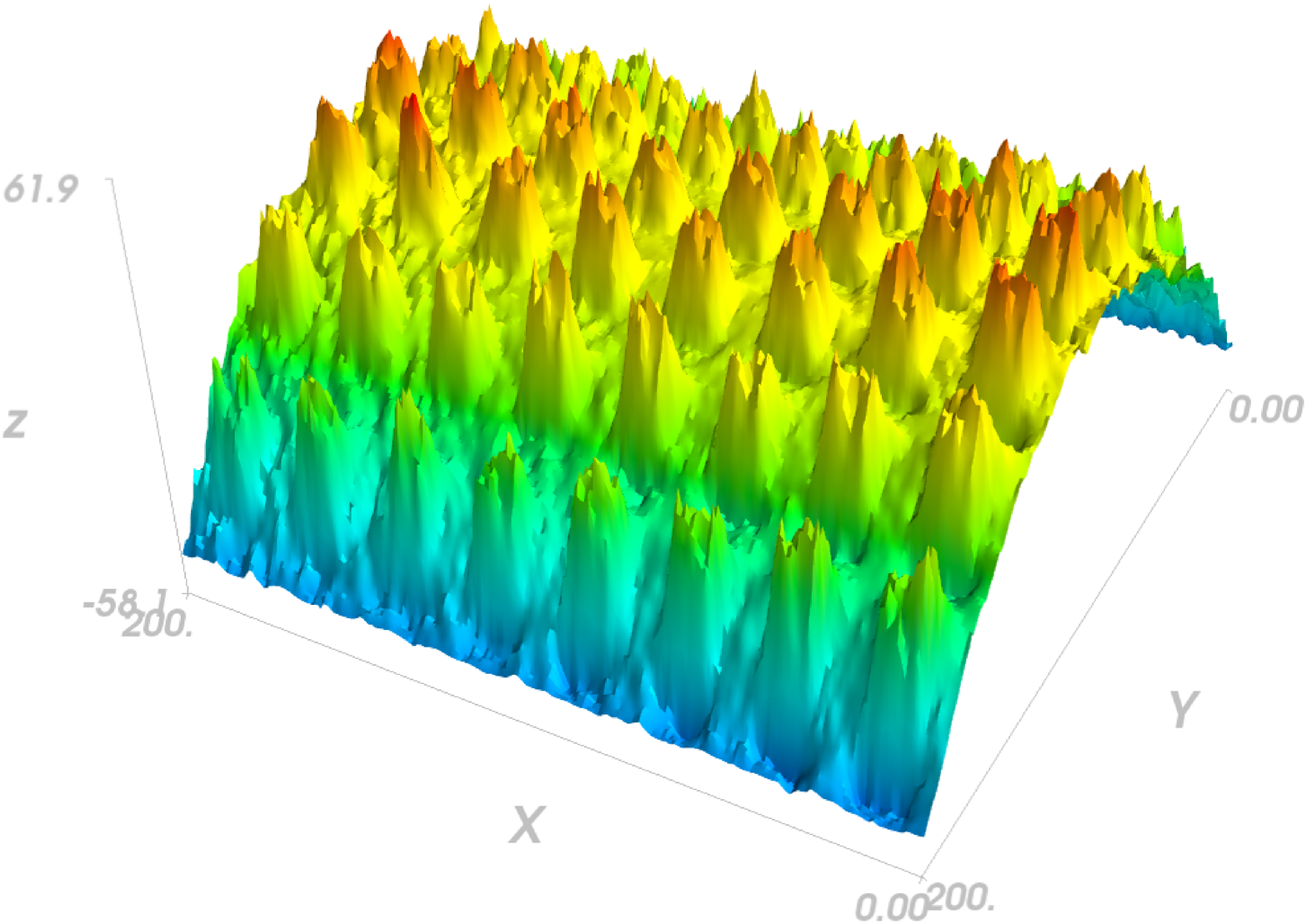}

\caption{Case 5.1.1: the density function and the first component of the electric field on the intersection $ x_3=0.45 $ and at $ T = 0.3 $.
(a) the reference solution $ \rho^\varepsilon(\mathbf{x},t)$ in a fine mesh;\, (b) the homogenized solution $
\rho^0(\mathbf{x},t)$ in a coarse mesh;\, (c) the first-order multiscale solution
$\rho_1^{\varepsilon}(\mathbf{x},t)$; \, (d)the second-order multiscale solution
$\rho_2^{\varepsilon}(\mathbf{x},t)$; \,(e) the reference solution $ \mathbf{E}^\varepsilon(\mathbf{x},t)$ in a fine mesh;\,
(f) the homogenized solution $\mathbf{E}^0(\mathbf{x},t) $ in a coarse mesh; \, (g) the first-order multiscale solution
$\mathbf{E}^{\varepsilon, (1)}(\mathbf{x},t) $;\, (h) the second-order multiscale solution
$\mathbf{E}^{\varepsilon, (2)}(\mathbf{x},t)$. The time step $ \Delta t=0.0025 $. }\label{fig5-4}
\end{center}
\end{figure}

\begin{figure}
\begin{center}
\tiny{(a)}\includegraphics[width=5.5cm,height=4cm] {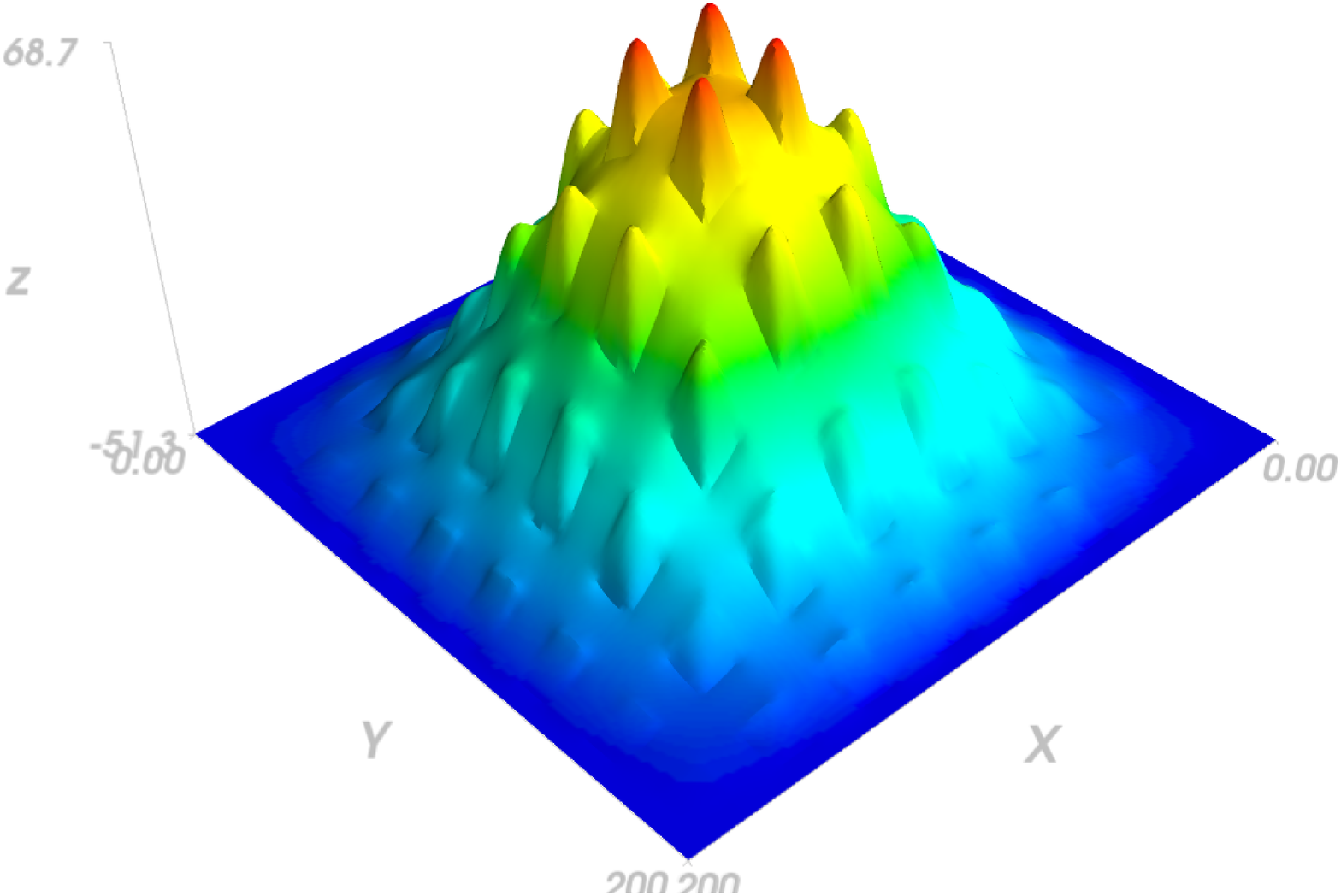}
\tiny{(b)}\includegraphics[width=5.5cm,height=4cm] {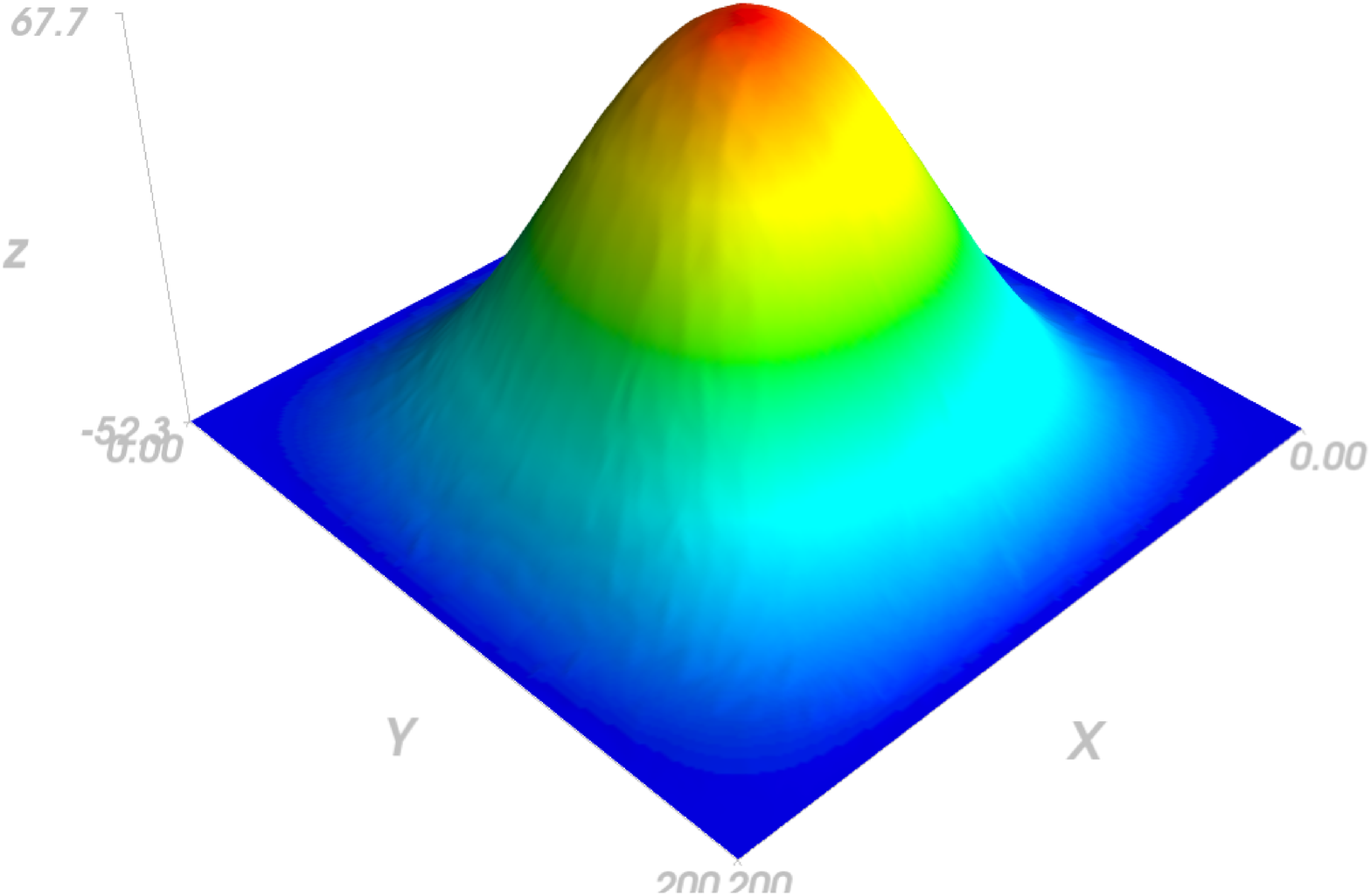}
\tiny{(c)}\includegraphics[width=5.5cm,height=4cm] {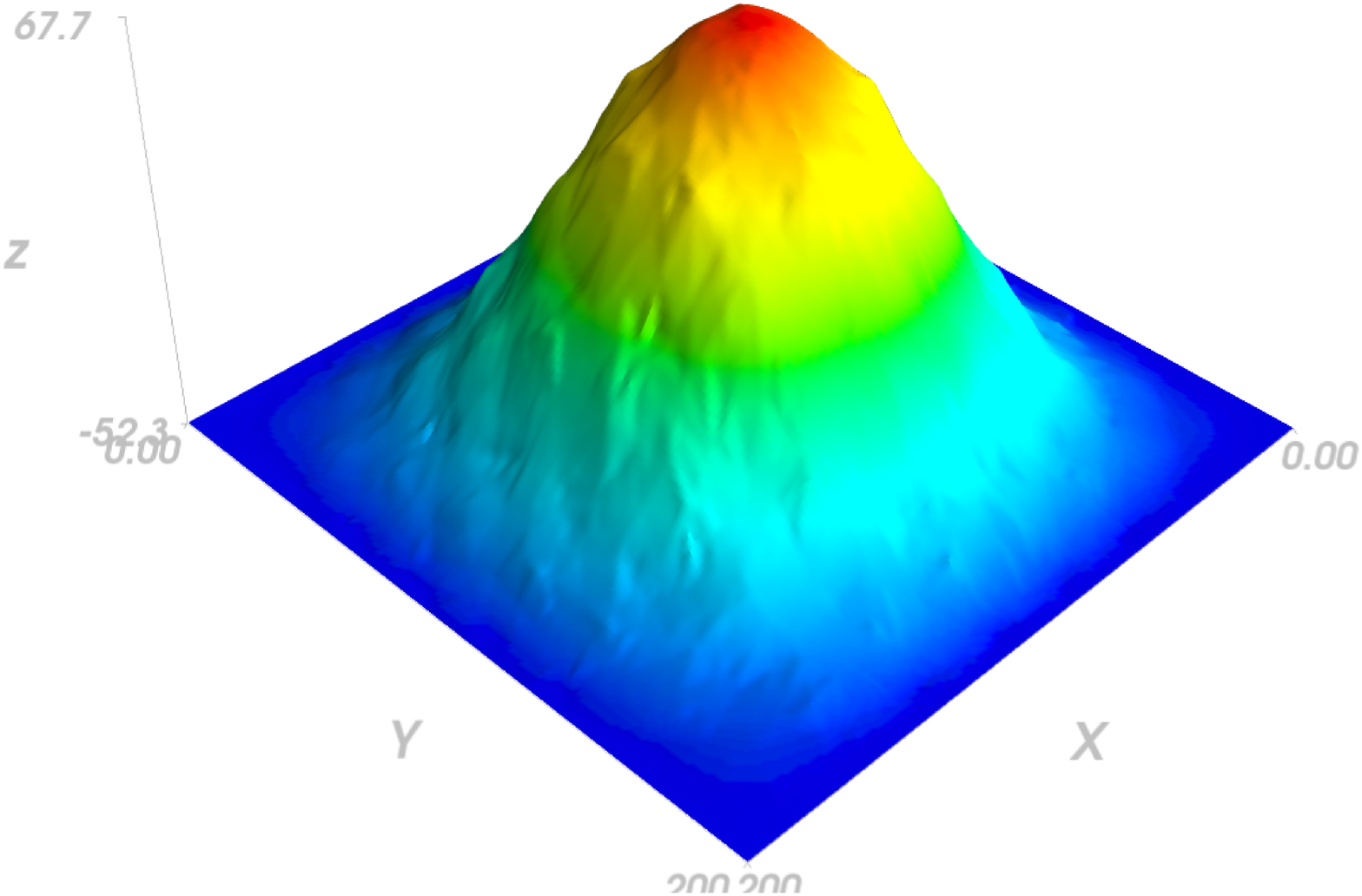}
\tiny{(d)}\includegraphics[width=5.5cm,height=4cm] {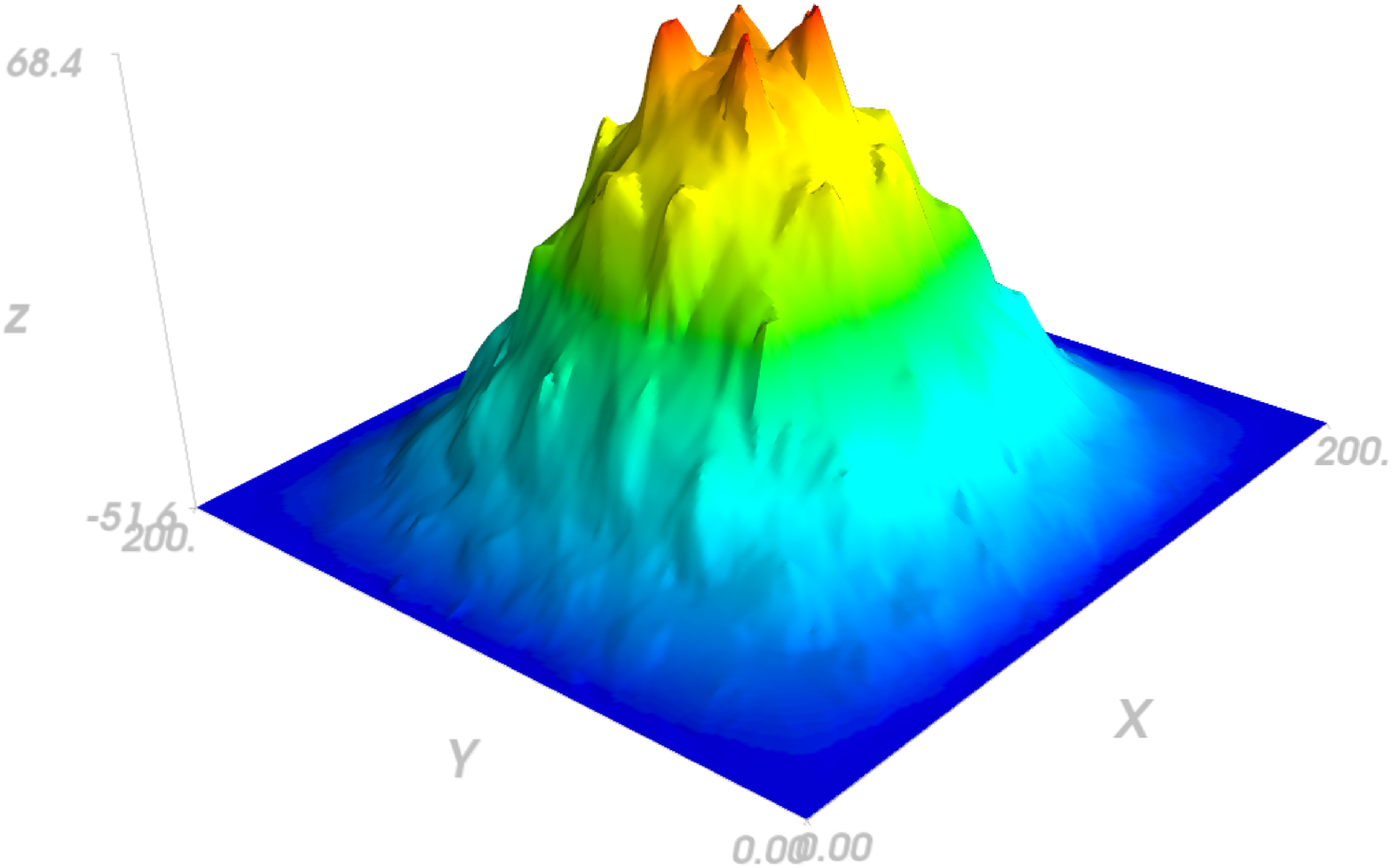}
\tiny{(e)}\includegraphics[width=5.5cm,height=4cm] {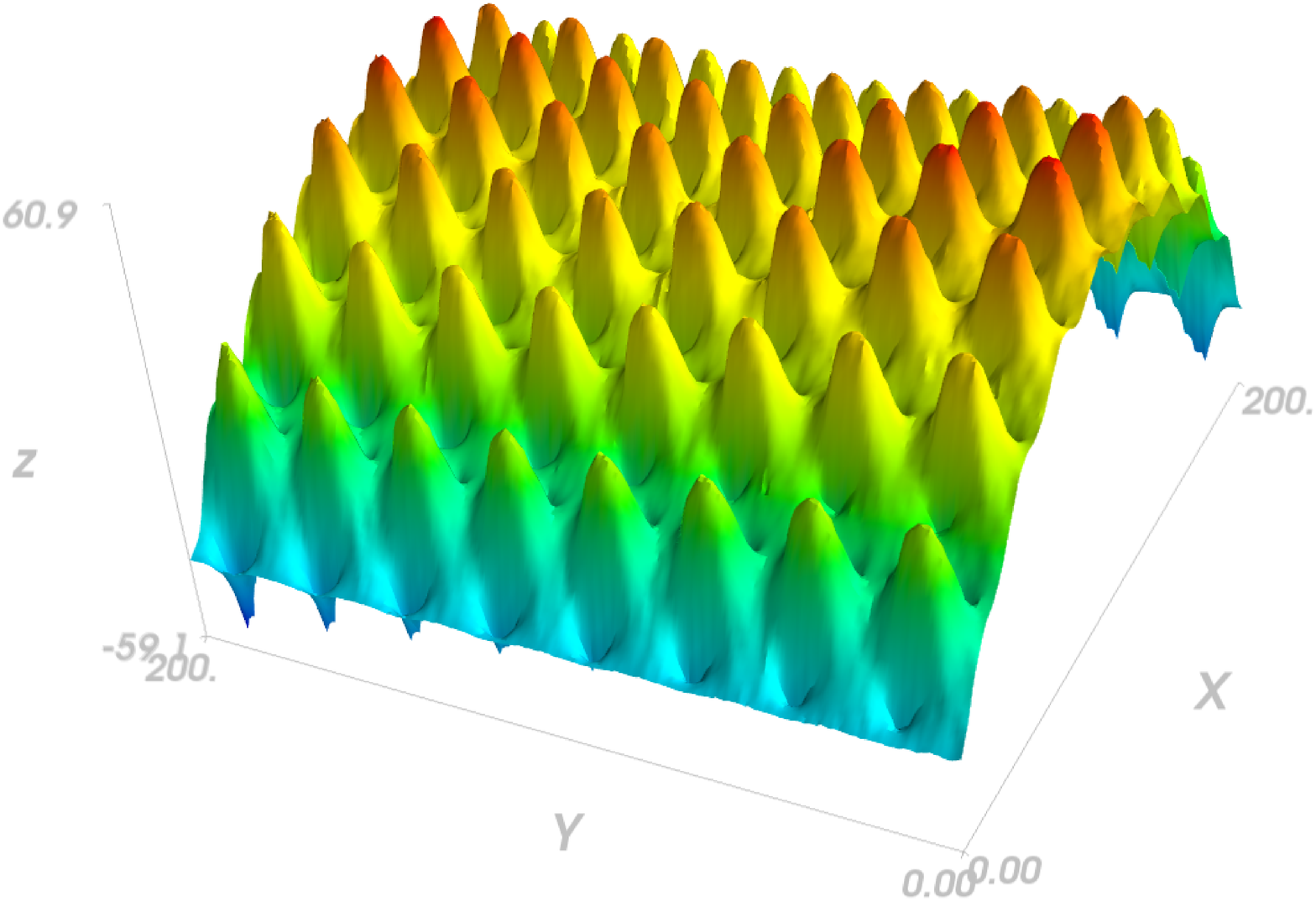}
\tiny{(f)}\includegraphics[width=5.5cm,height=4cm] {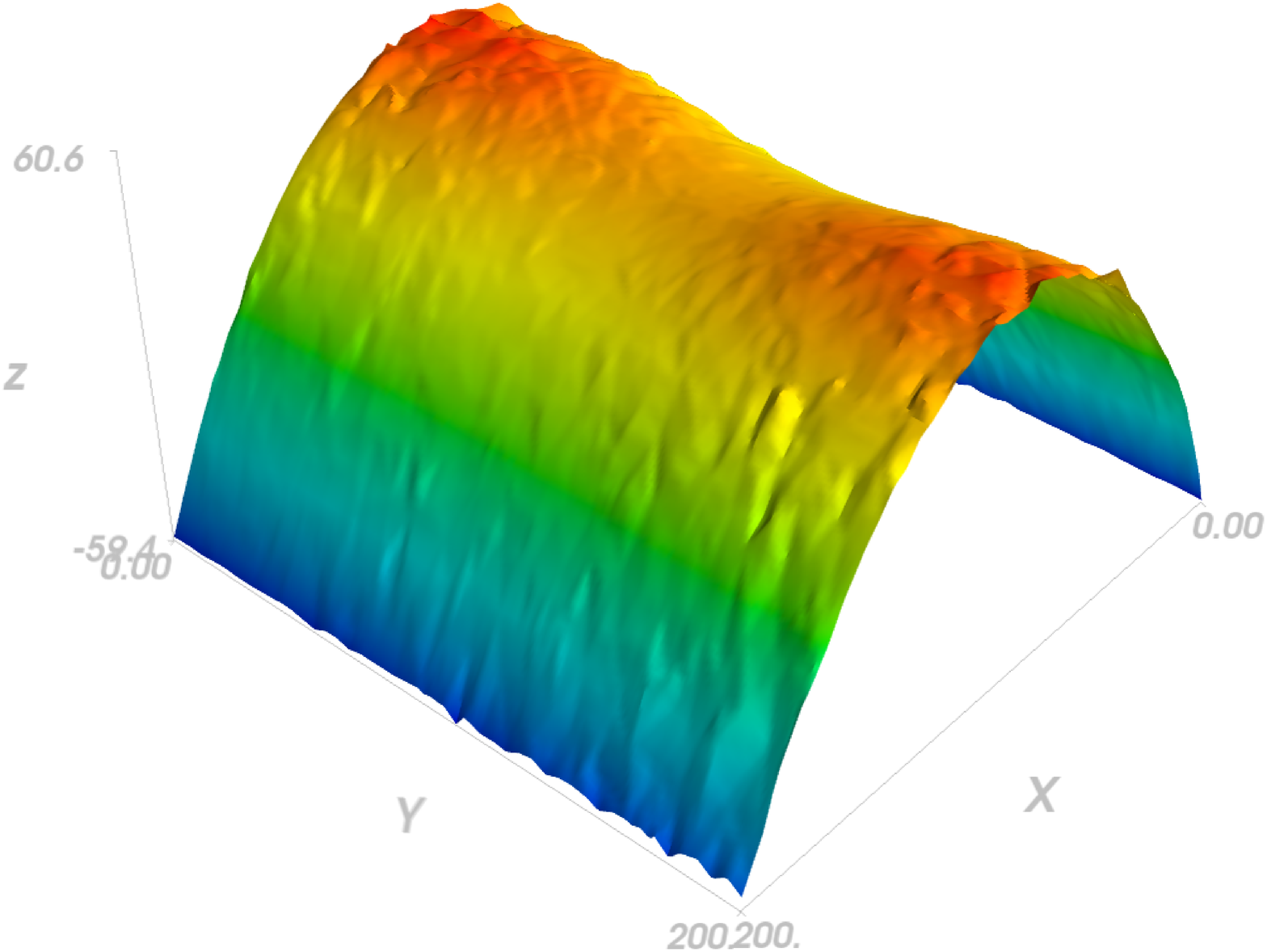}
\tiny{(g)}\includegraphics[width=5.5cm,height=4cm] {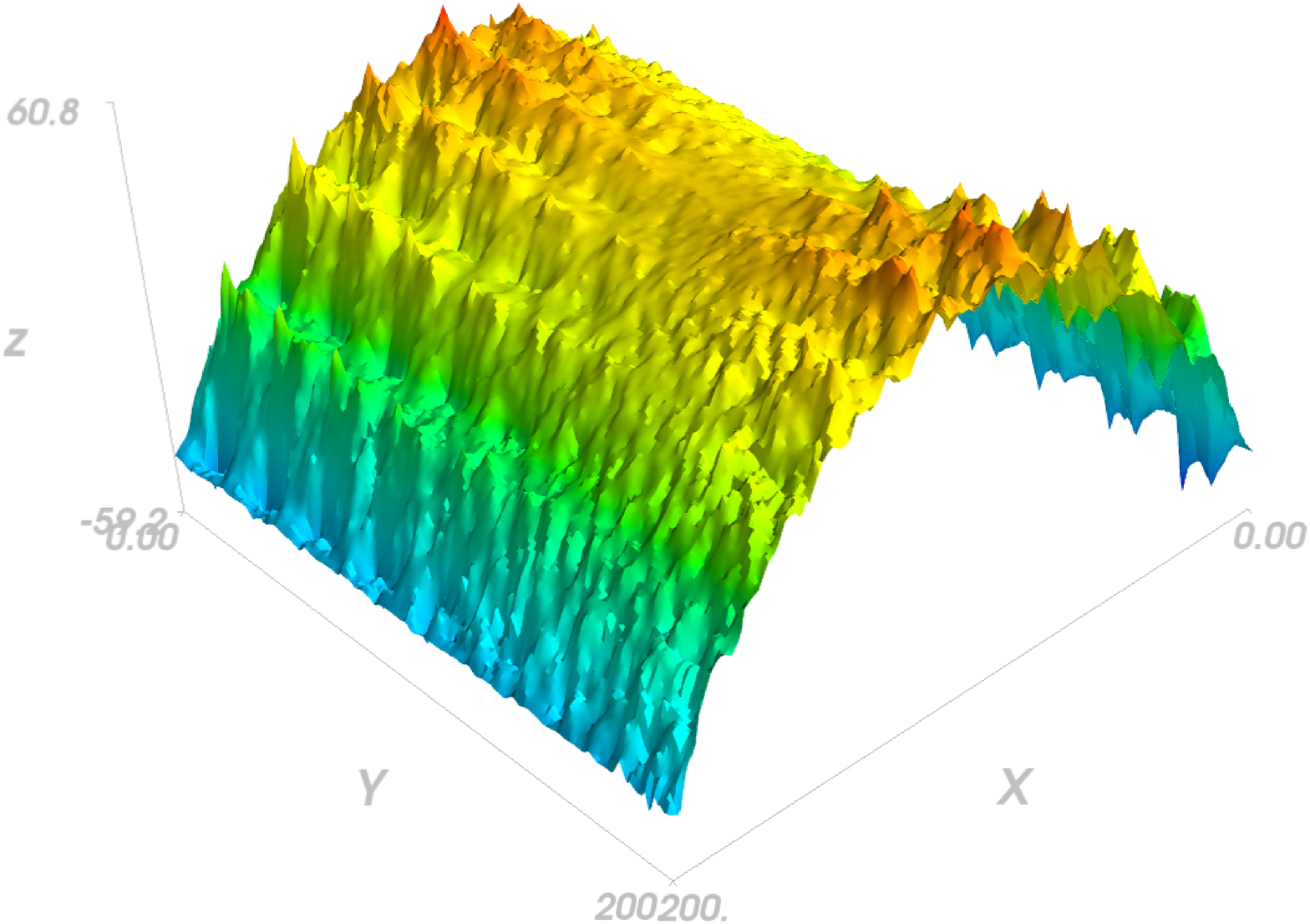}
\tiny{(h)}\includegraphics[width=5.5cm,height=4cm] {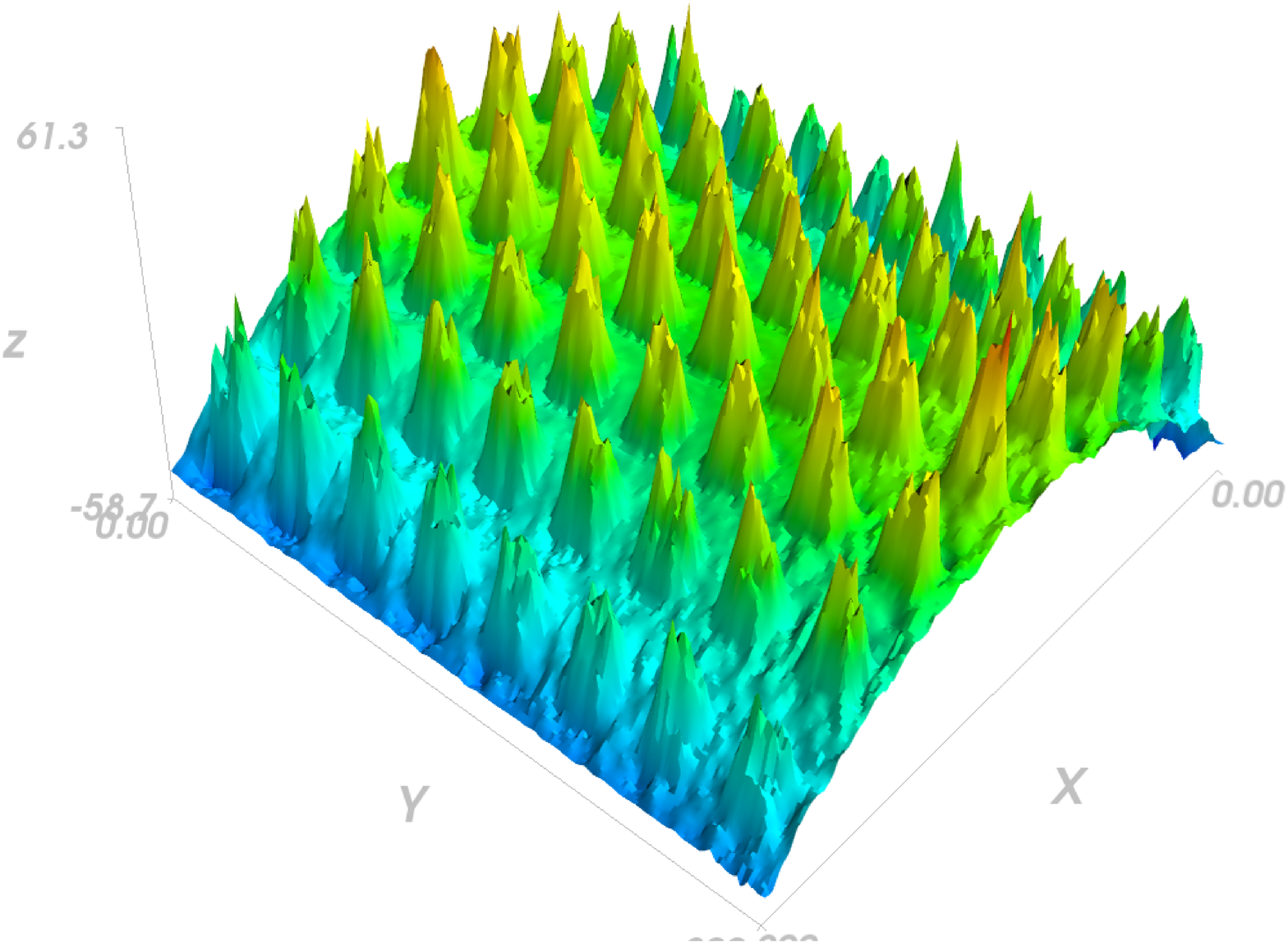}

\caption{Case 5.1.2: the density function and the second component of the electric field on the intersection $ x_3=0.45 $ and at $ T = 0.3 $.
(a) the reference solution $ \rho^\varepsilon(\mathbf{x},t)$ in a fine mesh;\, (b) the homogenized solution $
\rho^0(\mathbf{x},t)$ in a coarse mesh;\, (c) the first-order multiscale solution
$\rho_1^{\varepsilon}(\mathbf{x},t)$; \, (d)the second-order multiscale solution
$\rho_2^{\varepsilon}(\mathbf{x},t)$; \,(e) the reference solution $ \mathbf{E}^\varepsilon(\mathbf{x},t)$ in a fine mesh;\,
(f) the homogenized solution $\mathbf{E}^0(\mathbf{x},t) $ in a coarse mesh; \, (g) the first-order multiscale solution
$\mathbf{E}^{\varepsilon, (1)}(\mathbf{x},t) $;\, (h) the second-order multiscale solution
$\mathbf{E}^{\varepsilon, (2)}(\mathbf{x},t)$. The time step $ \Delta t=0.0025 $. }\label{fig5-5}
\end{center}
\end{figure}

\begin{figure}
\begin{center}
\tiny{(a)}\includegraphics[width=5.5cm,height=4cm] {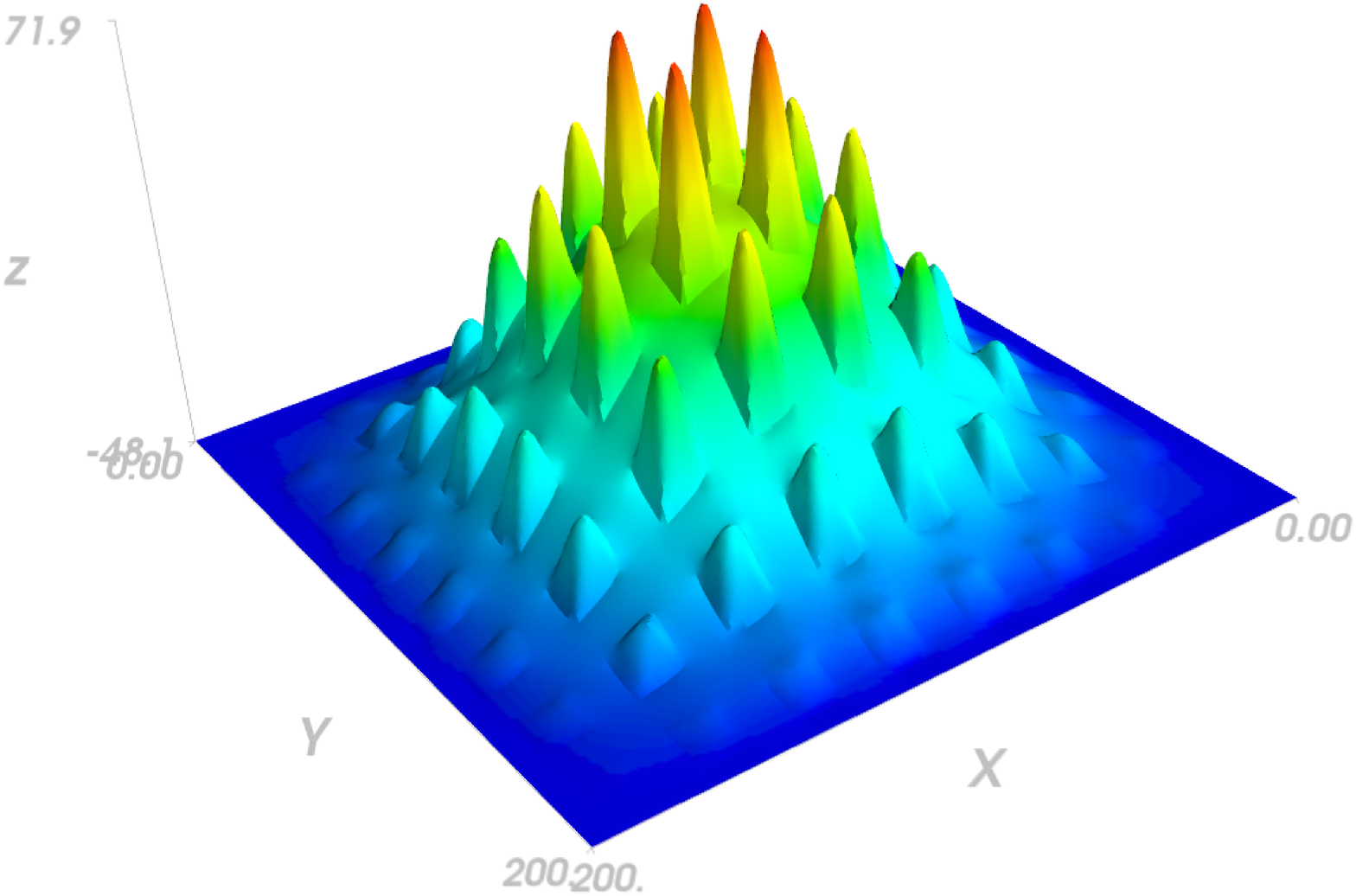}
\tiny{(b)}\includegraphics[width=5.5cm,height=4cm] {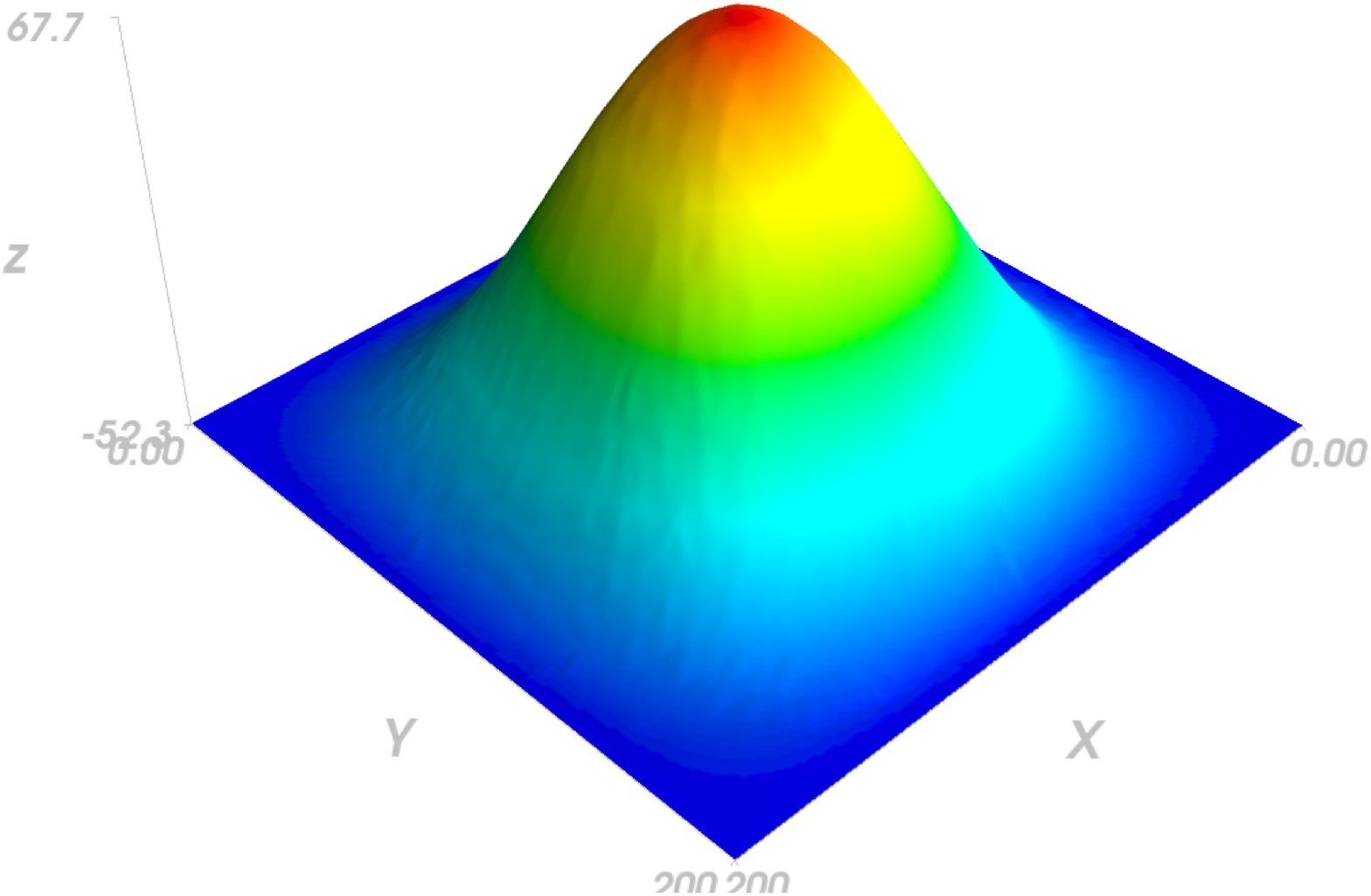}
\tiny{(c)}\includegraphics[width=5.5cm,height=4cm] {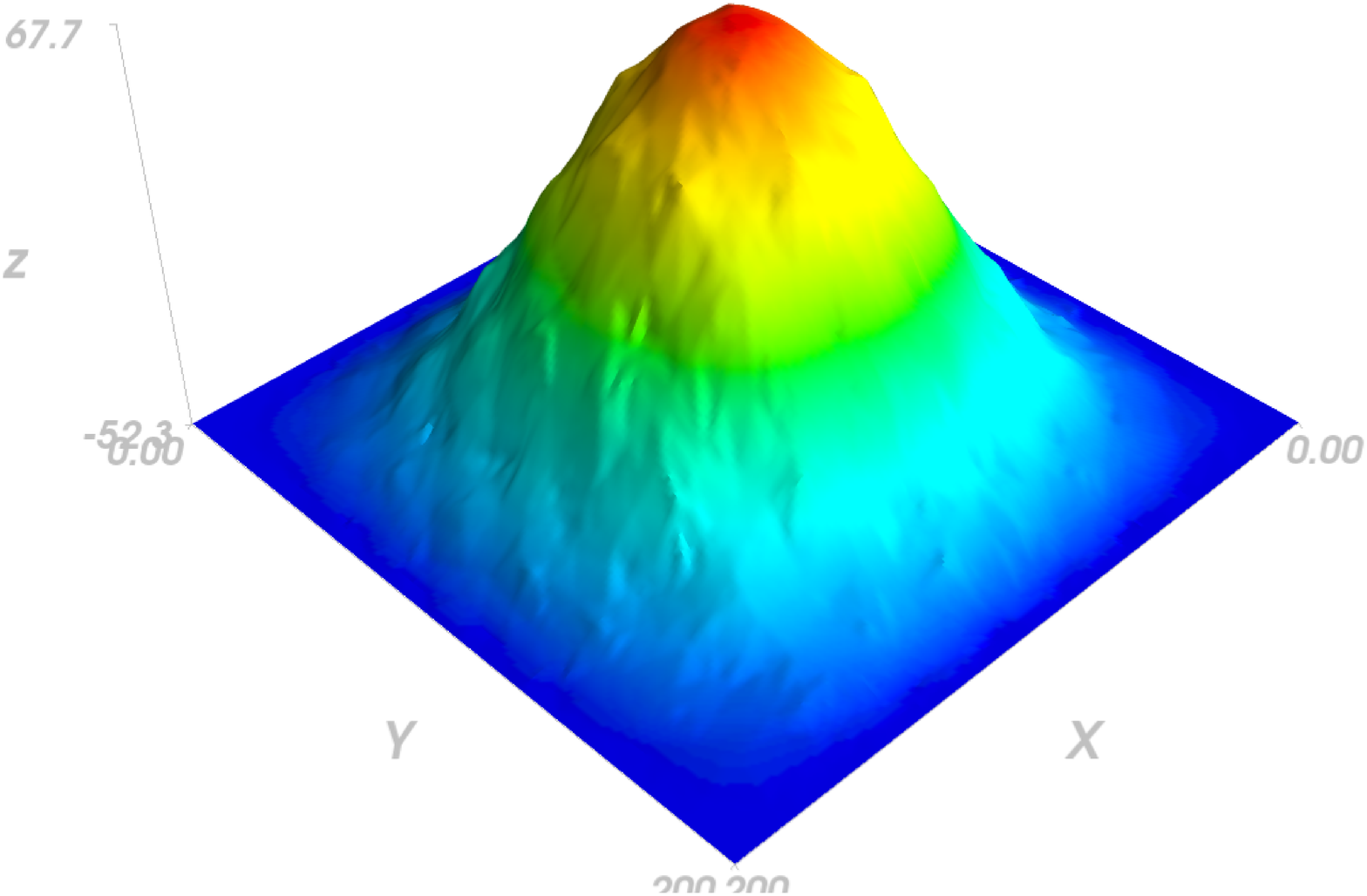}
\tiny{(d)}\includegraphics[width=5.5cm,height=4cm] {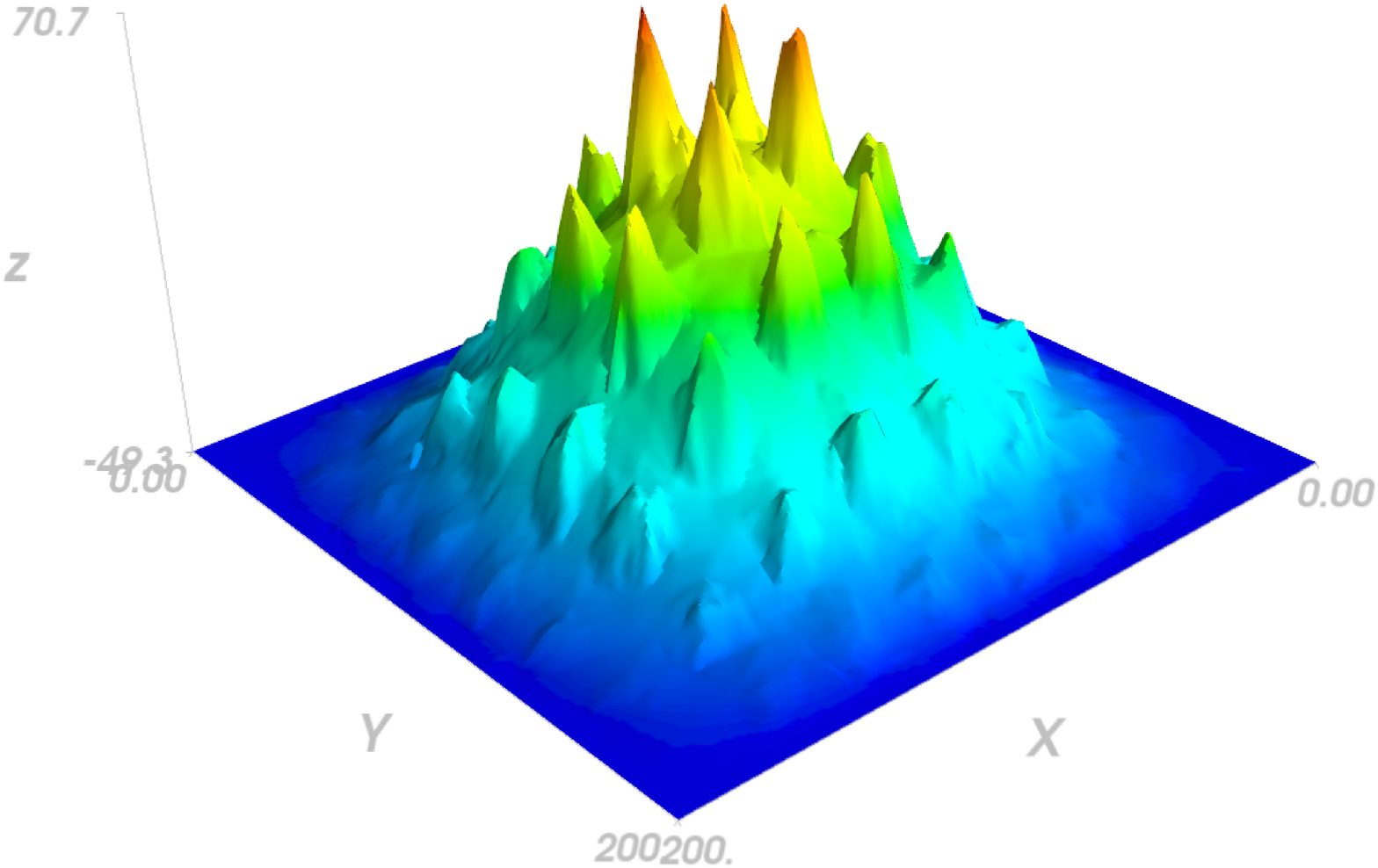}
\tiny{(e)}\includegraphics[width=5.5cm,height=4cm] {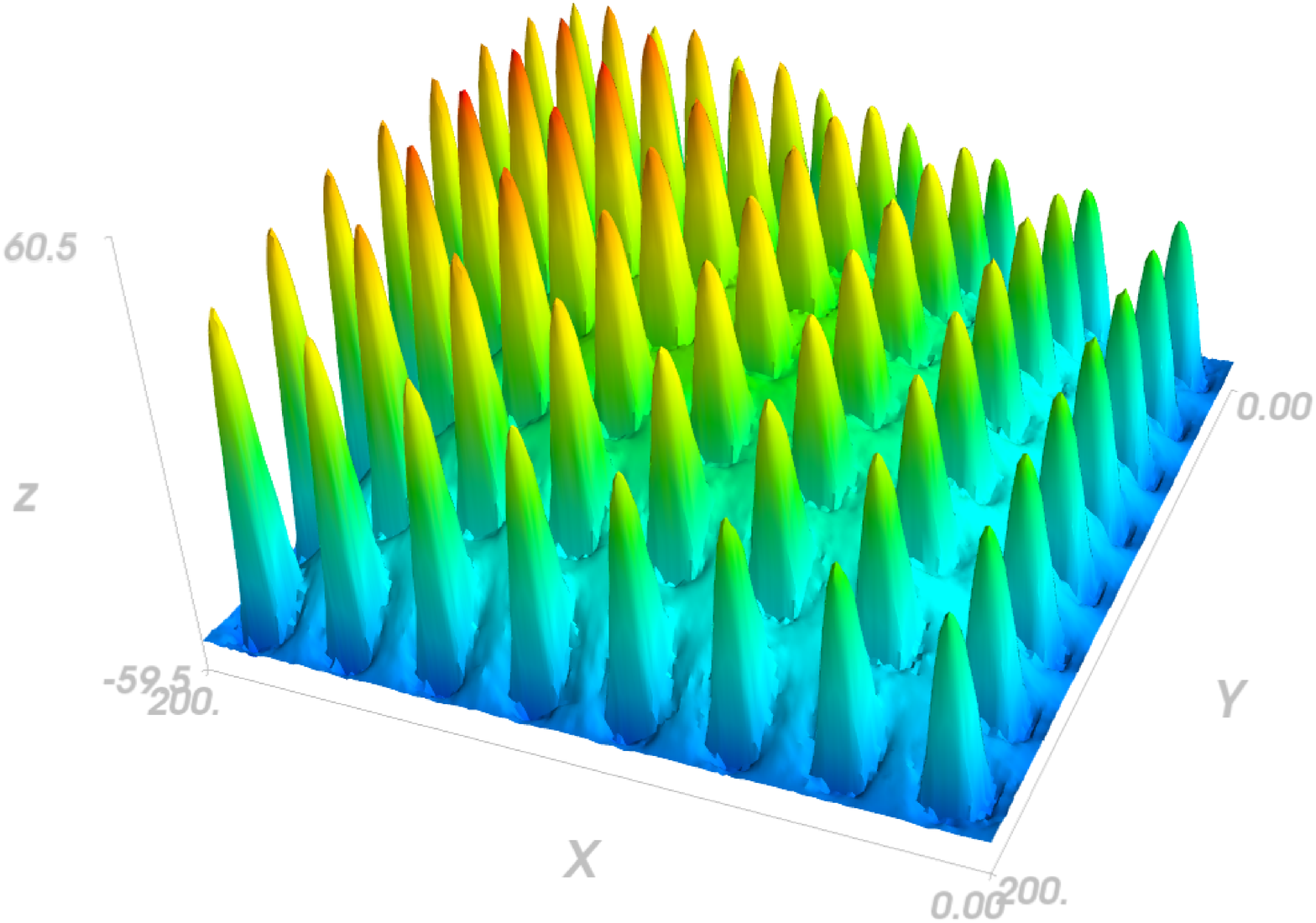}
\tiny{(f)}\includegraphics[width=5.5cm,height=4cm] {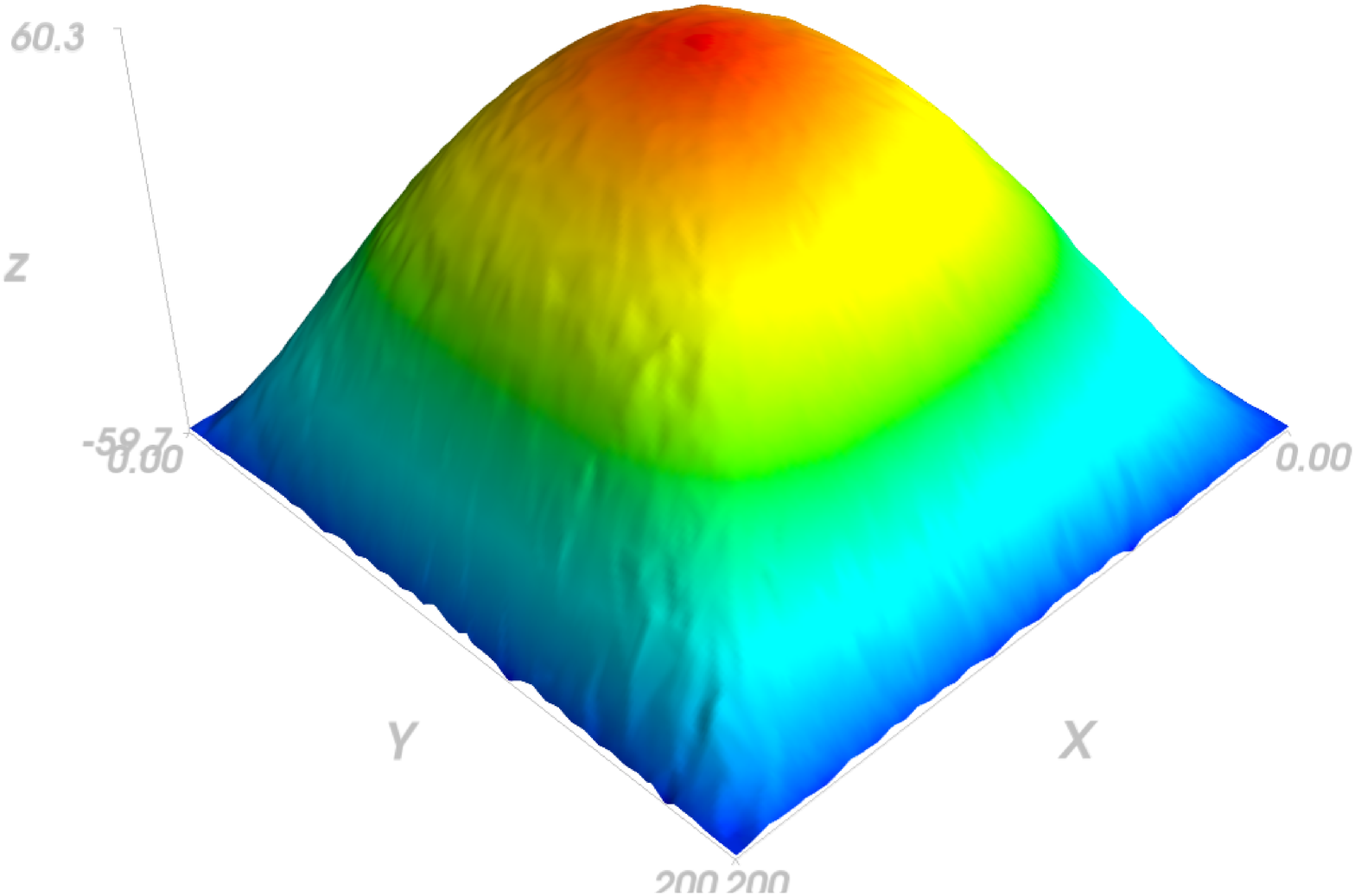}
\tiny{(g)}\includegraphics[width=5.5cm,height=4cm] {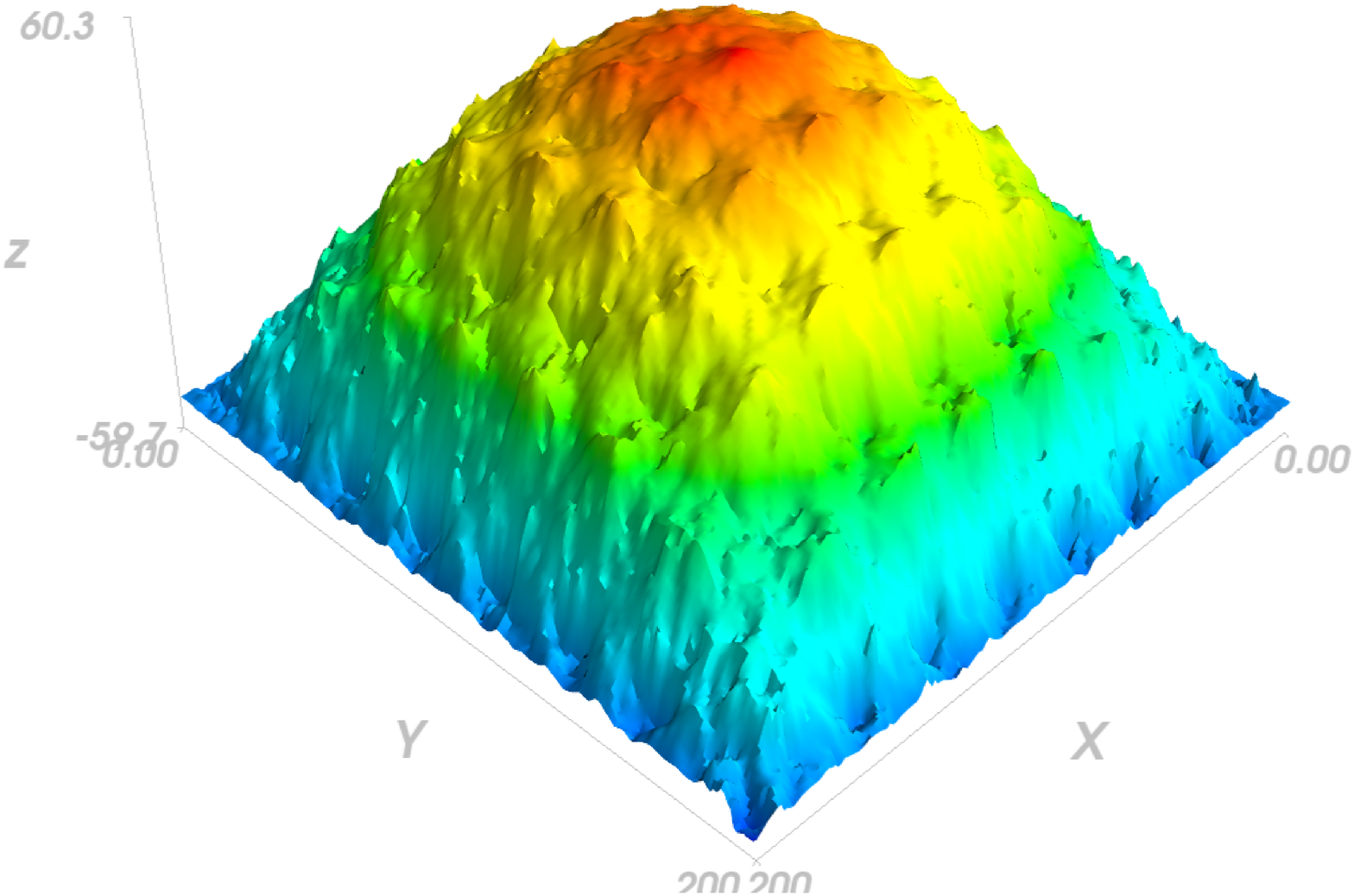}
\tiny{(h)}\includegraphics[width=5.5cm,height=4cm] {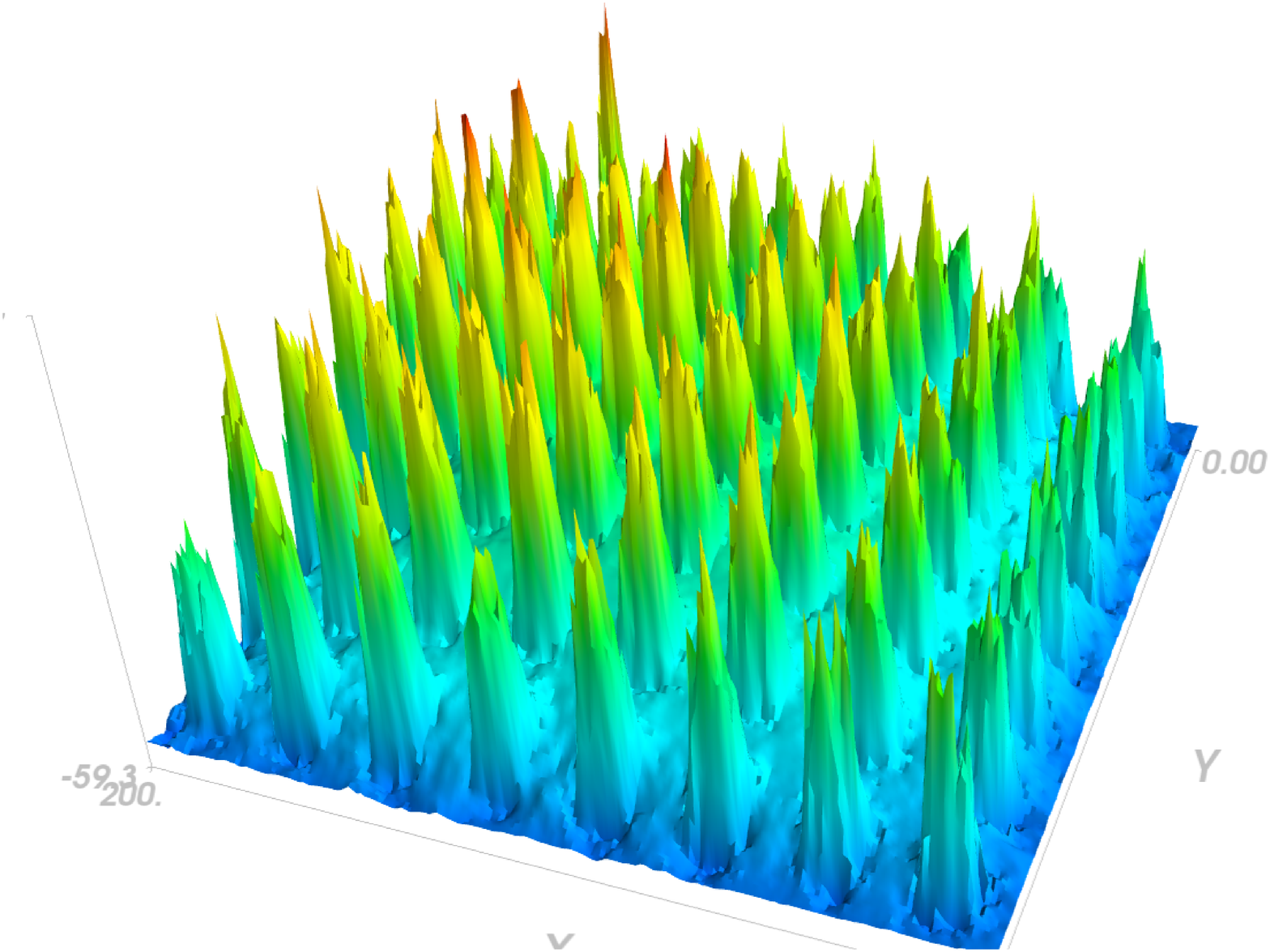}

\caption{Case 5.1.3: the density function and the third component of the electric field on the intersection $ x_3=0.45 $ and at $ T = 0.3 $.
(a) the reference solution $ \rho^\varepsilon(\mathbf{x},t)$ in a fine mesh;\, (b) the homogenized solution $
\rho^0(\mathbf{x},t)$ in a coarse mesh;\, (c) the first-order multiscale solution
$\rho_1^{\varepsilon}(\mathbf{x},t)$; \, (d)the second-order multiscale solution
$\rho_2^{\varepsilon}(\mathbf{x},t)$; \,(e) the reference solution $ \mathbf{E}^\varepsilon(\mathbf{x},t)$ in a fine mesh;\,
(f) the homogenized solution $\mathbf{E}^0(\mathbf{x},t) $ in a coarse mesh; \, (g) the first-order multiscale solution
$\mathbf{E}^{\varepsilon, (1)}(\mathbf{x},t) $;\, (h) the second-order multiscale solution
$\mathbf{E}^{\varepsilon, (2)}(\mathbf{x},t)$. The time step $ \Delta t=0.0025 $. }\label{fig5-6}
\end{center}
\end{figure}

\begin{exam}\label{exam5-2}
We consider the Maxwell-Schr\"{o}dinger system (\ref{eq:5-1})
with rapidly oscillating discontinuous coefficients.
Note that there is the exchange-correlation potential in (\ref{eq:5-1}) and we take
$ V_{xc}[\rho^\varepsilon]=-{(3 \rho^\varepsilon)}^{1/3} $.
Let
\begin{gather*}
\mbox{\bf Case 5.2.}\,\,\,
a_{ij}(\frac{\displaystyle \mathbf{x}}{\displaystyle \varepsilon})
=\left\{
\begin{array}{l}
 0.025\delta_{ij}, \quad  \rm{in\, each \,cube}\\
 \delta_{ij}, \quad  \rm{others}
\end{array}
\right. \quad \mu_{ij}(\frac{\displaystyle \mathbf{x}}{\displaystyle \varepsilon})
=\left\{
\begin{array}{l}
 \delta_{ij}, \quad  \rm{in\, each\, cube}\\
 0.01\delta_{ij}, \quad  \rm{others}.
\end{array}
\right.
\end{gather*}
\end{exam}
The others are the same as those in Example~\ref{exam5-1}.

The computational procedures in Example~\ref{exam5-2} are the same as those in
Example~\ref{exam5-1} except that we have to solve the Schr\"{o}dinger equation by  the self-consistent iterative method. The numerical results are listed in Table~\ref{tab5-5} and Table~\ref{tab5-6}, respectively.

\begin{table}[htb]
\caption{The computational results for the density function in Example~\ref{exam5-2}}\label{tab5-5}
\begin{center}
\begin{tabular}{|c|c|c|c|c|c|c|}
   \hline
    &$\frac{\|e_0\|_{0}}{\|\rho^{\varepsilon}\|_{0}}$
    & $\frac{\|e_1\|_{0}}{\|\rho^{\varepsilon}\|_{0}}$
    & $\frac{\|e_2\|_{0}}{\|\rho^{\varepsilon}\|_{0}}$
    & $\frac{\|e_0\|_{1}}{\|\rho^{\varepsilon}\|_{1}}$
    & $\frac{\|e_1\|_{1}}{\|\rho^{\varepsilon}\|_{1}}$
    & $\frac{\|e_2\|_{1}}{\|\rho^{\varepsilon}\|_{1}}$ \\
   \hline
    Case 5.2 &  0.058699 & 0.055766 & 0.011276 & 0.479818 & 0.451818 & 0.067282  \\
    \hline
  \end{tabular}
  \end{center}
\end{table}
\begin{table}[htb]
\caption{The computational results for the electric field in Example~\ref{exam5-2}}\label{tab5-6}
\begin{center}
\begin{tabular}{|c|c|c|c|c|c|c|}
   \hline
    &$\frac{\|\mathbf{e}_0\|_{(0)}}{\|\mathbf{E}^{\varepsilon}\|_{(0)}}$
    &$\frac{\|\mathbf{e}_1\|_{(0)}}{\|\mathbf{E}^{\varepsilon}\|_{(0)}}$
    &$\frac{\|\mathbf{e}_2\|_{(0)}}{\|\mathbf{E}^{\varepsilon}\|_{(0)}}$
    &$\frac{\|\mathbf{e}_0\|_{(1)}}{\|\mathbf{E}^{\varepsilon}\|_{(1)}}$
    &$\frac{\|\mathbf{e}_1\|_{(1)}}{\|\mathbf{E}^{\varepsilon}\|_{(1)}}$
    &$\frac{\|\mathbf{e}_2\|_{(1)}}{\|\mathbf{E}^{\varepsilon}\|_{(1)}}$\\
    \hline
    Case 5.2 &  0.195816 & 0.079736 &0.040340  & 1.147757 & 1.037087 & 0.814140  \\
    \hline
  \end{tabular}
  \end{center}
\end{table}

The computational results based on the homogenization method, the first-order and the second-order multiscale methods for the density function and the electric field on the line $ x_1=x_2=x_3 $ at $ T=0.4 $ in Case 5.2 are displayed in Fig.~\ref{fig5-7}. The evolution of relative error in $L^2$ norm of the multiscale solutions of the density function and the electric field
 is illustrated in Fig.~\ref{fig5-8}.

\begin{figure}
\begin{center}
\tiny{(a)}\includegraphics[width=6cm,height=6cm] {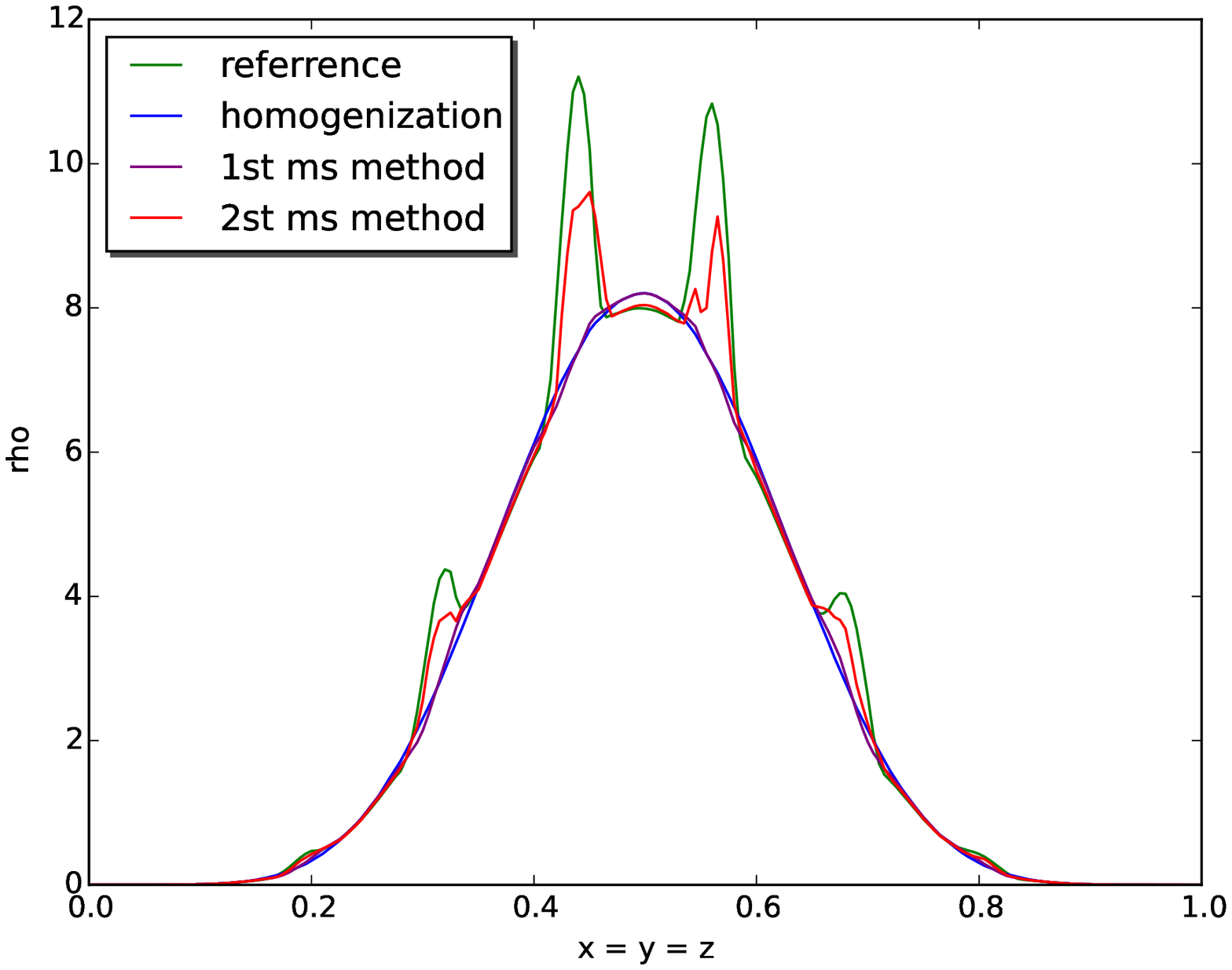}
\tiny{(b)}\includegraphics[width=6cm,height=6cm] {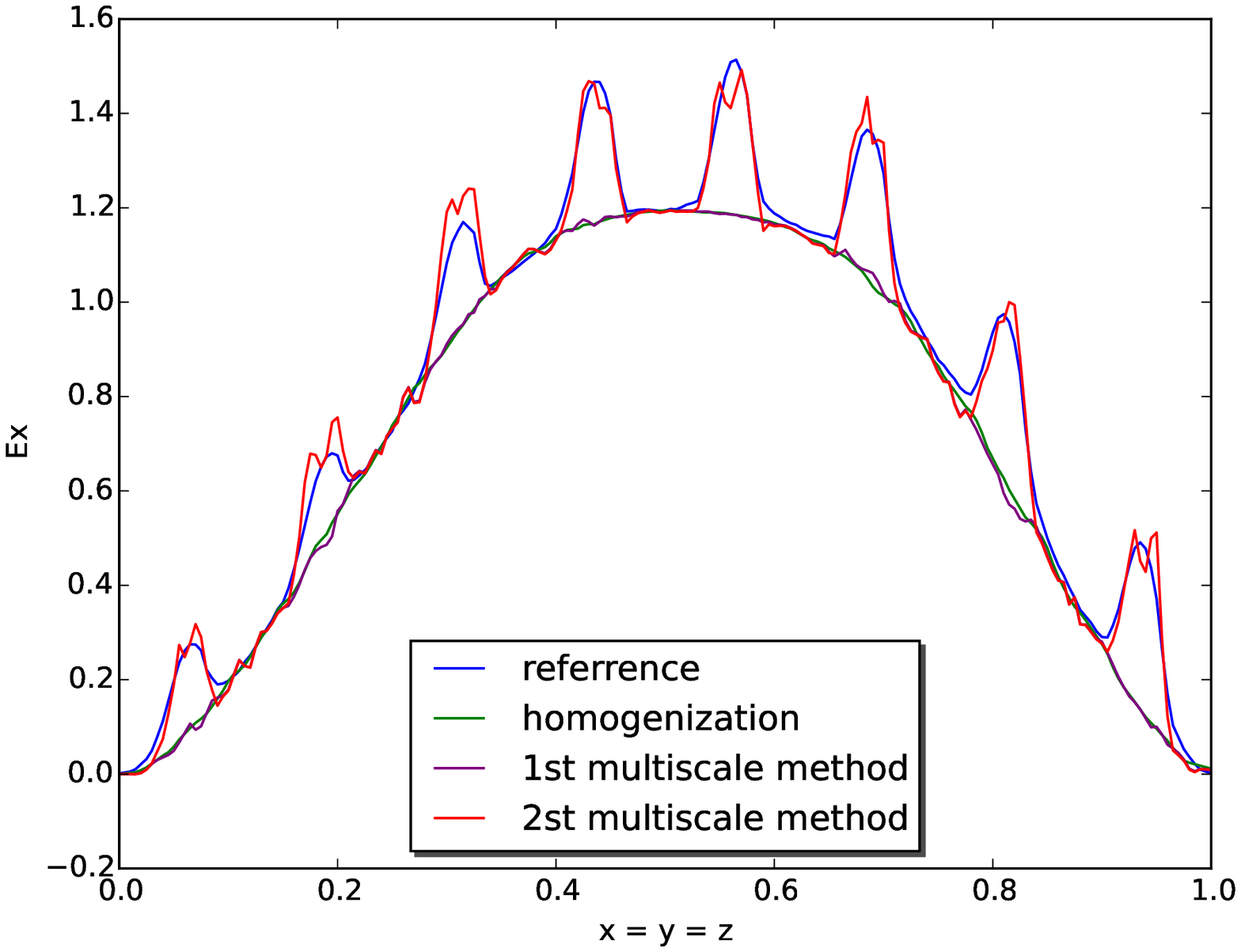}
\caption{ (a) The density function  on the line  $
x_1=x_2=x_3 $  at time $T=0.4$ in Case 5.2 ; \,(b) The third component of the  electric field  on the line  $
x_1=x_2=x_3 $  at time $T=0.4$ in Case 5.2 ; \,}\label{fig5-7}
\end{center}
\end{figure}

\begin{figure}
\begin{center}
\tiny{(a)}\includegraphics[width=6cm,height=6cm] {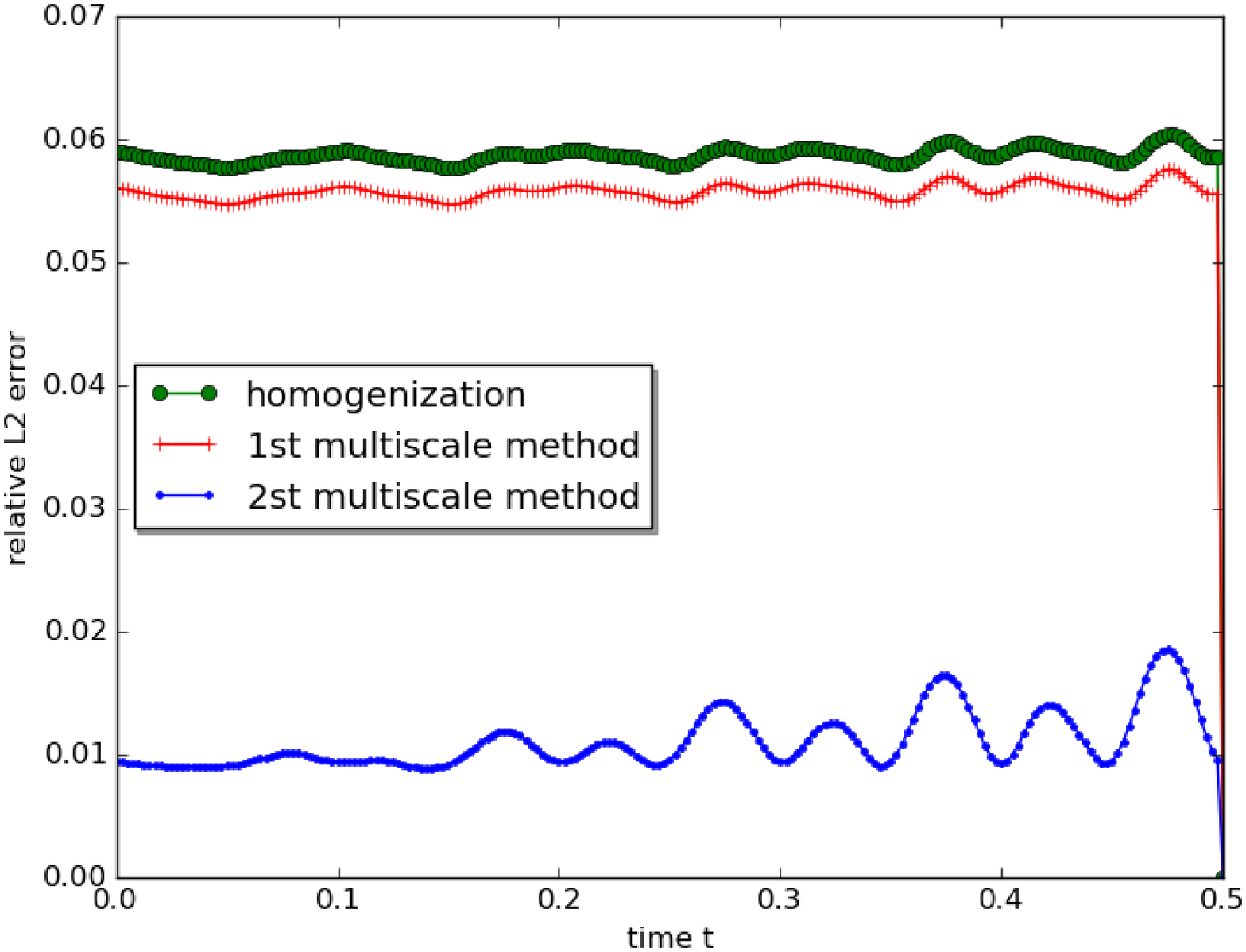}
\tiny{(b)}\includegraphics[width=6cm,height=6cm] {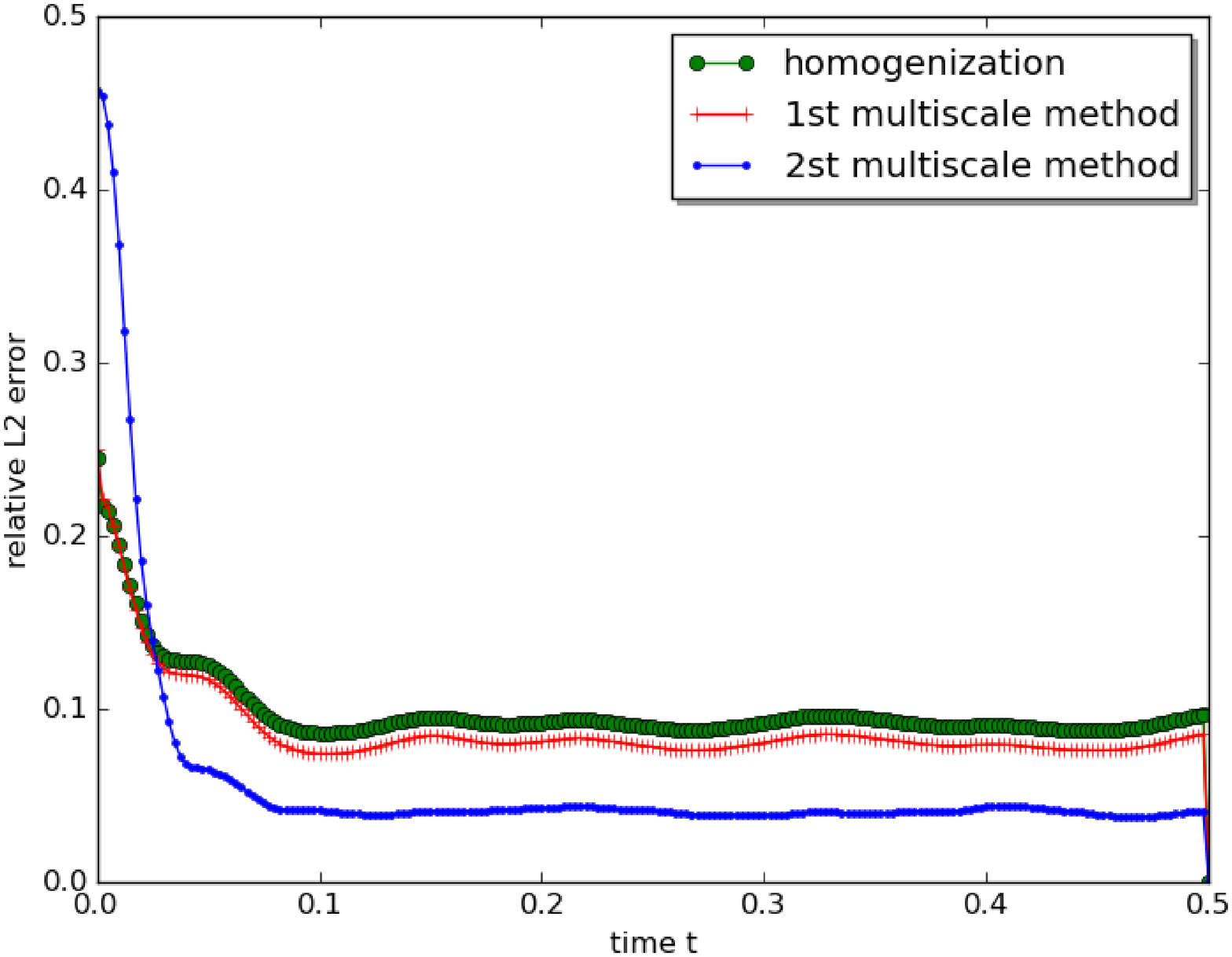}
\caption{ (a) The evolution of relative error in $L^2$ norm of
$\rho_1^{\varepsilon} $, $
\rho_2^{\varepsilon} $ and $
\rho^0 $ in Case 5.2. \,(b)The evolution of relative error in $L^2$ norm of
$\mathbf{E}^{\varepsilon, (1)} $, $
\mathbf{E}^{\varepsilon, (2)} $ and $
\mathbf{E}^0 $ in Case 5.2.}\label{fig5-8}
\end{center}
\end{figure}

\begin{rem}\label{rem5-1}
The comparison of computational costs listed in Tables~\ref{tab5-1} and \ref{tab5-2} clearly shows that the
multiscale asymptotic method provides a tremendous saving in computing resource,
in particular, for a sufficiently small periodic parameter $ \varepsilon>0 $.
\end{rem}

\begin{rem}\label{rem5-2}
From the results reported in Example~\ref{exam5-1} (also see Tables~\ref{tab5-3} and \ref{tab5-4}),
if there is a sharp difference between materials for the coefficient matrices
 $(a_{ij}(\frac{\displaystyle \mathbf{x}}{\displaystyle
\varepsilon})) $ and  $(\mu_{ij}(\frac{\displaystyle \mathbf{x}}{\displaystyle
\varepsilon})) $, the homogenization method and the first-order
multiscale method fail to provide satisfactory results. The
second-order multiscale approach, however, is capable of producing
accurate numerical solutions for the density function $
\rho^\varepsilon $ and the electric field $
\mathbf{E}^\varepsilon(\mathbf{x},t) $. The numerical results reported in Example~\ref{exam5-2}
clearly shows that the second-order corrector terms are crucial in
the proposed multiscale algorithm too.
\end{rem}

\begin{rem}\label{rem5-3}
The Maxwell's equations and the Schr\"{o}dinger equation can also be coupled through
the vector potential $\mathbf{A}$ and the scalar potential $\varphi$ instead of
$ \hat{V}(\mathbf{x},t)=-\mathbf{E}(\mathbf{x},t)\cdot \widehat{\bm{\zeta}}$, which is called ``length gauge".
Ohnuki et al. \cite{Oh} verified theoretically and numerically that
in the long-wavelength approximation, the length gauge is equivalent to the $ \mathbf{A}-\varphi$ method.
\end{rem}

\subsection*{Acknowledgments}

The authors wish to thank Prof.\ Linbo Zhang, Prof.\linebreak Zhiming Chen and Prof. Tao Cui for their help.
The use of the PHG computing platform for numerical simulations is
particularly acknowledged.


\end{document}